\documentclass[letterpaper]{article}
\usepackage[textheight=9in,
    textwidth=5.5in,
    top=1in,
    left=1.5in,
    right=1.5in,
    headheight=12pt,
    headsep=25pt,
    footskip=30pt]{geometry}

\usepackage{enumitem}

\usepackage{xr}

\usepackage[utf8]{inputenc} 
\usepackage[T1]{fontenc}    
\usepackage{hyperref,pifont}       
\usepackage{url}            
\usepackage{booktabs}       
\usepackage{amsfonts}       
\usepackage{nicefrac}       
\usepackage{microtype}      
\usepackage{xcolor}         
\usepackage{todonotes}
\usepackage{times}
\usepackage{amsmath,color}
\usepackage{subcaption}
\usepackage{amsthm}
\usepackage{array}
\usepackage{graphicx}
\usepackage{multirow}
\usepackage{multicol}
\usepackage{caption}
\usepackage{makecell}
\usepackage{bbm}
\usepackage{breqn}
\usepackage{bm}
\usepackage{hyperref}
\usepackage[utf8]{inputenc}
\usepackage{algorithm}
\usepackage{algorithmic}
\usepackage{authblk}
\usepackage{wrapfig}
\usepackage{tabularray}

\newtheorem{theorem}{Theorem}

\newtheorem{remark}{Remark}
\newtheorem{corollary}{Corollary}
\newtheorem{lemma}{Lemma}
\newtheorem{proposition}{Proposition}
\newtheorem{assumption}{Assumption}

\newlength\myindent
\newcommand{\vp}{V_{\perp}}

\newcommand{\ba}[1]{\begin{align}#1\end{align}}
\newcommand{\bas}[1]{\begin{align*}#1\end{align*}}
\newcommand{\csqone}{10} 
\newcommand{\etakub}{\frac{1}{10}}
\newcommand{\mult}[2]{\the\numexpr #1 * #2 \relax}
\newcommand{\add}[2]{\the\numexpr #1 + #2 \relax}
\newcommand{\mydiv}[2]{\the\numexpr #1/#2 \relax}
\newcommand{\csq}{\mult{\csqone}{2}}
\newcommand{\bb}[1]{\left(#1\right)}
\newcommand{\bbb}[1]{\left[#1\right]}

\newcommand{\bi}{\begin{itemize}}
\newcommand{\ib}{\end{itemize}}

\newcommand{\bk}{\color{black}}

\newcommand{\M}{\mathcal{M}}

\DeclareMathOperator{\diag}{diag}
\newcommand{\B}{B}
\newcommand{\newB}[2]{B_{#2,#1}}
\newcommand{\E}{\mathbb{E}}
\newcommand{\init}{u}
\DeclareMathOperator{\Tr}{Tr}

\DeclareMathOperator{\gap}{gap}
\DeclareMathOperator{\vect}{vec}

\newcolumntype{M}[1]{>{\centering\arraybackslash}m{#1}}
\usepackage{authblk,babel}
\makeatletter
\newcommand*{\addFileDependency}[1]{
\typeout{(#1)}
\@addtofilelist{#1}
\IfFileExists{#1}{}{\typeout{No file #1.}}
}\makeatother

\title{Streaming PCA for Markovian Data}
\author[1]{Syamantak Kumar\thanks{syamantak@utexas.edu}}
\affil[1]{Department of Computer Science, University of Texas at Austin}
\author[2]{Purnamrita Sarkar \thanks{purna.sarkar@utexas.edu}}
\affil[2]{Department of Statistics and Data Sciences, University of Texas at Austin}
\date{}

\begin{document}
\maketitle
\begin{abstract}
  Since its inception in 1982, Oja's algorithm has become an established method for streaming principle component analysis (PCA). We study the problem of streaming PCA, where the data-points are sampled from an irreducible, aperiodic, and reversible Markov chain. Our goal is to estimate the top eigenvector of the unknown covariance matrix of the stationary distribution. This setting has implications in scenarios where data can solely be sampled from a Markov Chain Monte Carlo (MCMC) type algorithm, and the objective is to perform inference on parameters of the stationary distribution.
  Most convergence guarantees for Oja's algorithm in the literature assume that the data-points are sampled IID. For data streams with Markovian dependence, one typically downsamples the data to get a "nearly" independent data stream. In this paper, we obtain the first sharp rate for Oja's algorithm on the entire data, where we remove the logarithmic dependence on the sample size, $n$, resulting from throwing data away in downsampling strategies. 
\end{abstract}
\vspace{-1em}
\section{Introduction}
\vspace{-.5em}
Streaming Principal Component Analysis (PCA) is an important and well studied problem where the  principal eigenvector of the sample covariance matrix of a dataset is computed one data-point at a time.
One of the most popular algorithms for streaming PCA was introduced by Erikki Oja in 1982 \cite{oja-simplified-neuron-model-1982, oja1985stochastic}. Most existing analyses of Oja's algorithm are done when the data is sampled IID. 

However, in many practical applications, the data-points are dependent and are sampled from an MCMC process converging to a target stationary distribution. This naturally arises in the context of token algorithms for Federated PCA settings~\cite{even2023stochastic, DBLP:journals/corr/abs-1907-08059, 9679131} with multiple machines communicating via a fixed and connected graph topology. Each machine contains an arbitrary fraction of data-points and the goal is to design a streaming algorithm that respects this topology and returns the principal component of the whole dataset. This is typically achieved using a Metropolis-Hastings scheme that uses local information to design the transition matrix of a Markov chain with any desired stationary distribution. The stationary distribution, $\pi$, of the random walk is chosen so that the distribution of the samples under $\pi$ matches the uniform distribution over data-points. Governed by this Markov chain, a random walker then travels the network of machines and samples one data-point at a time from the current machine, and computes the update. 
However, even under the stationary distribution, the data-points are dependent, which deviates from the IID setup. 
\textit{Our goal is to obtain a sharp analysis of the $\sin^2$ error of the estimated vector w.r.t true top eigenvector of the unknown covariance matrix in the Markovian setting.}

\paragraph{Estimating the first principal component with streaming PCA}: 
Let $X_t$ be a mean zero $d$ dimensional vector with covariance matrix $\Sigma$, and let $\eta_t$ be a decaying learning rate.
The update rule of Oja's algorithm is given as -
\ba{
    w_{t} \leftarrow (I + \eta_{t}X_tX_t^{T})w_{t-1}, \;\; w_{t} \leftarrow \frac{w_{t}}{\|w_{t}\|_{2}} \label{alg:oja}
}
where $w_{t}$ is the estimate of $v_1$ and $\eta_{t}$ is the step-size at timestep $t$. We aim to analyse the $\sin^{2}$ error of Oja's iterate at timestep $t$, defined as $1 - \left\langle w_{t},v_{1} \right\rangle^{2}$, where $v_1$ is the top eigenvector of $\Sigma$.


\paragraph{Streaming PCA in the IID setting:}
For an IID data stream with $\E\bbb{X_i}=0$ and $\E\bbb{X_iX_i^T}=\Sigma$, there has been a lot of work on determining the non-asymptotic convergence rates for Oja's algorithm and its various adaptations \cite{jain2016streaming, allenzhu2017efficient, chen2018dimensionality, yang2018history, henriksen2019adaoja, DBLP:journals/corr/abs-2102-03646, mouzakis2022spectral, lunde2022bootstrapping, monnez2022stochastic}. Amongst these, \cite{jain2016streaming}, \cite{allenzhu2017efficient} and \cite{DBLP:journals/corr/abs-2102-03646} match the optimal offline sample complexity bound, suggested by the independent and identically distributed (IID) version of Theorem \ref{theorem:offline_sample_complexity} (See Theorem 1.1 in \cite{jain2016streaming}). 

We consider Oja's algorithm in the setting where the data is generated from a reversible, irreducible, and aperiodic Markov chain with stationary distribution $\pi$. We denote by $\E_{\pi}[.]$ the expectation under the stationary distribution. In this setting our goal is to estimate the principal eigenvector of $\E_{\pi}\bbb{X_iX_i^T}$. As in the IID setting, $\E_{\pi}[X_i]=0$. The challenge is that the data, even when it reaches stationarity, is dependent. Here the degree of dependence is captured by the second eigenvalue in the magnitude of the transition matrix $P$ (denoted as $|\lambda_2(P)|$) of the Markov chain. This is closely related to the mixing time of a Markov chain ~\cite{levin2017Markov}, denoted as $\tau_{\text{mix}}$, which is the time after which the conditional distribution of a state is close in total variational distance to its stationary distribution, $\pi$ (See Section~\ref{section:preliminaries}). 

\textbf{Our contribution:}
Using a series of approximations, we obtain an optimal error rate for the $\sin^2$ error, which is worse by a factor of $1/(1-|\lambda_2(P)|)$ from the corresponding error rate of the IID case. Previous work~\cite{chen2018dimensionality} has established rates worse by a poly-logarithmic factor by using downsampling, i.e. applying the update on every $k^{th}$ datapoint. In Figure~\ref{fig:online_pca_comparison1}, we compare Oja's algorithm with its downsampled and offline variants (see Section~\ref{section:experiments} for more details on setup). We see that Oja's algorithm performs significantly better than the downsampled variant, and similarly to the offline variant where for the $i^{th}$ data point we compute the eigenvector of the sample covariance matrix of all data-points up-to $i$. Our work provides a concrete and novel result that explains these observations.
In Table~\ref{tab:rate_comparison}, we compare our bounds with related analyses of Oja's algorithm. The last row shows that we are the first to obtain an error whose main term is \textit{free of logarithmic dependence} on $n$ or $d$ for \textit{streaming} PCA in the \textit{Markovian} case.

We \textit{break the logarithmic barrier} in previous work by considering a series of approximations of finer granularity which uses reversibility of the Markov chain and standard mixing conditions of irreducible and aperiodic Markov chains. Our rates are comparable to the recent work of ~\cite{neeman2023concentration} (Proposition~\ref{theorem:offline_sample_complexity}) that establishes an offline error analysis for estimating the principal component of the empirical covariance matrix of Markovian data by using a Matrix Bernstein inequality. Our results also imply a linearly convergent decentralized algorithm for streaming PCA in a distributed setting. As a simple byproduct of our theoretical result, we also obtain a rate for Oja's algorithm applied on downsampled data, which is worse by a factor of $\log n$, as shown in Figure \ref{fig:online_pca_comparison1}. To our knowledge, this is the first work that analyzes the  Markovian streaming PCA problem without any downsampling that matches the error of the offline algorithm. 

The crux of our analysis uses the mixing properties of the Markov chain. Strong mixing intuitively says that the conditional distribution of a state $s$ in timestep $k$ given the starting state is exponentially close to the stationary distribution of $s$, the closeness being measured using the total variation distance. All previous work on Markovian data exploits this property by conditioning on states many time steps before. However, it is crucial to a) adaptively find how far to look back and b) bound the error of the sequence of matrices we ignore between the current state and the state we are conditioning on. Observe that these two components are related. Looking back too far makes the dependence very small but increases the error resulting from approximating a larger matrix product of intermediate matrices.
 We present a fine analysis that balances these two parts and then uses spectral theory to bound the second part within a factor of a variance parameter that characterizes the variability of the matrices and shows up in the analysis of~\cite{jain2016streaming,neeman2023concentration}.


\begin{table}[htb]
   \hspace{-3.5em} \centering
    \begin{tabular}{
|c|c|c|c|c|c|}
         \hline
         \multirow{2}{*}{Paper} & \multirow{2}{*}{Markov?} & \multirow{2}{*}{Online?} & Log-free & \multirow{2}{*}{$\sin^2$ error rate} & \multirow{2}{*}{Sample-Complexity} \\
         & & & main-term & & \\
         \hline
         \multirow{2}{*}{Jain et al. } & \multirow{2}{*}{N} & Y & Y&$O\left(\frac{\mathcal{V}}{\gap^2}\frac{1}{n}\right)$ & $O\left(\frac{\mathcal{V}}{\gap^2}\frac{1}{\epsilon}\right)$ \\
         \cite{jain2016streaming}& & N & N&$O\left(\frac{\mathcal{V}\log\bb{d}}{\gap^2}\frac{1}{n}\right)$ & $O\left(\frac{\mathcal{V}\log\bb{d}}{\gap^2}\frac{1}{\epsilon}\right)$ \\
         \hline
         Chen et al.  &  Y &Y&N & - & $O\bb{\frac{G}{\gap^2}\frac{1}{\epsilon}\log^{2}\bb{\frac{G}{\gap^2}\frac{1}{\epsilon}}}$ \\
         \cite{chen2018dimensionality}&&&&&\\
         \hline
         Neeman et al.  & Y & N &N &$O\left(\frac{\mathcal{V}\log\bb{d^{2-\frac{\pi}{4}}}}{\bb{1-|\lambda_{2}\bb{P}|}\gap^2}\frac{1}{n}\right)$ & $O\left(\frac{\mathcal{V}\log\bb{d^{2-\frac{\pi}{4}}}}{\bb{1-|\lambda_{2}\bb{P}|}\gap^2}\frac{1}{\epsilon}\right)$\\
         \cite{neeman2023concentration}&&&&&\\
         \hline
         \textbf{Theorem \ref{theorem:oja_convergence_rate}} & Y & Y &Y& $O\left(\frac{\mathcal{V}}{\bb{1-|\lambda_{2}\bb{P}|}\gap^2}\frac{1}{n}\right)$ & $O\left(\frac{\mathcal{V}}{\bb{1-|\lambda_{2}\bb{P}|}\gap^2}\frac{1}{\epsilon}\right)$ \\
         \hline
    \end{tabular}
    \vspace{5pt}
    \caption{Comparison of $\sin^2$ error rates and sample complexities. Here $\gap := \bb{\lambda_{1}-\lambda_{2}}$, where $\lambda_{1},\lambda_{2}$ are the top 2 eigenvalues of $\Sigma$ and the sample complexities represent the number of samples required to achieve $\sin^{2}$ error at most $\epsilon$. We note that \cite{allenzhu2017efficient} and \cite{DBLP:journals/corr/abs-2102-03646} also match the online sample complexity bound provided in \cite{jain2016streaming}.  Further, for the offline algorithm with IID data, \cite{jin2015robust} removes the $\log\bb{d}$ factor in exchange for a constant probability of success for large enough $n$.\bk}
    \label{tab:rate_comparison}
\end{table}

\textbf{Related work on streaming PCA and online matrix decomposition on Markovian data:}
Amongst recent work, \cite{chen2018dimensionality} is very relevant to our setting, since it analyzes Oja's algorithm with Markovian Data samples. Inspired by the ideas of \cite{duchi2012ergodic}, the authors propose a downsampled version of Oja's algorithm to reduce dependence amongst samples and provide a Stochastic Differential Equation (SDE) based analysis to achieve a sample complexity of $O\bb{\frac{G}{\bb{\lambda_{1}-\lambda_{2}}^2}\frac{1}{\epsilon}\log^{2}\bb{\frac{G}{\bb{\lambda_{1}-\lambda_{2}}^2}\frac{1}{\epsilon}}}$ for $\sin^{2}$ error smaller than $\epsilon$, where $G$ is a variance parameter. We obtain a similar rate in Corollary \ref{cor:data-drop} through our techniques. However, comparing with Theorem \ref{theorem:offline_sample_complexity}, we observe that downsampling leads to an extra $O\bb{\log\bb{n}}$ factor. It is important to point out that \cite{chen2018dimensionality} provides an analysis for estimating top $k$ principal components, whereas this paper focuses on obtaining a sharp rate for the first principal component. ~\cite{needell2020Markov} consider the harder problem of online non-negative matrix factorization for Markovian data. Their analysis establishes asymptotic convergence of error, but does not provide a rate.

\begin{wrapfigure}[23]{r}{0.5\linewidth}
        \centering
     \includegraphics[width=0.6\textwidth]{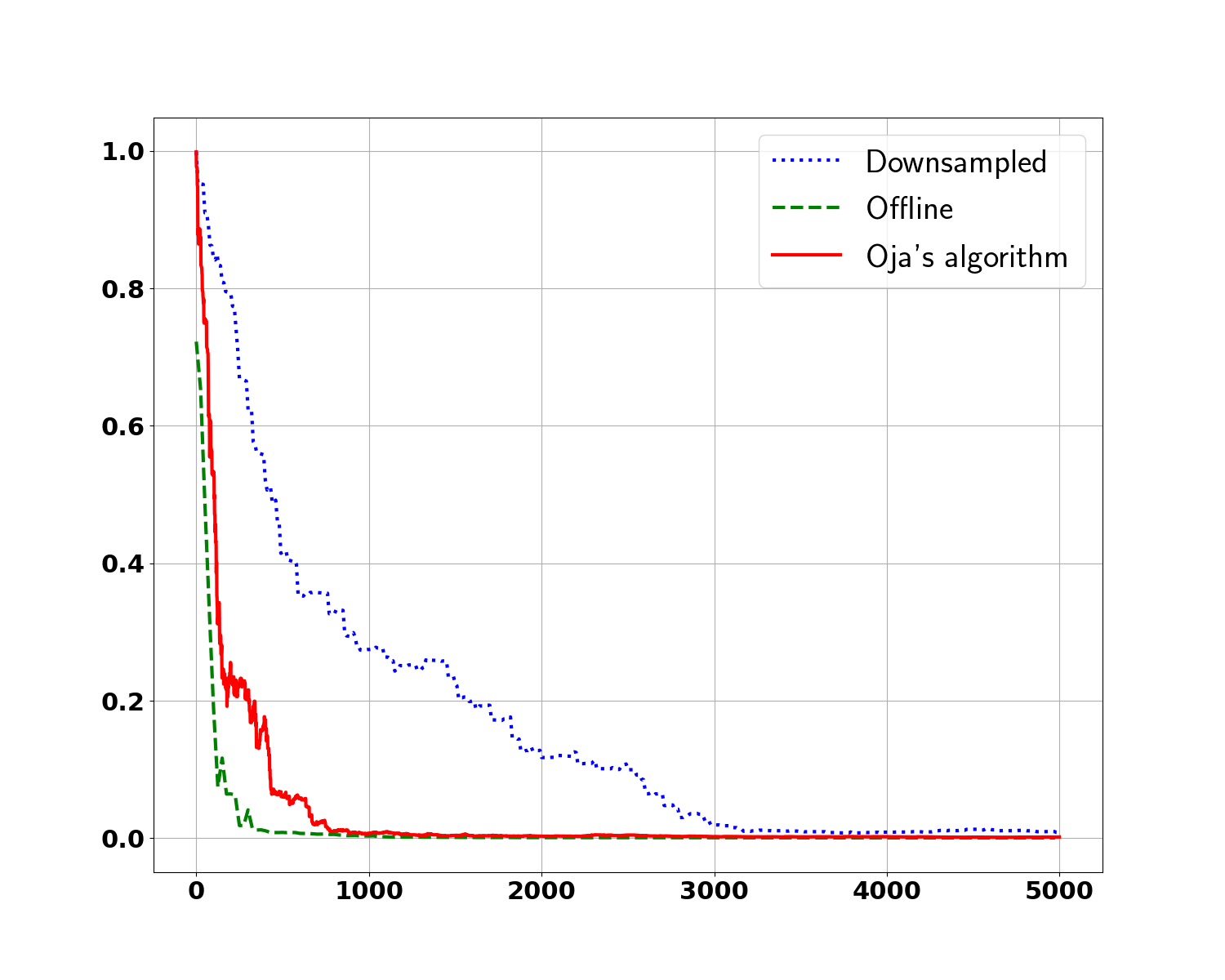}
        \caption{Comparison of Oja's algorithm with downsampled and offline variants. The X-axis represents the sample size and the Y axis represents the $\sin^{2}$ error of each algorithm's estimate of the leading eigenvector. The experimental setup is available in Section ~\ref{section:experiments}. }
         \label{fig:online_pca_comparison1}
\end{wrapfigure}
\textbf{Stochastic Optimization with Markovian Data} : 
Markovian models are often considered in Reinforcement Learning and Linear Dynamic Systems\cite{DBLP:journals/corr/abs-1806-02450, doan2020convergence, durmus2021stability, srikant2019finitetime, chen2020explicit,NIPS1996_e0040614, kushner2003applications, MOKKADEM1988309}. 
There have been many notable nonasymptotic bounds for stochastic gradient descent (SGD) methods for general convex  and nonconvex functions with Markovian data~\cite{duchi2012ergodic,sun2018Markov, doan2020finitetime, dorfman2022adapting, even2023stochastic, ziemann2023learning, truong2022generalization}. 
The convergence rates (sample complexities) obtained in these works apply to more general problems but do not exploit the matrix product structure inherent to Oja's algorithm. In this work, we develop novel techniques to show that a sharper analysis is possible for the PCA objective.
The paper is organized as follows. Section~\ref{section:problem_setup} contains the problem setup and preliminaries about Markov Chains. Section~\ref{section:main_results} contains Theorem~\ref{theorem:oja_convergence_rate}. We present a sketch of the main technical tools in Section~\ref{sec:maintools}, intermediate theorems needed for the main theorem in Section~\ref{sec:intermed} and conclude with simulations in Section~\ref{section:experiments}.
\section{Problem Setup and Preliminaries}
\label{section:problem_setup}
This section presents the problem setup and outlines important properties of the Markov chain that will be utilized subsequently. We assume that:
\begin{assumption}\label{ass:model}
    The Markov chain is irreducible, aperiodic, reversible, and starts in stationarity, with state distribution $\pi$\footnote{The last assumption may be eliminated by observing an initial burn-in period of $\tau_{\text{mix}}$.}.
\end{assumption}  Such a Markov chain can arise in various situations, for e.g., while performing random walks on expander graphs which are used extensively in fields such as computer networks, error-correcting codes, and pseudorandom generators. Each state $s$ of the Markov chain is associated with a distribution $D(s)$ over $d$-dimensional vectors with mean $\mu_{s} \in \mathbb{R}^{d}$ and covariance matrix $\Sigma_{s} \in \mathbb{R}^{d \times d}$. 
\\ \\ \noindent
For a random walk $s_{1}, s_{2}, \cdots s_{t}$ on $C$, we define the sequence of random variables $X_{1}, X_{2} \cdots X_{t}$, where conditioned on the state $s_i$, $X_{i} \sim D\bb{s_{i}}$.
We represent the mean as $\mu := \E_{s \sim\pi}\bbb{\mu_{s}}=\sum_s\pi_s\mu_s$ and the covariance matrix as $\Sigma \in \mathbb{R}^{d \times d}$, which, for $i\in [n]$ can be expressed as:
\bas{
    \Sigma &:= \mathbb{E}_{s_i \sim \pi}\mathbb{E}_{D(s_i)}\bbb{\bb{X_i-\mu}\bb{X_i-\mu}^{T}} 
    = \mathbb{E}_{s \sim \pi}\bbb{\Sigma_{s}}  + \mathbb{E}_{s \sim \pi}\bbb{\mu_{s}\mu_{s}^{T}} -  \mu\mu^{T}
}
In this work, we assume $\mu = 0$,  which is a common assumption in the IID setting (see \cite{jain2016streaming, allenzhu2017efficient}). While it may be possible to extend our analysis to the non-zero mean case, it is out of the scope of this paper.
Therefore, $\Sigma = \mathbb{E}_{s_i \sim \pi}\mathbb{E}_{D(s_i)}\bbb{X_iX_i^{T}}$ for $i\in [n]$

Let the eigenvalues of $\Sigma$ be denoted as $\lambda_{1} > \lambda_{2} \geq \lambda_{3} \cdots \lambda_{d}$. Let $v_{1}$ denote the leading eigenvector of $\Sigma$ and $\vp$ denote the $\mathbb{R}^{d \times \bb{d-1}}$ matrix with the remaining eigenvectors as columns. 
We proceed under the following standard assumptions for $i\in[n]$, (see for eg. \cite{ DBLP:journals/corr/abs-2102-03646}). 
\begin{assumption}
$\|\mathbb{E}_{s_i \sim \pi}\mathbb{E}_{D(s_i)}[(X_iX_i^{T}-\Sigma)^{2}]\|_{2} \; \leq \mathcal{V}$.
\label{assumption:variance_bound}
\end{assumption}

\begin{assumption}
    $\|X_iX_i^{T}-\Sigma\|_{2} \; \leq \mathcal{M}$ with probability 1. \label{assumption:norm_bound}
\end{assumption}
\noindent
Assumption \ref{assumption:norm_bound} also implies $\|\Sigma_{s} + \mu_{s}\mu_{s}^{T}-\Sigma\|_{2} \; \leq \mathcal{M}$ with probability 1. WLOG, we assume $\mathcal{M} + \lambda_{1} \geq 1$. We use $\E\bbb{.} := \mathbb{E}_{s \sim \pi}\mathbb{E}_{D(s)}\bbb{.}$ to denote the expectation over state $s\sim \pi$  and over the state-specific distributions $D\bb{.}$, unless otherwise specified.
\noindent
Define the matrix product
\ba{
    & \B_{t} := \bb{I+\eta_{t}X_tX_t^{T}}\bb{I+\eta_{t}X_{t-1}X_{t-1}^{T}}\dots\bb{I+\eta_{1}X_1X_1^{T}} \label{definition:Bt} 
}
Unrolling the recursion in \ref{alg:oja}, the output of Oja's algorithm at timestep $t$ is given as $ w_{t} = \left.B_{t}w_{0}\right/\left\|\B_{t}w_{0}\right\|_{2}$.
In this work, $\|.\|_{2}$ denotes the Euclidean $L_{2}$ norm for vectors and the operator norm for matrices unless otherwise specified. $I$ denotes the identity matrix.

\subsection{Markov chain mixing times}
\label{section:preliminaries}
Now we will discuss some well-known properties of an irreducible, aperiodic, and reversible Markov chain (also see~\cite{levin2017Markov}). Let $|\lambda_{2}\bb{P}|$ denote the second largest absolute eigenvalue of the Markov chain; let the state-distribution of the Markov chain at timestep $t$ with $s_1 = x$ be $P^{t}(x,.)$. For any two probability distributions $\nu_1$ and $\nu_2$, recall that the total variational distance is $TV\bb{\nu_1, \nu_2} := \|\nu_1-\nu_2\|_{TV} := \frac{1}{2}\sum_{x\in\Omega}|\nu_1(x)-\nu_2(x)|.$
The distance from $\pi$ at the $t^{\text{th}}$ timestep is defined as $
 d_{\text{mix}}(t) := \sup_{x \in \Omega}TV(P^{t}(x,.), \pi) $.
For irreducible and aperiodic Markov chains, by Theorem 4.9 in \cite{levin2017Markov}, we have $d_{\text{mix}}(t) \leq C\exp(-ct)$ \text{ for some } $C, c > 0$.
The mixing time of the Markov chain is defined as:
\ba{
 \tau_{\text{mix}}(\epsilon) := \inf \{t : d_{\text{mix}}(t) \leq \epsilon \}\label{eq:tmix}
 }
As in~\cite{levin2017Markov}, we will denote $\tau_{\text{mix}} := \tau_{\text{mix}}\bb{\frac{1}{4}} $. Then, we have $\begin{aligned}\tau_{\text{mix}}(\epsilon) \leq \left\lceil \log_{2}(1/\epsilon) \right\rceil \tau_{\text{mix}}\end{aligned}$.
It is worth mentioning the useful relationship between $d_{\text{mix}}\bb{.}$ and $\tau_{\text{mix}}$, given as
$\begin{aligned}
d_{\text{mix}}\bb{l\tau_{\text{mix}}} \leq 2^{-l}\qquad \forall l \in \mathbb{N}_{0}.\end{aligned}$
 These results about mixing time are valid for general irreducible and aperiodic Markov chains. A reversible Markov chain satisfies $\forall \; x,y \in \Omega$, $ \pi\bb{x}P\bb{x,y} = \pi\bb{y}P\bb{y,x}$.
For a reversible, irreducible, and aperiodic Markov chain, the gap $1-|\lambda_2(P)|$, is inversely proportional to $\tau_{\text{mix}}$~\cite{levin2017Markov}. 
\section{Main Results}
\label{section:main_results}
In this section, we present our main result, a near-optimal convergence rate for Oja's algorithm on Markovian data. As a corollary, we also establish a rate of convergence for Oja's algorithm applied on downsampled data, where every $k^{\text{th}}$ data-point is considered. Supplement \ref{appendix:proofs_of_main_results} contains comprehensive proofs of Theorem \ref{theorem:oja_convergence_rate} and Corollary \ref{cor:data-drop} while the proof of Proposition \ref{theorem:offline_sample_complexity} can be found in Supplement Section \ref{appendix:offline_pca_with_Markovian_data}.

\begin{theorem} \label{theorem:oja_convergence_rate}
    Fix a $\delta \in \bb{0,1}$ and let the step-sizes be $\eta_{i} := \frac{\alpha}{\bb{\lambda_{1}-\lambda_{2}}\bb{\beta + i}}$ with $\eta_{0} \leq \frac{1}{e}, \alpha > 2$. Under assumptions~\ref{ass:model},~\ref{assumption:variance_bound} and~\ref{assumption:norm_bound}, for sufficiently large number of samples $n$ such that $
    \frac{n}{\log\bb{\frac{1}{\eta_{n}}}} > \frac{\beta}{\log\bb{\frac{1}{\eta_{0}}}}
$,
\bas{   
        \beta := \frac{1000\alpha^{2}\max\left\{\tau_{\text{mix}}\log\bb{\frac{1}{\eta_0}}\bb{\mathcal{M}+\lambda_{1}}^{2}, \frac{\bb{\frac{\mathcal{V}}{1 - \left|\lambda_{2}\bb{P}\right|} + \lambda_{1}^{2}}}{100}\right\}}{\bb{\lambda_{1}-\lambda_{2}}^{2}\log\bb{1+\frac{\delta}{200}}}
}
the output $w_{n}$ of Oja's algorithm (\ref{alg:oja}) satisfies 
    \bas {
        & 1 - \bb{w_{n}^{T}v_{1}}^{2} \leq  \frac{C\log\bb{\frac{1}{\delta}}}{\delta^{2}}\bbb{d\bb{\frac{2\beta}{n}}^{2\alpha} + \frac{C_{1}\mathcal{V}}{\bb{\lambda_{1}-\lambda_{2}}^{2}\bb{1-|\lambda_{2}\bb{P}|}}\frac{1}{n} + \frac{C_{2}\mathcal{M}\bb{\mathcal{M}+\lambda_{1}}^{2}}{\bb{\lambda_{1}-\lambda_{2}}^{3}}\frac{\tau_{\text{mix}}\bb{\eta_{n}^{2}}^{2}}{n^{2}}}
    }
    with probability atleast $\bb{1-\delta}$. Here $C$ is an absolute constant and 
    \bas{
        C_{1} := \frac{\alpha^{2}\bb{3 + 7|\lambda_{2}\bb{P}|}}{2\alpha - 1}, \;\; C_{2} := \frac{35\alpha^{3}}{\alpha - 1}
    }
\end{theorem}
Next, we compare the rate of convergence proposed in Theorem \ref{theorem:oja_convergence_rate} with the offline algorithm having access to the entire dataset $\left\{X_i\right\}_{i=1}^{n}$ using a recent result from \cite{neeman2023concentration}. Here, the authors extend the Matrix Bernstein inequality \cite{vershynin2010introduction, tropp2012user}, to Markovian random matrices. Their setup is much like ours except that the matrix at any state is fixed, i.e., there is no data distribution $D(s)$ as in our setup. However, it is easy to extend their result to our setting by observing that conditioned on the state sequence, the matrices $X_iX_i^{T}, i \in [n]$ are independent under our model, and we can push in the expectation over the state-specific distributions, $D(s)$, whenever required. Therefore, we have the following result - 

\begin{proposition}[Theorem~2.2 of~\cite{neeman2023concentration}+Wedin's theorem] \label{theorem:offline_sample_complexity}
Fix $\delta \in \bb{0,1}$. Consider an irreducible and aperiodic Markov chain. Under assumptions~\ref{assumption:variance_bound} and~\ref{assumption:norm_bound}, with probability $1-\delta$, the leading eigenvector $\hat{v}$ of 
 $ \sum_{i=1}^{n}X_iX_i^{T}/n$ satisfies
\ba{
    1 - \bb{\hat{v}^{T}v_{1}}^{2} \leq C_{1}'\frac{\mathcal{V}\log\bb{\frac{d^{2-\frac{\pi}{4}}}{\delta}}}{\bb{\lambda_{1}-\lambda_{2}}^{2}}\bb{\frac{1 + \left|\lambda_{2}\bb{P}\right|}{1 - \left|\lambda_{2}\bb{P}\right|}}.\frac{1}{n} + C_{2}'\bb{\frac{\mathcal{M}\log\bb{\frac{d^{2-\frac{\pi}{4}}}{\delta}}}{\bb{\lambda_{1}-\lambda_{2}}\bb{1-|\lambda_{2}\bb{P}|}}}^{2}.\frac{1}{n^{2}}
}
for absolute constants $C_{1}'$ and $C_{2}'$.
\end{proposition}
Observe that Theorem \ref{theorem:oja_convergence_rate} matches the leading term $\frac{\mathcal{V}}{\bb{\lambda_{1}-\lambda_{2}}^{2}\bb{1-|\lambda_{2}\bb{P}|}}$ in Eq~\ref{theorem:offline_sample_complexity} except the $\log(d)$ term. We believe, much like the IID case (also see the remark in~\cite{jain2016streaming}), this logarithmic term in~\cite{neeman2023concentration}'s result is removable for large $n$ and a constant probability of success.
\begin{remark}
    (\textit{Comparison with IID algorithm}) Fix a $\delta \in \bb{0,1}$. If the data-points $\left\{X_i\right\}_{i=1}^{n}$ are sampled IID from the stationary distribution $\pi$, then under assumptions~\ref{assumption:variance_bound} and~\ref{assumption:norm_bound}, using Theorem 4.1 from \cite{jain2016streaming}, we have that the output $w_{n}$ of Oja's algorithm \ref{alg:oja} satisfies - 
    \ba{
        1 - \bb{w_n^{T}v_{1}}^{2} \leq \frac{C\log\bb{\frac{1}{\delta}}}{\delta^{2}}\bbb{d\bb{\frac{\beta'}{n}}^{2\alpha} + \frac{\alpha^{'2}\mathcal{V}}{\bb{2\alpha' - 1}\bb{\lambda_{1}-\lambda_{2}}^{2}}\frac{1}{n}} \label{theorem:iid_convergence_rate}
    }
\end{remark}
The leading term of Theorem \ref{theorem:oja_convergence_rate} is worse by a factor of $\frac{1}{1-|\lambda_{2}\bb{P}|}$. Further, it has an additive lower order term $O\bb{\frac{\log^{2}\bb{n}}{n^{2}}}$ due to the covariance between data-points in the Markovian case.
\begin{corollary}
\label{cor:data-drop} (Downsampled Oja's algorithm)
Fix a $\delta \in \bb{0,1}$. If Oja's algorithm is applied on the downsampled data-stream with every $k^{\text{th}}$ data-point, where $k := \tau_{\text{mix}}\bb{\eta_{n}^{2}}$ then under the conditions of Theorem \ref{theorem:oja_convergence_rate} with appropriately modified $\alpha$ and $\beta$, the output $w_{n}$ satisfies 
    \bas {
        & 1 - \bb{w_{n}^{T}v_{1}}^{2} \leq  \\ 
        & \;\;\;\;\;\;\;\;\; \frac{C\log\bb{\frac{1}{\delta}}}{\delta^{2}}\bbb{d\bb{\frac{2\beta\tau_{\text{mix}}\log\bb{n}}{n}}^{2\alpha} + \frac{C_{1}\mathcal{V}\tau_{\text{mix}}}{\bb{\lambda_{1}-\lambda_{2}}^{2}}\frac{\log\bb{n}}{n} + \frac{C_{2}\mathcal{M}\bb{\mathcal{M}+\lambda_{1}}^{2}}{\bb{\lambda_{1}-\lambda_{2}}^{3}}\frac{\log^{2}\bb{n}\tau_{\text{mix}}\bb{\eta_{n}^{2}}^{2}}{n^{2}}}
    }
    with probability atleast $\bb{1-\delta}$. Here $C$ is an absolute constant and $C_{1} := \frac{30\alpha^{2}}{2\alpha - 1}, \;\; C_{2} := \frac{35\alpha^{3}}{\alpha - 1}$.
\end{corollary}


\begin{remark} Data downsampling to reduce dependence amongst samples has been suggested in recent work \cite{bresler2020squares, ma2022data, chen2018dimensionality}. 
In Corollary~\ref{cor:data-drop}, we establish that the rate obtained is sub-optimal compared to Theorem~\ref{theorem:oja_convergence_rate} by a $\log\bb{n}$ factor. We prove this by a simple yet elegant observation: the downsampled data stream can be considered to be drawn from a Markov chain with transition kernel $P^{k}\bb{.,.}$ since each data-point is $k$ steps away from the previous one. For sufficiently large $k$, this implies that the mixing time of this chain is $\Theta\bb{1}$. These new parameters are used to select the modified values of $\alpha, \beta$ according to Lemma \ref{lemma:learning_rate_schedule} in the Supplement.
\end{remark}

The proof of Theorem~\ref{theorem:oja_convergence_rate} follows the same general recipe as in~\cite{jain2016streaming} for obtaining a bound on the $\sin^2$ error. However, the original proof techniques heavily rely on the IID setting. We carry out a refined analysis for each step under the Markovian data model by a careful control of error terms arising out of dependence.
The first step involves obtaining a high-probability bound on the $\sin^{2}$ error, by noting that Oja's algorithm on $n$ data-points can be viewed as a single iteration of the power method on $B_{n}$. Therefore, fixing a $\delta \in \bb{0,1}$ using Lemma 3.1 from \cite{jain2016streaming}, we have with probability at least $\bb{1-\delta}$,
    \ba{
        \sin^{2}\bb{w_{n},v_{1}} \leq \frac{C\log\bb{\frac{1}{\delta}}}{\delta}\frac{\Tr\left(V_{\perp}^{T}\B_{n}\B_{n}^{T}V_{\perp}\right)}{v_{1}^{T}\B_{n}\B_{n}^{T}v_{1}}, \label{eq:sin2errorbound}
    }
where $C$ is an absolute constant. The numerator is bounded by first bounding its expectation (see Theorem~\ref{theorem:Vperpupperbound}) and then using Markov's inequality. To bound the denominator, similar to~\cite{jain2016streaming}, we will use Chebyshev's inequality. Theorem~\ref{theorem:v1lowerbound} provides a lower bound for the expectation $\E\bbb{v_{1}^{T}\B_{n}\B_{n}^{T}v_{1}}$. Chebyshev's inequality also requires upper-bounding the variance of $\E\bbb{v_{1}^{T}\B_{n}\B_{n}^{T}v_{1}}$,
 which requires us to bound $\E\bbb{\bb{v_{1}^{T}\B_{n}\B_{n}^{T}v_{1}}^2}$
(see Theorem~\ref{theorem:v1squareupperbound}). 

\section{Main Technical Tools}\label{sec:maintools}
In this section, we provide a sketch of the main arguments used in our proof. 

\textbf{Warm-up with downsampled Oja's algorithm: }
We start with the simple downsampled Oja's algorithm to build intuition. Here, one applies Oja's update rule (Eq~\ref{alg:oja}) to every $k^{th}$ data-point, for a suitably chosen $k$. For $k=\lceil L\tau_{\text{mix}}\log n\rceil$, the total variation distance between any consecutive data-points in the downsampled data stream is $O(n^{-L})$. As we show in Corollary~\ref{cor:data-drop}, the error of this algorithm is similar to the error of Oja's algorithm applied to $n/k$ data-points in the IID setting, i.e., $O(\mathcal{V}\tau_{\text{mix}}\log n/n)$.

We will take $\E\bbb{v_1^TB_nB_n^Tv_1}$ as an example. Let us introduce some notation. 
\ba{
    \newB{i}{j} := \bb{I+\eta_{j}X_jX_j^T}\bb{I+\eta_{j-1}X_{j-1}X_{j-1}^T}\dots\bb{I+\eta_{i}X_iX_i^T} \label{definition:Bji}
} 
We peel this quantity one matrix at a time from the inside. Note that for a reversible Markov chain, standard results imply (see Lemma~\ref{lemma:reverse_mixing}) that the mixing conditions apply to the conditional distribution of a state given another state $k$ steps in the ``future'' (see Supplement section \ref{appendix:useful_results} for a proof). Recall $d_{\text{mix}}(k)$ from Section~\ref{section:preliminaries}.
\begin{lemma}
\label{lemma:reverse_mixing}
    Under Assumption~\ref{ass:model},
    $\begin{aligned}
        \frac{1}{2}\sup_{t \in \Omega}\sum_{s}\left|\mathbb{P}\bb{Z_{t}=s | Z_{t+k} = t} - \pi\bb{s}\right| = d_{\text{mix}}\bb{k}
    \end{aligned}$.
\end{lemma}
\noindent
It will be helpful to explain our analysis by comparing it with the IID setting. For this reason, we will use $\E_{\text{IID}}[.]$ to denote the expectation under the IID data model.
\hspace{-4em}
\ba{\alpha_{n,1} &:=\E\bbb{v_1^TB_nB_n^Tv_1}=\E\bbb{v_1^TB_{n,2}\bb{I+\eta_1\Sigma+\eta_1(X_1X_1^T-\Sigma)}\bb{I+\eta_1\Sigma+\eta_1(X_1X_1^T-\Sigma)}^TB_{n,2}^Tv_1}\notag\\
&=\E\bbb{v_1^TB_{n,2}\bb{I+\eta_1\Sigma}^2B_{n,2}^Tv_1}+2\eta_1T_1+\eta_1^2 T_2,
\label{eq:v1bound}} 
where the first term is smaller than $(1+\eta_1\lambda_1)^2\alpha_{n,2}$. We define $T_1$ and $T_2$ as follows.
$T_1:=\E\bbb{v_1^TB_{n,2}\bb{I+\eta_1\Sigma}\bb{X_1X_1^T-\Sigma}B_{n,2}^Tv_1}$, and $T_2:=\E\bbb{v_1^TB_{n,2}\bb{X_1X_1^T-\Sigma}^2B_{n,2}^Tv_1}$.

For the IID setting, the \textit{second term is zero}, and the third term can be bounded as follows:
\bas{
\E_{\text{IID}}\bbb{v_1^TB_{n,2}\bb{X_1X_1^T-\Sigma}^2B_{n,2}^Tv_1}=\E_{\text{IID}}\bbb{v_1^TB_{n,2}\E\bbb{\bb{X_1X_1^T-\Sigma}^2}B_{n,2}^Tv_1}\leq \mathcal{V}\E_{\text{IID}}\bbb{v_1^TB_{n,2}B_{n,2}^Tv_1}
}
Let us denote the IID version of $\alpha_{n,i}$ by $\alpha^{\text{IID}}_{n,i}=\E_{\text{IID}}[v_1^TB_{n,i}B_{n,i}^Tv_1]$. The  final recursion for the IID case becomes:
$\begin{aligned}
\alpha^{\text{IID}}_{n,1}\leq (1+2\eta_1\lambda_1+\eta_1^2\bb{\lambda_1^2+\mathcal{V}})\alpha^{\text{IID}}_{n,1}
\end{aligned}.$
So, for our Markovian data model, the hope is that the cross term $T_1$ (which has a multiplicative factor of $\eta_1$) is $O(\eta_1)$ and $T_2$ is $O(\eta_1^2)$. We will start with the $T_1$ term, which is zero in the IID setting.


We hope to reduce the product $B_{n,2}(X_1X_1^T-\Sigma)$ into a product of nearly independent matrices. One hope is that if instead of $B_{n,2}$, we had $B_{n,2+k}$ for some suitably large integer $k$, then using (reverse) mixing properties of the Markov chain, we could argue using Lemma~\ref{lemma:reverse_mixing} that $\E[X_1X_1^T-\Sigma|s_{1+k},\dots, s_n]$ is very close to zero.
The following lemma formally bounds the deviation of the length-$k$ matrix product from identity.
\begin{lemma}
\label{lemma:etakproduct}
Let Assumption~\ref{assumption:norm_bound} hold. If $\forall i \in [n], \eta_{i}k_{i}\bb{\mathcal{M}+\lambda_{1}} \leq \epsilon, \epsilon \in \bb{0,1}$ and $\eta_{i}$ forms a non-increasing sequence then $\forall \; m \leq n - k_n$, 
\ba {
    & \;\; \left\| \newB{m}{m+k_{m}-1} - I \right\|_{2} \leq \bb{1+\epsilon}k_{m}\eta_{m}\bb{\mathcal{M}+\lambda_{1}} \text{ and } \label{eq:approx1}\\
    & \;\; \left\| \newB{m}{m+k_{m}-1} - I - \sum_{t=m}^{m+k_{m}-1}\eta_{t}X_tX_t^{T} \right\|_{2} \leq k_{m}^{2}\eta_{m}^{2}\bb{\mathcal{M}+\lambda_{1}}^{2}\label{eq:approx2}
}
\end{lemma}
Lemma \ref{lemma:etakproduct} bounds the norm of the matrix product $\newB{t}{t+k_{t}-1}$ at two levels. The first result provides a coarse bound, approximating linear and higher-order terms. The second result provides a finer bound, preserving the linear term and approximating quadratic and higher-order terms. The proofs involve a straightforward combinatorial expansion of $\newB{t}{t+k_{t}-1}$ and are deferred to the Supplement section~\ref{appendix:useful_results}. 

Approximating $\prod_{i=2}^{k+1}(I+\eta_iX_iX_i^T)$ requires $\eta_1 k$ to be small. Since this is a recursive argument, we would need $\eta_i k$ to be small for $i=1,\dots n$, which is satisfied by the strong condition $\eta_1 k$ is small. To obtain a tight analysis, we choose $k$ adaptively. We set $k_i=\tau_{\text{mix}}(\eta_i^2)$ (see definition in Eq~\ref{eq:tmix}). 

As we will show in detail in the Supplement, Lemma~\ref{lemma:etakproduct} Eq~\ref{eq:approx2} along with the adaptive choice of $k_i$ gives us a sharp error bound. Using it, we can bound $T_1$ (see Eq~\ref{eq:v1bound}) as:
\ba{
T_1
\leq\sum_{j=2}^{k+1}\eta_j\E\bbb{v_1^TB_{n,k+2} \underbrace{\E\bbb{\bb{ X_jX_j^T}\bb{I+\eta_1\Sigma}\bb{X_1X_1-\Sigma}|X_{k+2},\dots,X_n}}_{T_{1,j}}B_{n,k+2}^Tv_1}+O(\eta_1^2k_1^2)\alpha_{n,k+2}\notag
}
Naively bounding the $T_{1,j}$ term by $O(1)$ leads to the same rate as downsampled Oja's algorithm. 



In the following lemma, we will establish that, indeed, $T_{1,j}$ has a much smaller norm. The novelty of our bound is not just in using the mixing properties of the Markov chain but also in teasing out the variance parameter $\mathcal{V}$. We will state the lemma, in a slightly more general form as -
\begin{lemma}
\label{lemma:bound_covariance}
Under Assumptions~\ref{ass:model},~\ref{assumption:variance_bound} and~\ref{assumption:norm_bound}, for $i < j \leq i+k_{i}$,
    \bas{\left\|\E\bbb{\bb{X_iX_i^{T}-\Sigma}SX_jX_j^{T}|s_{i+k_{i}},\dots s_{n}}\right\|_{2} \leq \bb{\left|\lambda_{2}\bb{P}\right|^{j-i}\mathcal{V} + 8\eta_{i}^{2}\mathcal{M}\bb{\mathcal{M}+\lambda_{1}}}\left\|S\right\|_{2}
    }
where $k_{i}$ is as defined in Lemma~\ref{lemma:learning_rate_schedule} and $S$ is a constant symmetric positive semi-definite matrix.
\end{lemma}
Lemma~\ref{lemma:bound_covariance} bounds the norm of the covariance between matrices $\bb{X_iX_i^{T}-\Sigma}S$ and $X_jX_j^{T}$. In particular, this implies that the norm of $T_{1,j}$ decays as $\left|\lambda_{2}\bb{P}\right|^{j-1}$. The proof uses a spectral argument that replaces a coarse approximation by a sum of $k_i$ $O(1)$ terms to sum of $k$ \textit{exponentially decaying} terms, thereby removing the dependence on $k_i$, which can be as large as $\log(n)$. The proof is deferred to the Supplement section~\ref{appendix:proofs_of_bounding_convergence}.
The details can be found in Supplement section~\ref{appendix:proofs_of_bounding_convergence}.

Let $\left\{c_1,c_2,c_3,c_4\right\}$ be positive constants for ease of notation. Coming back to Eq~\ref{eq:v1bound}, we can bound $T_1$ as follows:
$\begin{aligned}
T_1\leq \alpha_{n,k+2}\bb{\eta_1 \frac{c_1|\lambda_2(P)|\mathcal{V}}{1-|\lambda_2(P)|}+c_2\eta_1^2k_1^2}
\end{aligned}$. A similar argument can be applied to bound $T_2$ as:
$T_2\leq \alpha_{n,k+2}\bb{\mathcal{V}+c_3\eta_1k_1^2}$.
Putting everything together in \ref{eq:v1bound}, we have 
\bas{
    \alpha_{n,1} &\leq \underbrace{\bb{\left(1+\eta_{1}\lambda_{1}\right)^{2} + \mathcal{V}}\alpha_{n,2}}_{\text{Recursion for IID setting}} \;\; + \;\;  \underbrace{\bb{\frac{c_1|\lambda_{2}\bb{P}|}{1 - \left|\lambda_{2}\bb{P}\right|}}\mathcal{V} \eta_{1}^{2}\alpha_{n,k+2}}_{\text{Error due to Markovian dependence}} \;\; + \;\; \underbrace{c_4\eta_{1}^{3}k_1^{2}\alpha_{n,k+2}}_{\text{Error due to approximation of matrix product}} 
}
Recursing on this inequality gives us our bound on $\E\bbb{v_1^TB_nB_n^Tv_1}$ (Theorem \ref{theorem:v1upperbound}). We are now ready to present all our accompanying theorems.
\section{Intermediate Theorems for Convergence Analysis}\label{sec:intermed}
In this section, we present our accompanying theorems which are used to obtain the main result in Theorem \ref{theorem:oja_convergence_rate}. But before doing so, we will need to establish some notation. Let $k_{i} := \tau_{\text{mix}}\bb{\eta_{i}^{2}}$, and the step-sizes be set as $\eta_{i} := \frac{\alpha}{\bb{\lambda_{1}-\lambda_{2}}\bb{\beta + i}}$ with $\alpha,\beta$ as defined in Theorem \ref{theorem:oja_convergence_rate}. Let $\epsilon := \frac{1}{100}$. As shown in Lemma \ref{lemma:learning_rate_schedule} in Supplement Section \ref{appendix:useful_results} our choice of step-sizes satisfy, $ \forall i \in [n]$,

\begin{enumerate}[label=\textbf{C.\arabic*}]
    \item[\textbf{C.1}]$\eta_{i}k_{i}\bb{\mathcal{M}+\lambda_{1}} \leq \epsilon$
    \qquad\quad \ \ \textbf{C.2  }\ \ (Slow decay) $\eta_{i} \leq \eta_{i-k_{i}} \leq \bb{1+2\epsilon}\eta_{i} \leq 2\eta_{i}$
\end{enumerate}
Further, we define  scalar variables -
\begin{align}
    r &:= 2\bb{1+\epsilon}k_{n}\eta_{n}\left(\mathcal{M}+\lambda_{1}\right), \qquad \zeta_{k,t} := 40k_{t+1}\bb{\mathcal{M}+\lambda_{1}}^{2}\notag\\
    \psi_{k,t} &:= 6\mathcal{M}\bbb{1 + 3k_{t+1}^{2}\bb{\mathcal{M}+\lambda_{1}}^{2}}, \qquad     \mathcal{V}' := \frac{1 + \bb{3 + 4\epsilon}|\lambda_{2}\bb{P}|}{1 - \left|\lambda_{2}\bb{P}\right|}\mathcal{V}
\end{align}
and recall the definitions of $B_{t}$ and $\newB{i}{j}$ in Eqs~\ref{definition:Bt} and~\ref{definition:Bji}, respectively. 
We are now ready to present the theoretical results needed to prove our main result. For simplicity of notation, we present versions of the results by using $\eta_{i} := \frac{\alpha}{\bb{\lambda_{1}-\lambda_{2}}\bb{\beta + i}}$ with $\alpha,\beta$ as defined in Theorem \ref{theorem:oja_convergence_rate}. However, these theorems are in fact valid under more general step-size schedules. We state and prove the more general versions in the Supplement Section~\ref{appendix:proofs_of_bounding_convergence}.
\begin{theorem}
\label{theorem:v1upperbound}
Under Assumptions~\ref{ass:model},~\ref{assumption:variance_bound} and~\ref{assumption:norm_bound}, for all $n>k_n$, and $\eta_i$ satisfying~\textbf{C.1} and~\textbf{C.2}, we have:
$\begin{aligned}
    & \E\bbb{v_1^TB_nB_n^Tv_1} \leq \bb{1+r}^{2}\exp\left(\sum_{t=1}^{n-k_{n}}\bb{2\eta_{t}\lambda_{1} + \eta_{t}^{2}\bb{\mathcal{V}'+\lambda_1^2}+\eta_t^3\psi_{k,t}}\right).
\end{aligned}$ 
\end{theorem}
The three primary differences with the IID case are a) the $\bb{1+r}^{2}$ term, which arises since the recursion sketched in Section~\ref{sec:maintools} leaves out the last $k_n$ terms which are bounded by $\bb{1+r}^{2}$; (b) the new factor of $\frac{1}{1-|\lambda_{2}\bb{P}|}$ with $\mathcal{V}$ due to the Markovian dependence between terms; and c) the extra lower order term $\eta_t^3\psi_{k,t}$ arising from the use of Lemmas~\ref{lemma:etakproduct} and~\ref{lemma:bound_covariance}. 
\begin{theorem}
\label{theorem:Vperpupperbound}
Let $\init := \min\left\{t : t \in [n], t-k_{t} \geq 0 \right\}$. Under Assumptions~\ref{ass:model},~\ref{assumption:variance_bound} and~\ref{assumption:norm_bound}, for all $n>\init$, and $\eta_i$ satisfying~\textbf{C.1} and~\textbf{C.2},
\begin{align*}
\mathbb{E}\left[\Tr\left(V_{\perp}^{T}\B_{n}\B_{n}^{T}V_{\perp}\right)\right] &\leq \bb{1+5\epsilon}\exp\left(\sum_{t=\init+1}^n 2\eta_t\lambda_2+\eta_{t-k_{t}}^2\bb{\mathcal{V}'+\lambda_1^2}+\eta_{t-k_{t}}^3\psi_{k,t}\right)\\
&\qquad\times\left(d + \sum_{t=\init+1}^{n}\bb{\mathcal{V}'+\eta_t\psi_{k,t}}C_{k,t}'\eta_{t-k_{t}}^{2}\exp\left(\sum_{i=\init+1}^{t}2\eta_i\left( \lambda_1-\lambda_{2}\right)\right)\right)
\end{align*}
where $C_{k,t}' := \bb{1+\frac{\delta}{200}}\exp\bb{2\lambda_{1}\sum_{i=1}^{\init}\eta_{j}}$.
\end{theorem}
Here, the difference is mainly in the new variable $\init$, arising since the recursion stops at $u$, not $1$. $\bb{1+5\epsilon}$ represents the approximation of the first $\init$ terms.
\begin{theorem}
\label{theorem:v1lowerbound}
Under Assumptions~\ref{ass:model},~\ref{assumption:variance_bound} and~\ref{assumption:norm_bound}, for all $n>k_n$, $\eta_i$ satisfying~\textbf{C.1} and~\textbf{C.2}, and $s := 2r + \frac{\delta}{1000}$, we have:
$\begin{aligned}
& \mathbb{E}\left[v_{1}^{T}\B_{n}\B_{n}^{T}v_{1}\right]  \geq \bb{1-s}\exp\bb{\sum_{t=1}^{n-k_{n}}2\eta_{t}\lambda_{1} - \sum_{t=1}^{n-k_{n}}4\eta_{t}^{2}\lambda_{1}^{2}}.
\end{aligned}$
\end{theorem}
This differs from its IID counterpart by a multiplicative factor of $\bb{1-s}$ for the same reason as before, which also makes the sums go up to $\bb{n-k_n}$ instead of $n$. Note that for sufficiently large $n$ (Lemma \ref{lemma:additional_assumptions}), $r = O\bb{\frac{\log\bb{n}}{n}}$ is very small and $\delta \in \bb{0,1}$. Therefore, $\bb{1-s} \approx 1$ as large $n$. 
\begin{theorem}
\label{theorem:v1squareupperbound}
Under Assumptions~\ref{ass:model},~\ref{assumption:variance_bound} and~\ref{assumption:norm_bound}, for all $n>k_n$, and $\eta_i$ satisfying~\textbf{C.1} and~\textbf{C.2}, we have:
$\begin{aligned}
 \mathbb{E}\left[\left(v_{1}^{T}\B_{n}\B_{n}^{T}v_{1}\right)^{2}\right] \leq \bb{1+r}^{4}\exp\left(\sum_{t=1}^{n-k_{n}}4\eta_{t}\lambda_{1} + \sum_{t=1}^{n-k_{n}}\eta_{t}^{2}  \zeta_{k,t} \right)
\end{aligned}$.
\end{theorem}
The differences are similar to the last theorems involving $v_1$. Surprisingly, for this, the coarse approximation suffices, leading to an absence of the $\mathcal{V}$ term in the bound.
Having established these results, the final step is to substitute them into Eq~\ref{eq:sin2errorbound} and follow the proof recipe described earlier. This requires significant calculations and is deferred to the Supplement Section \ref{appendix:proofs_of_main_results}. 
\section{Experimental Validation}
\vspace{-.5em}
\label{section:experiments}
In this section, we present some simple experiments to validate our theoretical results. For more detailed experiments, see the Supplement. We design a Markov chain with $|\Omega|=10$ states, where the transition matrix entries $P_{ij}$ equal $\rho/(|\Omega|-1)$ for $i\neq j$ and $1-\rho$ for $i=j$. 
Smaller values of $\rho$ lead to larger mixing times. It can be verified that the stationary distribution $\pi = \mathcal{U}\bb{\Omega}$ is uniform over the state-space and $|\lambda_{2}\bb{P}| \approx \bb{1-\rho}$. We set $\rho = 0.2$ for Figures~\ref{fig:online_pca_comparison1} and \ref{fig:online_pca_comparison3}, and vary it in Figure \ref{fig:online_pca_lambda_variation}. Each point in the plot is averaged over 20 random runs over different Markov chains, datasets, and initialization.


Each state $s \in \Omega$ is associated with $D(s) :=$ Bernoulli($p_s$) distribution. We set $d=1000$ and select $p_{s}\sim \mathcal{U}\bb{0,0.05}$ at the start of each random run. The covariance matrix,  $\Sigma_{s}$, for each state is set as $\Sigma_{s}\bb{i,j} = \exp\bb{-|i-j|c_s}\sigma_{i}\sigma_{j}$ where $c_s := 1 + 9\bb{\frac{s-1}{|\Omega|-1}}, \sigma_{i} := 5i^{-\beta}$. 
We start with the stationary distribution $\pi$, and for each state $s_i$, we draw IID samples $Z_i \sim D\bb{s_i}$. We standardize $Z_i$ such that all components have zero mean and unit variance under the state distribution, $D\bb{s_i}$. We then generate the sample data-point for PCA as $X_i = \Sigma_{i}^{\frac{1}{2}}Z_i$. By construction, 
 $\E_{{D\bb{s_i}}}\bbb{X_iX_i^{T}} = \Sigma_{i}$ and $\E[X_i]=0^{d}$. The step sizes for Oja's algorithm are set as $\eta_{i} = \frac{\alpha}{\bb{\beta + i}\bb{\lambda_{1}-\lambda_{2}}}$ for $\alpha = 5, \beta = \frac{5}{1-|\lambda_{2}\bb{P}|}$. For the downsampled variant, every $10^{th}$ data-point is considered, and $\beta$ is accordingly divided by 10. For the offline algorithm, we recompute the leading eigenvector of the sample covariance matrix of data-points seen so far. 

Figure~\ref{fig:online_pca_comparison1} compares the performance of different algorithms for the Bernoulli distribution.  Here, we are checking if the results obtained in Theorem \ref{theorem:oja_convergence_rate}, Proposition \ref{theorem:offline_sample_complexity}, and Corollary \ref{cor:data-drop} are reflected in the experiments.
\begin{figure}[t]
    \centering
    \begin{subfigure}[b]{0.45\textwidth}
        \centering
        \includegraphics[width=\textwidth]{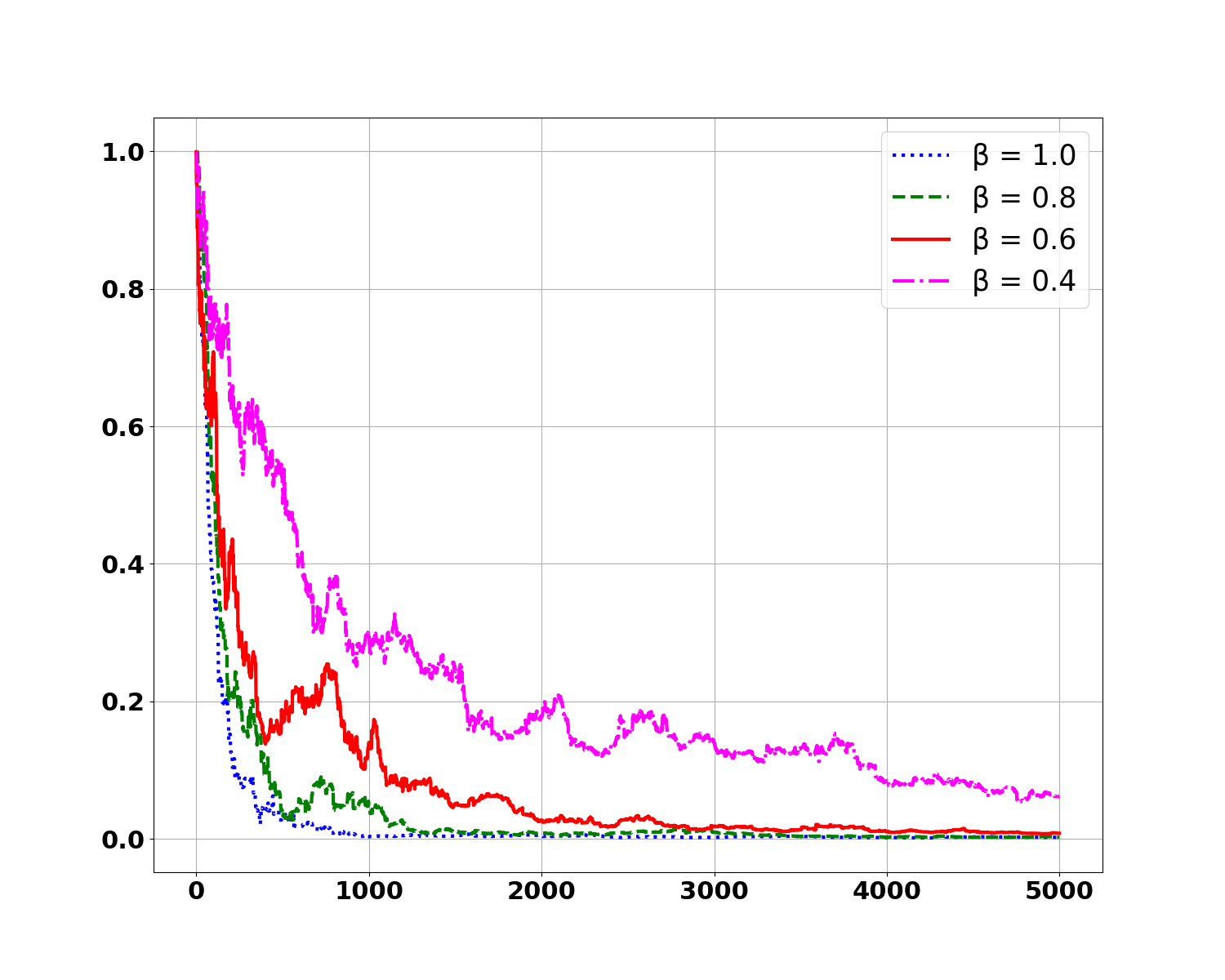}
        \caption{Variation of $\sin^2$ error with $\beta$}
         \label{fig:online_pca_comparison3}
    \end{subfigure}
    \begin{subfigure}[b]{0.45\textwidth}
        \centering
        \includegraphics[width=\textwidth]{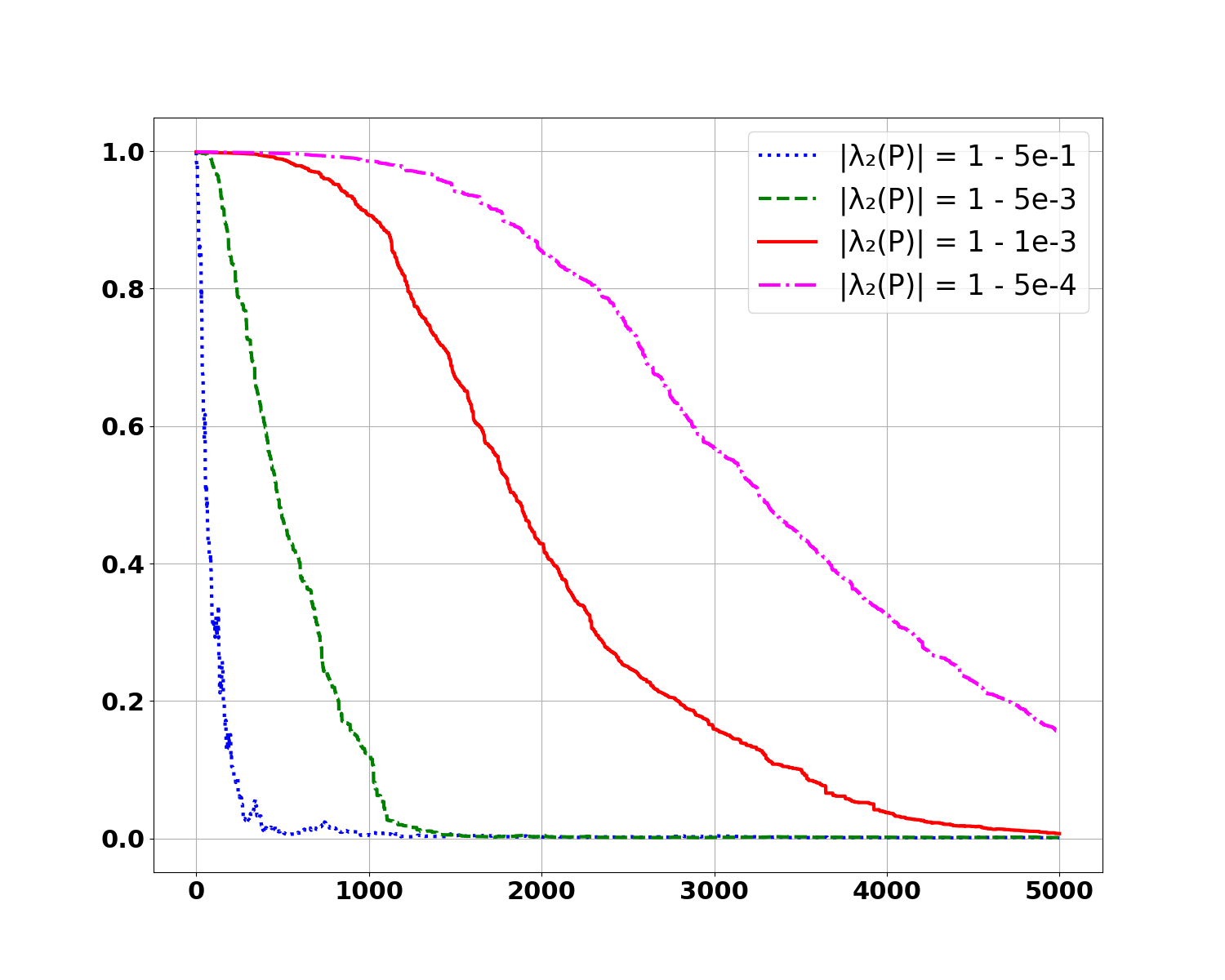}
        \caption{Variation of $\sin^2$ error with $|\lambda_{2}\bb{P}|$}
    \label{fig:online_pca_lambda_variation}
    \end{subfigure}
    \vspace{5pt}
    \caption{X axis represents the sample size, and Y axis represents the $\sin^{2}$ error.}
    \label{fig:experimental_validation}
\end{figure}
 The experimental results demonstrate that Oja's algorithm performs significantly better than the downsampled version, consistent with the theoretical results. It also shows that Oja's algorithm performs similarly to the offline algorithm, which is also confirmed by our theoretical results and that of~\cite{neeman2023concentration}. 
 Figure~\ref{fig:online_pca_comparison3} compares the performance of Oja's algorithm for different covariance matrices. Smaller values of $\beta$ decrease the eigengap $\lambda_{1}-\lambda_{2}$, and hence lead to a slower convergence. Figure \ref{fig:online_pca_lambda_variation} confirms that smaller values of $\rho$ (larger values of $|\lambda_{2}\bb{P}|$) also worsen the rate, which matches with our theoretical results. 


\vspace{-1em}
\section{Conclusion}
\vspace{-1em}
We have considered the problem of streaming PCA for Markovian data, which has implications in various settings like decentralized optimization, reinforcement learning, etc. The analysis of streaming algorithms in such settings has seen a renewed surge of interest in recent years. However, the dependence between data-points makes it difficult to obtain sharp bounds. 
We provide, to our knowledge, the first sharp bound for obtaining the first principal component from a Markovian data stream that breaks the logarithmic barrier present in the analysis done for downsampled data. We believe that the theoretical tools that we have developed in this paper would enable one to obtain sharp bounds for other dependent data settings, learning top $k$ principal components, and online inference algorithms with updates involving products of matrices. 

\section{Acknowledgements}
We gratefully acknowledge NSF grants 2217069 and DMS 2109155. We are also grateful to Rachel Ward and Bobby Shi for valuable discussions.

\bibliographystyle{plain}
\bibliography{refs}

\newpage

\appendix

\renewcommand{\thelemma}{S.\arabic{lemma}}
\renewcommand{\theproposition}{S.\arabic{proposition}}
\renewcommand{\thesection}{S.\arabic{section}}
\renewcommand{\thefigure}{S.\arabic{figure}}
\renewcommand{\theequation}{S.\arabic{equation}}

\setcounter{lemma}{0}
\setcounter{section}{0}
\setcounter{proposition}{0}
\setcounter{figure}{0}

\section*{Supplement}

    The Supplement is organized as follows - 

\begin{itemize}
    \item Section~\ref{appendix:notation} introduces notation that will be useful for concise representation.
    \item Section~\ref{appendix:offline_pca_with_Markovian_data} provides the proof of Proposition \ref{theorem:offline_sample_complexity}.
    \item Section~\ref{appendix:useful_results} contains useful intermediate results which are used in subsequent proofs of our main results.
    \item Section~\ref{appendix:proofs_of_bounding_convergence} proves bounds on $v_{1}B_{n}B_{n}^{T}v_{1}^{T}$ and $\vp B_{n}B_{n}^{T}\vp^{T}$ (Theorems~\ref{theorem:v1upperbound},~\ref{theorem:Vperpupperbound}, \ref{theorem:v1lowerbound} and~\ref{theorem:v1squareupperbound}).
    \item Section~\ref{appendix:proofs_of_main_results} puts everything together and provides proofs of our main result - Theorem~\ref{theorem:oja_convergence_rate}, along with Corollary~\ref{cor:data-drop}.
    \item Section~\ref{appendix:additional_experiments} provides additional experiments to further support our claims.
\end{itemize}

\section{Notation and assumptions}
\label{appendix:notation}
For conciseness, we define the stochastic function $A: \Omega \rightarrow \mathbb{R}^{d \times d}$ which maps each state variable of the Markov chain to a ($d \times d$) positive semi-definite symmetric matrix as 
\bas{
    A\bb{s_{t}} := X_{t}X_{t}^{T}
}
Where $X_{t} \sim D\bb{s_{t}}$ is drawn from the distribution corresponding to the state at timestep $s_{t}$.
All the theoretical results are derived under Assumptions~\ref{ass:model},~\ref{assumption:variance_bound} and~\ref{assumption:norm_bound}.
\section{Offline PCA with Markovian Data}
\label{appendix:offline_pca_with_Markovian_data}

In this section, we prove Proposition \ref{theorem:offline_sample_complexity}. 
We note that \cite{neeman2023concentration} considers $F_{j}\bb{s_{j}}$ to be random only with respect to the states. 
Therefore, we first show that their results generalize to our setting as well, using $F_{j}\bb{s_{j}} := A\bb{s_{j}}-\Sigma$. From Eq $\bb{5}$ in \cite{neeman2023concentration}, we have 
\bas{
    \left\|\prod_{j=1}^{n}\exp\bb{\frac{\theta}{2}\bb{A\bb{s_{j}}-\Sigma}}\right\|_{F}^{2} &= \Tr\bb{\prod_{j=1}^{n}\exp\bb{\frac{\theta}{2}\bb{A\bb{s_{j}}-\Sigma}}\prod_{j=n}^{1}\exp\bb{\frac{\theta}{2}\bb{A\bb{s_{j}}-\Sigma}}} \\ 
    &= \vect\bb{I_{d}}^{T}\bb{\prod_{j=1}^{n}\exp\bb{\theta H\bb{s_{j}}}}\vect\bb{I_{d}}
}
where $H\bb{s_{j}} := \frac{1}{2}\bbb{\bb{A\bb{s_{j}}-\Sigma}\otimes I_{d} +  I_{d}\otimes\bb{A\bb{s_{j}}-\Sigma}}$. Noting that conditioned on the state sequence, the matrices $A\bb{s_{i}}, i \in [n]$ are independent under our model, we can push in the expectation over the state-specific distributions inside. Let $\E_{\pi}$ denote the expectation over the stationary state-sequence of the Markov chain, and $\E_{D}$ denote the distribution over states. Therefore,
\bas{
    \E_{\pi}\E_{D}\bbb{\left\|\prod_{j=1}^{n}\exp\bb{\frac{\theta}{2}\bb{A\bb{s_{j}}-\Sigma}}\right\|_{F}^{2}} &= \E_{\pi}\bbb{\vect\bb{I_{d}}^{T}\bb{\prod_{j=1}^{n}\E_{D\bb{s_j}}\bbb{\exp\bb{\theta H\bb{s_{j}}}}}\vect\bb{I_{d}}}
}
Defining the multiplication operator $\bb{E_{j}^{\theta}\mathbf{h}}\bb{x} = \E_{D\bb{x}}\bbb{\exp\bb{\theta H_{j}\bb{x}}}\mathbf{h}\bb{x}$ for any vector-valued function $\mathbf{h}$, we note that Eq $\bb{8}$ from \cite{neeman2023concentration} holds for our case as well.
\\ 
\\ 
Next, we adapt Proposition 5.3 from \cite{neeman2023concentration} for our setting. Specifically, we have the following lemma -
\begin{lemma}
    \label{lemma:prop53}
    Consider the operator $H\bb{x} := \frac{1}{2}\bbb{\bb{A\bb{x}-\Sigma}\otimes I_{d} +  I_{d}\otimes\bb{A\bb{x}-\Sigma}}$. Then, under assumptions \ref{assumption:norm_bound} and \ref{assumption:variance_bound} and the definition of $\Sigma$, we have,
    \begin{enumerate}
        \item $\E_{\pi}\E_{D\bb{x}}\bbb{H\bb{x}} = 0$
        \item $H\bb{x} \preceq \mathcal{M}I$
        \item $\left\|\E_{\pi}\E_{D\bb{x}}\bbb{H\bb{x}^{2}}\right\|_{2} \leq \mathcal{V}$
    \end{enumerate}
\end{lemma}
\begin{proof}
    The proof follows by using the same arguments as Proposition 5.3 from \cite{neeman2023concentration} and using the expectation $\E_{\pi}\E_{D\bb{x}}$ over both the state sequence and the distribution over states, along with assumptions \ref{assumption:norm_bound} and \ref{assumption:variance_bound}.
\end{proof}
\noindent
Finally, to prove Bernstein's inequality, we prove that Lemma 6.7 from \cite{neeman2023concentration} holds for our case. To note this, we start with equation (57) in their work. We have, using Lemma \ref{lemma:prop53}, 
\bas{
    \left|\left\langle v_{2}, \E_{\pi}\E_{D\bb{x}}\bbb{\exp\bb{\theta H\bb{x}}} v_{1} \right\rangle\right| &= \left|\left\langle v_{2}, \E_{\pi}\E_{D\bb{x}}\bbb{\exp\bb{\theta H\bb{x}}} v_{1} \right\rangle\right| \\ 
    &= \left|\left\langle v_{2}, \bb{I + \E_{\pi}\E_{D\bb{x}}\bbb{ H\bb{x}} + \sum_{k=2}^{\infty}\frac{\theta^{k}}{k!}\E_{\pi}\E_{D\bb{x}}\bbb{H\bb{x}^{k}}} v_{1} \right\rangle\right| \\ 
    &= \left|\left\langle v_{2}, v_{1} \right\rangle + \left\langle 
v_2\bb{\sum_{k=2}^{\infty}\frac{\theta^{k}}{k!}\E_{\pi}\E_{D\bb{x}}\bbb{H\bb{x}^{k}}} v_{1} \right\rangle\right| \\ 
&\leq \left|\left\langle v_{2}, v_{1} \right\rangle\right| \bb{1 + \mathcal{V}\bb{\sum_{k=2}^{\infty}\frac{\theta^{k}}{k!}\mathcal{M}^{k-2}}}
}
Therefore, Eq $(60)$ from \cite{neeman2023concentration} follows. The other bounds in the proof of Lemma 6.7 from \cite{neeman2023concentration} follow similarly. Therefore, we have the following version of Theorem 2.2 from \cite{neeman2023concentration} -
\begin{proposition}
    \label{proposition:tail_bound_Markov_sum}
    Under assumptions \ref{assumption:variance_bound} and \ref{assumption:norm_bound}, we have
    \bas{
        P\bb{\left\|\frac{1}{n}\sum_{j=1}^{n}A\bb{s_{j}} - \Sigma\right\|_{2} \geq t} \leq d^{2-\frac{\pi}{4}}\exp\bb{\frac{t^{2}/\frac{32}{\pi^{2}}}{\frac{1+|\lambda_{2}\bb{P}|}{1-|\lambda_{2}\bb{P}|}n\mathcal{V} + \frac{8/\pi}{1-|\lambda_{2}\bb{P}|}\mathcal{M}t}
        }
    }
\end{proposition}
The proof of Proposition \ref{theorem:offline_sample_complexity} now follows by converting the tail bound into a high probability bound and using Wedin's theorem \cite{wedin1972perturbation}. See proof of Theorem 1.1 in \cite{jain2016streaming} for details.

\section{Useful Results}
\label{appendix:useful_results}

This section presents some useful lemmas and their proofs that are subsequently used in our proofs.

\begin{lemma}
    (Reverse mixing) Consider a reversible, irreducible, and aperiodic Markov chain started from the stationary distribution. Then, 
    \bas{
        \frac{1}{2}\sup_{t \in \Omega}\sum_{s}\left|\mathbb{P}\bb{Z_{t}=s | Z_{t+k} = t} - \pi\bb{s}\right| = d_{\text{mix}}\bb{k}
    }
\end{lemma}

\begin{proof}
Let the transition probabilities of the Markov chain be represented as $P(x|y):= P(Z_{t+1}=x|Z_{t}=y)$. Consider the time-reversed chain $Y_{i} := Z_{n-i+1}$ for $i = 1, 2, \dots n$. Then, 
    \bas{
        & \mathbb{P}\bb{Y_{l}=s_{l} | Y_{l-1}=s_{l-1}, Y_{l-2}=s_{l-2} \dots Y_{1}=s_{1}} \\
        & \;\;\;\; = \mathbb{P}\bb{Z_{n-l+1}=s_{l} | Z_{n-l+2}=s_{l-1}, Z_{n-l+3}=s_{l-2}, \dots Z_{n}=s_{1}} \\
        & \;\;\;\; = \mathbb{P}\bb{Z_{n-l+1}=s_{l} | Z_{n-l+2}=s_{l-1}} \;\; \text{ using Lemma \ref{lemma:reverse_Markov}} \\
        & \;\;\;\; = \frac{\mathbb{P}\bb{Z_{n-l+1}=s_{l}, Z_{n-l+2}=s_{l-1}}}{\mathbb{P}\bb{Z_{n-l+2}=s_{l-1}}} \\
        & \;\;\;\; = \frac{\pi\bb{s_{l}}P(s_{l-1}|s_{l})}{\pi\bb{s_{l-1}}} \\
        & \;\;\;\; = P(s_{l}|s_{l-1}) \;\; \text{ using reversibility }
    }
    This proves that $Y_{n}$ is an irreducible Markov chain with the same transition probabilities as the original Markov chain. The irreducibility of $Y_{n}$ follows from the original Markov chain being irreducible. Therefore, 
    \ba{
        \mathbb{P}\bb{Z_{t} = s_{1} | Z_{t+k} = s_{2}} = \mathbb{P}\bb{Y_{n+1-t} = s_{1} | Y_{n+1-t-k} = s_{2}} \label{eq:reversible_Markov_chain_k_step_transition}
    }
    Then,
    \bas{
    \frac{1}{2}\sup_{t \in \Omega}\sum_{s}\left|\mathbb{P}\bb{Z_{t}=s | Z_{t+k} = t} - \pi\bb{s}\right| = \frac{1}{2}\sup_{t \in \Omega}\sum_{s}\left|\mathbb{P}\bb{Y_{n+1-t} = s | Y_{n+1-t-k} = t} - \pi\bb{s}\right| = d_{\text{mix}}\bb{k}
    }
    where the last inequality follows from the forward mixing properties of the Markov chain.
\end{proof}

\begin{lemma}
\label{lemma:inner_product_monotonicity}
Let $C_{j,i}=\prod_{t=j}^{i}(I+Z_t)$ for $i \leq j \leq n$, where $Z_t \in \mathbb{R}^{d \times d}$ are symmetric PSD matrices. Let $U \in \mathbb{R}^{d \times d'}$. Then,
\bas{\Tr\bb{U^TC_{j,i+1}C_{j,i+1}^TU}
\leq \Tr\bb{U^TC_{j,i}C_{j,i}^TU}
}
\end{lemma}
\begin{proof}
\bas{
\Tr\bb{U^TC_{j,i}C_{j,i}^TU} &= \Tr\bb{U^TC_{j,i+1} (I+2Z_i+Z_i^2)C_{j,i+1}^TU}\\
&=\Tr\bb{U^{T}C_{j,i+1}C_{j,i+1}^TU}+\Tr\bb{U^{T}C_{j,i+1}(2Z_i+Z_i^2)C_{j,i+1}^TU}
}
Since $Z_i$ and $Z_i^2$ are both PSD, the second term on the RHS is always positive. This yields the proof.
\end{proof}

\begin{lemma}
\label{lemma:inner_product_monotonicity_identity}
Let $B_t=\prod_{i=t}^1(I+Z_i)$, where $Z_i \in \mathbb{R}^{d \times d}$ are symmetric PSD matrices.
\bas{\Tr\bb{B_{n-1} B_{n-1}^T}
\leq \Tr\bb{B_{n} B_{n}^T}
}
\end{lemma}
\begin{proof}
\bas{
\Tr\bb{B_{n} B_{n}^T}&=\Tr\bb{(I+Z_n)B_{n-1} B_{n-1}^T(I+Z_n)}\\
&=\Tr\bb{B_{n-1} B_{n-1}^T}+\Tr\bb{Z_nB_{n-1}B_{n-1}^T}+\Tr\bb{B_{n-1}B_{n-1}^TZ_n}+\Tr\bb{Z_nB_{n-1}B_{n-1}^TZ_n}\\
&=\Tr\bb{B_{n-1} B_{n-1}^T}+2\Tr\bb{ B_{n-1}^TZ_nB_{n-1}}+\Tr\bb{ B_{n-1}^TZ_n^2B_{n-1}}
}
Since $Z_n$ and $Z_n^2$ are both PSD, the last two terms on the RHS are always positive. This yields the proof.
\end{proof}



\begin{lemma}
    \label{lemma:trace_ineq}
    Consider matrices $X \in \mathbb{R}^{d \times d'}$ and $A \in \mathbb{R}^{d \times d}$. Then, 
        \bas {
            \left|\Tr\bb{X^{T}AX}\right| \leq \|A\|_{2}\Tr\bb{X^{T}X}
        }
\end{lemma}
\begin{proof}
    For a matrix $Z \in \mathbb{R}^{d \times d}$, let the singular values be denoted as : 
    \bas {
        \sigma_{max}\bb{Z} = \sigma_{1}\bb{Z} \geq \sigma_{2}\bb{Z} \dots \geq \sigma_{d}\bb{Z}
    }
    Using Von-Neumann's trace inequality, we have
    \bas{
         \left|\Tr\bb{X^{T}AX}\right| &= \left|\Tr\bb{AXX^{T}}\right| \\
                                      &\leq \sum_{i=1}^{d}\sigma_{i}\bb{A}\sigma_{i}\bb{XX^{T}} \\
                                      &\leq \sigma_{max}\bb{A}\sum_{i=1}^{d}\sigma_{i}\bb{XX^{T}} \\
                                      &= \|A\|_{2}\Tr\bb{XX^{T}} \\
                                      &= \|A\|_{2}\Tr\bb{X^{T}X}
    }
\end{proof}

\begin{lemma}
    \label{lemma:reverse_Markov}
    Given the Markov property in a Markov chain, the reverse Markov property holds, i.e
    \bas{
        P\bb{Z_{t}=s | Z_{t+1} = w, Z_{t+2} = s_{t+2} \dots Z_{n} = s_{n}} = P\bb{Z_{t} = s | Z_{t+1} = w}
    }
\end{lemma}

\begin{proof}
    \bas{
    & P\bb{Z_{t}=s | Z_{t+1} = w, Z_{t+2} = s_{t+2} \dots Z_{n} = s_{n}} \\ 
    & \;\;\;\;\;\; = \frac{ P\bb{Z_{t}=s, Z_{t+1} = w, Z_{t+2} = s_{t+2} \dots Z_{n} = s_{n}}}{ P\bb{Z_{t+1} = t, Z_{t+2} = s_{t+2} \dots Z_{n} = s_{n}}} \\
    & \;\;\;\;\;\; = \frac{ P\bb{Z_{t}=s, Z_{t+1}=w}P\bb{Z_{t+2} = s_{t+2} \dots Z_{n} = s_{n}|Z_{t}=s, Z_{t+1}=w} } { P\bb{Z_{t+1} = w}P\bb{Z_{t+2} = s_{t+2} \dots Z_{n} = s_{n}|Z_{t+1} = w} } \\
    & \;\;\;\;\;\; = \frac{ P\bb{Z_{t}=s, Z_{t+1}=w}P\bb{Z_{t+2} = s_{t+2} \dots Z_{n} = s_{n}|Z_{t+1}=w} } { P\bb{Z_{t+1} = w}P\bb{Z_{t+2} = s_{t+2} \dots Z_{n} = s_{n}|Z_{t+1} = w} } \\
    & \;\;\;\;\;\; = \frac{ P\bb{Z_{t}=s, Z_{t+1}=w} } { P\bb{Z_{t+1} = w} } \\
    & \;\;\;\;\;\; = P\bb{Z_{t} = s | Z_{t+1} = w}
    }
\end{proof}
\subsection{Proof of Lemma \ref{lemma:etakproduct}}
Now we are ready to provide a proof of Lemma~\ref{lemma:etakproduct}.
\begin{proof}[Proof of Lemma \ref{lemma:etakproduct}]
Without loss of generality, we prove the statement for $m=1$. For convenience of notation, we denote $k := k_{1}$. Note that, 
\begin{align*}
    \newB{1}{k} = \sum_{r=0}^{k}\sum_{\left(i_{1}, i_{2} \dots i_{r}\right) \in G_{r}}\prod_{j=1}^{r}\eta_{i_{j}}A(s_{i_{j}}), \; G_{r} = \left\{\left(i_{1}, \dots, i_{r}\right) \in \left\{1, \dots, N\right\}^{r} : i_{1} < \dots < i_{r} \right\}
\end{align*}
with the convention that $\prod_{\phi} = I$. Therefore, since $\eta_{i}$ forms a non-increasing sequence and $|G_{r}| = {\binom{k}{r}}$, we have,
\begin{align}
    \left\|\newB{1}{k} - I\right\|_{2} &= \left\|\sum_{r=1}^{k}\sum_{\left(i_{1}, i_{2} \dots i_{r}\right) \in G_{r}}\prod_{j=1}^{r}\eta_{i_{j}}A(s_{i_{j}})\right\|_{2} \notag \\
    &\leq  \sum_{r=1}^{k}\sum_{\left(i_{1}, i_{2} \dots i_{r}\right) \in G_{r}}\left\|\prod_{j=1}^{r}\eta_{i_{j}}A(s_{i_{j}})\right\|_{2} \notag \\
    &\leq \sum_{r=1}^{k}{\binom{k}{r}}\left(\prod_{i=1}^{r}\eta_{i}\right)\left(\mathcal{M}+\lambda_{1}\right)^{r} \notag \\ 
    &\leq \sum_{r=1}^{k}\frac{k^{r}}{r!}\left(\prod_{i=1}^{r}\eta_{i}\right)\left(\mathcal{M}+\lambda_{1}\right)^{r} \notag \\
    & \leq \sum_{r=1}^{k}\frac{k^{r}}{r!}\eta_{1}^{r}\left(\mathcal{M}+\lambda_{1}\right)^{r} \notag \\ 
    & \leq \exp\bb{k\eta_{1}\bb{\mathcal{M}+\lambda_{1}}} - 1 \notag \\
    & \leq k\eta_{1}\bb{\mathcal{M}+\lambda_{1}}\bb{1+ k\eta_{1}\bb{\mathcal{M}+\lambda_{1}}} \text{ using } \ref{eq:exp_inequality} \notag \\ 
    & \leq \bb{1+\epsilon}k\eta_{1}\bb{\mathcal{M}+\lambda_{1}} \label{eq:first_order_approx}
\end{align}
where we have used the assumptions that $\|A(s)\|_{2} \leq \|A(s)-\Sigma\| + \|\Sigma\|_{2} = \left(\mathcal{M}+\lambda_{1}\right)$,  $k\eta_{1}\bb{\mathcal{M}+\lambda_{1}} < 1$ and the useful result that 
\ba{
e^{x} \leq 1 + x + x^{2}, x \in [0,1.79] \label{eq:exp_inequality}
} This completes the proof for $(a)$.
\\
\\
For part $\bb{b}$, we have
\begin{align}
    \left\|\newB{1}{k} - I - \sum_{t=1}^{k}\eta_{t}A\bb{s_t}\right\|_{2} &= \left\|\sum_{r=2}^{k}\sum_{\left(i_{1}, i_{2} \dots i_{r}\right) \in G_{r}}\prod_{j=1}^{r}\eta_{i_{j}}A(s_{i_{j}})\right\|_{2} \notag \\
    &\leq  \sum_{{ r=2}}^{k}\sum_{\left(i_{1}, i_{2} \dots i_{r}\right) \in G_{r}}\left\|\prod_{j=2}^{r}\eta_{i_{j}}A(s_{i_{j}})\right\|_{2} \notag \\
    &\leq \sum_{{ r=2}}^{k}{\binom{k}{r}}\left(\prod_{i=2}^{r}\eta_{i}\right)\left(\mathcal{M}+\lambda_{1}\right)^{r} \notag \\
    &\leq \sum_{r=2}^{k}\frac{k^{r}}{r!}\left(\prod_{i=2}^{r}\eta_{i}\right)\left(\mathcal{M}+\lambda_{1}\right)^{r} \notag \\
    &\leq \sum_{r=2}^{k}\frac{k^{r}}{r!}\eta_{1}^{r}\left(\mathcal{M}+\lambda_{1}\right)^{r} \notag \\ 
    & \leq \exp\bb{k\eta_{1}\bb{\mathcal{M}+\lambda_{1}}} - 1 - k\eta_{1}\bb{\mathcal{M}+\lambda_{1}} \notag \\
    & \leq k^{2}\eta_{1}^{2}\bb{\mathcal{M}+\lambda_{1}}^{2} \text{ using } \ref{eq:exp_inequality} \text{ along with } k\eta_{1}\bb{\mathcal{M}+\lambda_{1}} < 1 \label{eq:second_order_approx}
\end{align}
which completes the proof.
\end{proof}
\subsection{Proof of Lemma~\ref{lemma:bound_covariance}}
Before proving Lemma~\ref{lemma:bound_covariance}, we will need the following lemma.
\begin{lemma}
\label{lemma:matrix_cauchy_schwarz}
For arbitrary matrices $M_{i} \in \mathbb{R}^{d \times d}, i \in \bbb{n}$ and $Q \in \mathbb{R}^{n \times n}$, we have
\bas{
    \left\|\sum_{x,y \in \bbb{n}}Q\bb{x,y}M_{x}M_{y}^{T}\right\|_{2} &\leq \left\|Q\right\|_{2}\left\|\sum_{x \in \bbb{n}}M_{x}M_{x}^{T}\right\|_{2}
}
where $\left\|.\right\|_{2}$ denotes the spectral norm.
\end{lemma}
\begin{proof}
    Define matrix $X \in \mathbb{R}^{d \times nd}$ as 
        $X := \begin{bmatrix} 
                 M_{1} & M_{2} & \dots & M_{n}
              \end{bmatrix}$. We note that 
    \bas{
        \left\|X\right\|_{2} &= \sqrt{\lambda_{max}\bb{XX^{T}}} \\
                             &= \sqrt{\lambda_{max}\bb{\sum_{x\in\bbb{n}}M_{x}M_{x}^{T}}} \\
                             &= \sqrt{\left\|\sum_{x\in\bbb{n}}M_{x}M_{x}^{T}\right\|_{2}} \text{ since } \sum_{x\in\bbb{n}}M_{x}M_{x}^{T} \text{ is a symmetric matrix}
    }
    Then, we have, 
    \bas {
        \sum_{x,y \in \bbb{n}}Q\bb{x,y}M_{x}M_{y}^{T} &= X\bb{Q\otimes I_{d\times d}}X^{T}, \text{ where } \otimes \text{denotes the kronecker product} \\
        &\leq \left\|X\right\|^{2}_{2}\left\|Q\otimes I_{d\times d}\right\|_{2} \text{ using submultiplicativity of the spectral norm}  \\
        &= \left\|X\right\|^{2}_{2}\left\|Q\right\|_{2} \text{ since } \left\|A \otimes B\right\|_{2} = \|A\|_{2}\|B\|_{2}
    }
    which completes our proof.
\end{proof}
\begin{proof}[\textbf{Proof of Lemma \ref{lemma:bound_covariance}}]
    We denote $k_{i} := k$ for convenience of notation. By using reversibility (see \ref{eq:reversible_Markov_chain_k_step_transition}), we know that the time-reversed process is also a Markov chain with the same transition probabilities. Then, for $i < j \leq i+k$ and any $m$, 
    \ba{\label{eq:Markov-shift}
    P(s_i=s,s_j=t|s_{i+k}=u) &= P(s_i=s|s_j=t)P(s_j=t|s_{i+k}=u) \notag\\ 
                           &\stackrel{(i)}{=} P^{j-i}(t,s)P^{i+k-j}(u,t) \notag\\ 
                           &= P(s_m=s|s_{m-j+i}=t)P(s_{m-j+i}=t|s_{m-k}=u) \notag\\ 
                           &= P(s_{m}=s, s_{m-j+i}=t | s_{m-k}=u)
    }
    Step (i) uses reversibility.
    Therefore, 
    \bas{
        \E\bbb{\bb{A\bb{s_{i}}-\Sigma}SA\bb{s_{j}}|s_{i+k},\dots s_{n}} &= 
        \sum_{s,t}\bb{\Sigma_s + \mu_s\mu_s^{T} -\Sigma}S\bb{\Sigma_t+\mu_t\mu_t^{T}}P(s_i=s,s_j=t|s_{i+k},\dots s_{n})\\
        \text{ using Lemma~\ref{lemma:reverse_Markov} } 
        &=\sum_{s,t}\bb{\Sigma_s + \mu_s\mu_s^{T} -\Sigma}S\bb{\Sigma_t + \mu_t\mu_t^{T}}P(s_i=s,s_j=t|s_{i+k})\\
        \text{ using Eq~\ref{eq:Markov-shift} }
        &= \sum_{s,t}\bb{\Sigma_s + \mu_s\mu_s^{T} -\Sigma}S\bb{\Sigma_t + \mu_t\mu_t^{T}}P(s_{m}=s, s_{m-j+i}=t | s_{m-k}=u)\\
        &=\E\bbb{\bb{A\bb{s_{m}}-\Sigma}SA\bb{s_{m-j+i}}|s_{m-k}} \\ 
        &= \E\bbb{\bb{A\bb{s_{j}}-\Sigma}SA\bb{s_{i}}|s_{j-k}} \text{ setting } m := j
    }
    Therefore, without loss of generality, we proceed with the second form. 
    \bas{
        & \left\|\E\bbb{\bb{A\bb{s_{j}}-\Sigma}SA\bb{s_{i}}|s_{j-k}=x_{0}}\right\|_{2} \\
        &\leq \underbrace{\left\|\E\bbb{\bb{A\bb{s_{j}}-\Sigma}S\Sigma|s_{j-k}=x_{0}}\right\|_{2}}_{T_1} + \underbrace{\left\|\E\bbb{\bb{A\bb{s_{j}}-\Sigma}S\bb{A\bb{s_{i}}-\Sigma}|s_{j-k}=x_{0}}\right\|_{2}}_{T_2}
    }
    \ba{
        T_1&:= \left\|\E\bbb{\bb{A\bb{s_{j}}-\Sigma}S\Sigma|s_{j-k}=x_{0}}\right\|_{2} \notag \\
        & = \left\|\E\bbb{\E_{D\bb{s_{j}}}\bbb{\bb{A\bb{s_{j}}-\Sigma}}|s_{j-k}=x_{0}}S\Sigma\right\|_{2} \notag\\
        & = \left\|\E\bbb{\bb{\Sigma_{s_{j}} + \mu_{s_{j}}\mu_{s_{j}}^{T} -\Sigma}|s_{j-k}=x_{0}}S\Sigma\right\|_{2} \notag\\
        &=  \left\|\sum_{s \in \Omega}P^{k}(s_{j-k},s)\bb{\Sigma_{s} + \mu_{s}\mu_{s}^{T} -\Sigma}S\Sigma\right\|_{2} \notag \\
        &\leq \left\|\sum_{s \in \Omega}\left(P^{k}(s_{j-k},s)-\pi\left(s\right)\right)\bb{\Sigma_{s} + \mu_{s}\mu_{s}^{T} -\Sigma} + \underbrace{\mathbb{E}_{\pi}\left[\left(\Sigma_{s} + \mu_{s}\mu_{s}^{T}-\Sigma\right)\right]}_{=0}\right\|_{2}\left\|S\right\|_{2}\left\|\Sigma\right\|_{2} \notag \\
        &= \lambda_{1}\left\|S\right\|_{2}\bb{\left\|\sum_{s \in \Omega}\left(P^{k}(s_{j-k},s)-\pi\left(s\right)\right)\left(\Sigma_{s} + \mu_{s}\mu_{s}^{T}-\Sigma\right)\right\|_{2}} \notag \\
        &\leq \lambda_{1}\left\|S\right\|_{2}\mathcal{M}\sum_{s \in \Omega}\bigg|P^{k}(s_{j-k},s)-\pi\left(s\right)\bigg| \notag \\
        &\leq 2\lambda_{1}\left\|S\right\|_{2}\mathcal{M}d_{\text{mix}}\bb{k_{i+1}} \notag \\
        &\leq 2\eta_{i}^{2}\mathcal{M}\lambda_{1}\left\|S\right\|_{2} \label{eq:cov_T1_bound}
    }
    \ba{
    T_2 &= \left\|\E\bbb{\bb{A\bb{s_{j}}-\Sigma}S\bb{A\bb{s_{i}}-\Sigma}|s_{j-k}=x_{0}}\right\|_{2} \notag \\
    &= \left\|\sum_{x,y \in \Omega}\mathbb{P}\bb{s_{j}=x,s_{i}=y|s_{j-k}=x_{0}}\E_{D\bb{x}}\bbb{A\bb{x}-\Sigma}S\E_{D\bb{y}}\bbb{A\bb{y}-\Sigma}\right\|_{2} \text{using independence of } \notag \\
    & \;\;\;\;\; D\bb{x} \text{ and } D\bb{y} \text{ conditioned on } x \text{ and } y \notag \\
    &= \left\|\sum_{x,y \in \Omega}\mathbb{P}\bb{s_{j}=x,s_{i}=y|s_{j-k}=x_{0}}\underbrace{\bb{\Sigma_{x} +\mu_{x}\mu_{x}^{T} -\Sigma}S^{\frac{1}{2}}}_{W_{x}}\underbrace{S^{\frac{1}{2}}\bb{\Sigma_{y} + \mu_{y}\mu_{y}^{T} -\Sigma}}_{W_{y}^{T}}\right\|_{2} \notag \\
    &= \left\|\sum_{x,y \in \Omega}\mathbb{P}\bb{s_{j}=x|s_{i}=y}\mathbb{P}\bb{s_{i}=y|s_{j-k}=x_{0}}W_{x}W_{y}^{T}\right\|_{2} \;\; \text{using the Markov property} \notag \\
    &= \left\|\sum_{x,y \in \Omega}P^{j-i}\bb{y,x}P^{i-j+k}\bb{x_{0},y}W_{x}W_{y}^{T}\right\|_{2} \notag \\
    &= \left\|\sum_{x,y \in \Omega}\bb{P^{j-i}\bb{y,x}-\pi\bb{x}}P^{i-j+k}\bb{x_{0},y}W_{x}W_{y}^{T} + \sum_{x,y \in \Omega}\pi\bb{x}P^{i-j+k}\bb{x_{0},y}W_{x}W_{y}^{T}\right\|_{2} \notag \\
    &= \left\|\sum_{x,y \in \Omega}\bb{P^{j-i}\bb{y,x}-\pi\bb{x}}P^{i-j+k}\bb{x_{0},y}W_{x}W_{y}^{T} + \underbrace{\sum_{x \in \Omega}\pi\bb{x}W_{x}}_{=0}\sum_{y \in \Omega}P^{i-j+k}\bb{x_{0},y}W_{y}^{T}\right\|_{2} \notag }
    \ba{
    &= \left\|\sum_{x,y \in \Omega}\bb{P^{j-i}\bb{y,x}-\pi\bb{x}}P^{i-j+k}\bb{x_{0},y}W_{x}W_{y}^{T}\right\|_{2} \notag \\ 
    &\leq \underbrace{\left\|\sum_{x,y \in \Omega}\bb{P^{j-i}\bb{y,x}-\pi\bb{x}}\bb{P^{j-i+k}\bb{x_{0},y}-\pi\bb{y}}W_{x}W_{y}^{T}\right\|_{2}}_{T_{21}} + \underbrace{\left\|\sum_{x,y \in \Omega}\bb{P^{j-i}\bb{y,x}-\pi\bb{x}}\pi\bb{y}W_{x}W_{y}^{T}\right\|_{2}}_{T_{22}} \label{eq:cov_T2_bound}
    }
For $T_{21}$, we have, 
\ba{
    T_{21} &\leq \sum_{x,y \in \Omega}\left|P^{j-i}\bb{y,x}-\pi\bb{x}\right|\left|P^{i-j+k}\bb{x_{0},y}-\pi\bb{y}\right|\left\|W_{x}W_{y}^{T}\right\|_{2} \notag \\
            &\leq \left\|S\right\|_{2}\mathcal{M}^{2}\sum_{y \in \Omega}\left|P^{i-j+k}\bb{x_{0},y}-\pi\bb{y}\right|\sum_{x \in \Omega}\left|P^{j-i}\bb{y,x}-\pi\bb{x}\right| \notag \\
            &\leq 2\left\|S\right\|_{2}\mathcal{M}^{2}d_{\text{mix}}\bb{j-i}\sum_{y \in \Omega}\left|P^{i-j+k}\bb{x_{0},y}-\pi\bb{y}\right| \notag \\
            &\leq 4\left\|S\right\|_{2}\mathcal{M}^{2}d_{\text{mix}}\bb{j-i}d_{\text{mix}}\bb{i-j+k} \notag \\
            &\leq 4\left\|S\right\|_{2}\mathcal{M}^{2}2^{-\left\lfloor\frac{j-i}{\tau_{\text{mix}}}\right\rfloor}2^{-\left\lfloor\frac{i-j+k}{\tau_{\text{mix}}}\right\rfloor} \notag \\
            &\leq 8\left\|S\right\|_{2}\mathcal{M}^{2}2^{-\left\lfloor\frac{j-i+i-j+k}{\tau_{\text{mix}}}\right\rfloor} \text{ since } \forall a,b \;\; \lfloor a \rfloor + \lfloor b \rfloor \geq \lfloor a+b \rfloor - 1 \notag \\ 
            &\leq 8\left\|S\right\|_{2}\mathcal{M}^{2}2^{-\left\lfloor\frac{k}{\tau_{\text{mix}}}\right\rfloor} \leq  8\left\|S\right\|_{2}\mathcal{M}^{2}d_{\text{mix}}\bb{k} \leq 8\eta_{i}^{2}\mathcal{M}^{2}\left\|S\right\|_{2} \label{eq:cov_T21_bound} 
}
For $T_{22}$, we have,
\ba{
    T_{22} &= \left\|\sum_{x,y \in \Omega}\bb{P^{j-i}\bb{y,x}-\pi\bb{x}}\pi\bb{y}W_{x}W_{y}^{T}\right\|_{2} \notag \\
      &= \left\|\sum_{x,y \in \Omega}\frac{\bb{P^{j-i}\bb{y,x}-\pi\bb{x}}}{\sqrt{\pi\bb{x}}}\sqrt{\pi\bb{y}}\bb{\sqrt{\pi\bb{x}}W_{x}}\bb{\sqrt{\pi\bb{y}}W_{y}^{T}}\right\|_{2} \notag \\
      &= \left\|\sum_{x,y \in \Omega}\frac{\bb{P^{j-i}\bb{y,x}-\pi\bb{x}}}{\sqrt{\pi\bb{x}}}\sqrt{\pi\bb{y}}\bb{\sqrt{\pi\bb{x}}\bb{\Sigma_{x}+\mu_{x}\mu_{x}^{T}-\Sigma}S^{\frac{1}{2}}}\bb{\sqrt{\pi\bb{y}}S^{\frac{1}{2}}\bb{\Sigma_{y}+\mu_{y}\mu_{y}^{T}-\Sigma}}\right\|_{2} \notag \\
      &\stackrel{(i)} {\leq}\left\|Q\right\|_{2}\left\|\sum_{x \in \Omega}\pi\bb{x}\bb{\Sigma_{x}+\mu_{x}\mu_{x}^{T}-\Sigma}S\bb{\Sigma_{x}+\mu_{x}\mu_{x}^{T}-\Sigma}\right\|_{2} \notag  \\
      &= \left\|Q\right\|_{2}\left\|\mathbb{E}_{\pi}\bbb{\bb{\Sigma_{x}+\mu_{x}\mu_{x}^{T}-\Sigma}S\bb{\Sigma_{x}+\mu_{x}\mu_{x}^{T}-\Sigma}}\right\|_{2} \notag  \\
      &\leq \left\|Q\right\|_{2}\left\|S\right\|_{2}\left\|\mathbb{E}_{\pi}\bbb{\bb{\Sigma_{x}+\mu_{x}\mu_{x}^{T}-\Sigma}^{2}}\right\|_{2} \notag  \\
      &\leq \mathcal{V}\left\|Q\right\|_{2}\left\|S\right\|_{2} \label{eq:cov_T22_bound}
}
Step $(i)$ uses Lemma~\ref{lemma:matrix_cauchy_schwarz} with  $Q(y,x) := \frac{\bb{P^{i-j}\bb{y,x}-\pi\bb{x}}}{\sqrt{\pi\bb{x}}}\sqrt{\pi\bb{y}}$ and $M_x=\sqrt{\pi\bb{x}}\bb{\Sigma_{x}+\mu_{x}\mu_{x}^{T}-\Sigma}S^{\frac{1}{2}}$. Let's now bound $\left\|Q\right\|_{2}$. Let $\Pi := \diag\bb{\pi} \in \mathbb{R}^{\Omega \times \Omega}$ and $t := j-i$. Then, 
we have
\bas{
    Q &= \Pi^{\frac{1}{2}}\bb{P^{t} - \mathbbm{1}\mathbbm{1}^{T}\Pi}\Pi^{-\frac{1}{2}} \\
      &= \Pi^{\frac{1}{2}}P^{t}\Pi^{-\frac{1}{2}} - \Pi^{\frac{1}{2}} \mathbbm{1}\mathbbm{1}^{T}\Pi^{\frac{1}{2}}
}
Now, since we have a reversible Markov chain, $\Pi P = P^{T}\Pi$. Therefore, 
\bas {
    \Pi^{\frac{1}{2}}P\Pi^{-\frac{1}{2}} &= \Pi^{\frac{1}{2}}\Pi^{-1}P^{T}\Pi\Pi^{-\frac{1}{2}} \\
    &= \Pi^{-\frac{1}{2}}P^{T}\Pi^{\frac{1}{2}}
}
Therefore, $P$ is similar to the self-adjoint matrix $\Pi^{\frac{1}{2}}P\Pi^{-\frac{1}{2}}$ and their eigenvalues are real and the same. Further note that $\Pi^{\frac{1}{2}}\mathbbm{1}$ is the leading eigenvector of $\Pi^{\frac{1}{2}}P\Pi^{-\frac{1}{2}}$ with eigenvalue 1 since 
\bas {
    \Pi^{\frac{1}{2}}P\Pi^{-\frac{1}{2}}\Pi^{\frac{1}{2}}\mathbbm{1} &= \Pi^{\frac{1}{2}}P\mathbbm{1} \\
    &= \Pi^{\frac{1}{2}}\mathbbm{1} \text{ since } P \text{ is a stochastic matrix}
}
Now, 
\bas{
    \|Q\|_{2} &= \left\|\Pi^{\frac{1}{2}}P^{t}\Pi^{-\frac{1}{2}} - \Pi^{\frac{1}{2}} \mathbbm{1}\mathbbm{1}^{T}\Pi^{\frac{1}{2}}\right\|_{2} \\
    &= \left\|\bb{\Pi^{\frac{1}{2}}P\Pi^{-\frac{1}{2}}}^{t} - \Pi^{\frac{1}{2}} \mathbbm{1}\mathbbm{1}^{T}\Pi^{\frac{1}{2}}\right\|_{2} \\
    &\leq \left|\lambda_{2}\bb{\Pi^{\frac{1}{2}}P\Pi^{-\frac{1}{2}}}\right|^{t} \\
    &= \left|\lambda_{2}\bb{P}\right|^{t}
}
where $|\lambda_{2}\bb{.}|$ denotes the second-largest eigenvalue in magnitude. Therefore, using \ref{eq:cov_T1_bound}, \ref{eq:cov_T21_bound} and \ref{eq:cov_T22_bound}, we have
\bas{
    \E\bbb{\bb{A\bb{s_{i}}-\Sigma}SA\bb{s_{j}}|s_{i+k},...s_{n}} &\leq \bb{\left|\lambda_{2}\bb{P}\right|^{j-i}\mathcal{V} + 8\eta_{i}^{2}\mathcal{M}^{2} + 2\eta_{i}^{2}\mathcal{M}\lambda_{1}}\left\|S\right\|_{2} \\
    &\leq 
    \bb{\left|\lambda_{2}\bb{P}\right|^{j-i}\mathcal{V} + 8\eta_{i}^{2}\mathcal{M}\bb{\mathcal{M}+\lambda_{1}}}\left\|S\right\|_{2}
}
Hence proved.
\end{proof}
    
\begin{lemma}
\label{lemma:trace_linear_bound}
    Let $\forall i \in [n], \eta_{i}k_{i}\bb{\mathcal{M}+\lambda_{1}} \leq \epsilon, \epsilon \in \bb{0,1}$ and $\eta_{i}$ forms a non-increasing sequence. Set $k_{i} := \tau_{\text{mix}}\bb{\gamma\eta_{i}^{2}}, \gamma \in (0,1]$. Then for constant matrix $U \in \mathbb{R}^{d \times d'}$, and constant positive semi-definite matrix $G \in \mathbb{R}^{d \times d}$, $i \leq j \leq n, \; j-i \geq k_{i}$, we have 
    \begin{align*}
        & \left|\mathbb{E}\left[\Tr\bb{U^{T}\newB{i+1}{j}G\left(A_{i}-\Sigma\right)\newB{i+1}{j}^{T}U}\right]\right| \\ 
        & \;\;\;\;\;\;\;\;\;\;  \leq \eta_{i+1}\|G\|_{2}\bb{\frac{2\mathcal{V}\left|\lambda_{2}\bb{P}\right|}{1 - \left|\lambda_{2}\bb{P}\right|} + \eta_{i+1}\mathcal{M}\bb{2\gamma\bb{1+8\epsilon} + \bb{2+\bb{1+\epsilon}^{2}}k_{i+1}^{2}\bb{\mathcal{M}+\lambda_{1}}^{2}}} \\ 
        & \;\;\;\;\;\;\;\;\;\;\;\;\;\;\;\;\;\;\;\; \times \mathbb{E}\left[\Tr\bb{U^{T}\newB{i+k_{i+1}}{j}\newB{i+k_{i+1}}{j}^{T}U} \right]
    \end{align*}
    where $\newB{i}{j}$ is defined in \ref{definition:Bji}.
\end{lemma}
\begin{proof}
For the convenience of notation, we denote $k_{i+1} := k$. Let $\newB{i+1}{j} = \newB{i+k}{j}\left(I + R\right)$, then 
\begin{align}
    & \mathbb{E}\left[\Tr\bb{U^{T}\newB{i+1}{j}G\left(A_{i}-\Sigma\right)\newB{i+1}{j}^{T}U}\right] = \notag \\
    & \;\;\;\;\; \mathbb{E}\left[\underbrace{\Tr\bb{U^{T}\newB{i+k}{j}G\left(A_{i}-\Sigma\right)\newB{i+k}{j}^{T}U}}_{T_{1}}\right] + \mathbb{E}\left[\underbrace{\Tr\bb{U^{T}\newB{i+k}{j}G\left(A_{i}-\Sigma\right)R^{T}\newB{i+k}{j}^{T}U}}_{T_{2}}\right] + \notag \\
    & \;\;\;\;\; \mathbb{E}\left[\underbrace{\Tr\bb{U^{T}\newB{i+k}{j}RG\left(A_{i}-\Sigma\right)\newB{i+k}{j}^{T}U}}_{T_{3}}\right] + \mathbb{E}\left[\underbrace{\Tr\bb{U^{T}\newB{i+k}{j}RG\left(A_{i}-\Sigma\right)R^{T}\newB{i+k}{j}^{T}U}}_{T_{4}}\right] \label{eq:decompose_error}
\end{align}
We will now bound each of the terms $\mathbb{E}\left[T_{1}\right], \mathbb{E}\left[T_{2}\right], \mathbb{E}\left[T_{3}\right]$ and $\mathbb{E}\left[T_{4}\right]$. 
\begin{align*}
    \mathbb{E}\left[T_{1}\right] &= \mathbb{E}\left[\Tr\bb{U^{T}\newB{i+k}{j}G\left(A_{i}-\Sigma\right)\newB{i+k}{j}^{T}U}\right] \\
    &= \mathbb{E}\left[\mathbb{E}\left[\Tr\bb{U^{T}\newB{i+k}{j}G\left(A_{i}-\Sigma\right)\newB{i+k}{j}^{T}U} \bigg| s_{i+k},  \dots s_{j-1}, s_{j} \right]\right] \\
    &= \mathbb{E}\left[\Tr\bb{U^{T}\newB{i+k}{j}G\; \mathbb{E}\left[\left(A_{i}-\Sigma\right)\bigg| s_{i+k},  \dots s_{j-1}, s_{j} \right]\newB{i+k}{j}^{T}U} \right] \\
    &= \mathbb{E}\left[\Tr\bb{U^{T}\newB{i+k}{j}G\; \mathbb{E}\left[\left(A_{i}-\Sigma\right)\bigg| s_{i+k} \right]\newB{i+k}{j}^{T}U} \right] \text{ using Lemma \ref{lemma:reverse_Markov} }
\end{align*}
Now, using Lemma \ref{lemma:reverse_mixing}, we have, 
\begin{align}
    \left\|\mathbb{E}\left[\left(A_{i}-\Sigma\right)\bigg| s_{i+k} \right]\right\|_{2} &=  \left\|\sum_{s \in \Omega}P^{k}(s_{i+k},s)\bb{A_{i}-\Sigma}\right\|_{2} \notag \\
    &= \left\|\sum_{s \in \Omega}\left(P^{k}(s_{i+k},s)-\pi\left(s\right)\right)\left(A_{i}-\Sigma\right) + \underbrace{\mathbb{E}_{\pi}\left[\left(A_{i}-\Sigma\right)\right]}_{=0}\right\|_{2} \notag \\
    &= \left\|\sum_{s \in \Omega}\left(P^{k}(s_{i+k},s)-\pi\left(s\right)\right)\left(A_{i}-\Sigma\right)\right\|_{2} \notag \\
    &\leq \mathcal{M}\sum_{s \in \Omega}\bigg|P^{k}(s_{i+k},s)-\pi\left(s\right)\bigg| \notag \\
    &\leq 2\mathcal{M}d_{\text{mix}}\bb{k_{i+1}} \notag \\
    &\leq 2\gamma\eta_{i+1}^{2}\mathcal{M} \label{eq:mixing_time_bound}
\end{align}
where we have used Lemma \ref{lemma:trace_ineq}. Therefore, 
\begin{align}
    \left|\mathbb{E}\left[T_{1}\right]\right| \leq \gamma\eta_{i+1}^{2}\mathcal{M}\|G\|_{2} \mathbb{E}\left[\Tr\bb{U^{T}\newB{i+k}{j}\newB{i+k}{j}^{T}U}\right] \label{eq:boundT1}
\end{align}
We will now bound $\mathbb{E}\left[T_2\right]$. Let $R_{0} := \sum_{\ell=i+1}^{i+k-1}\eta_{\ell}A_{\ell}$. Using Lemma \ref{lemma:etakproduct} we have 
\bas{\|R-R_{0}\|_{2} \; \leq \eta_{i+1}^2k_{i+1}^2\bb{\mathcal{M}+\lambda_{1}}^{2}} Then, 
\bas{
\E[T_2] &= \E\bbb{\Tr\bb{U^{T}\newB{i+k}{j}G\left(A_{i}-\Sigma\right)R^{T}\newB{i+k}{j}^{T}U}} \\
&= \E\bbb{\Tr\bb{U^{T}\newB{i+k}{j}G\left(A_{i}-\Sigma\right)R_{0}^{T}\newB{i+k}{j}^{T}U}} + \E\bbb{\Tr\bb{U^{T}\newB{i+k}{j}G\left(A_{i}-\Sigma\right)\bb{R - R_{0}}^{T}\newB{i+k}{j}^{T}U}} \\
&= \E\bbb{\Tr\bb{U^{T}\newB{i+k}{j}G\E[\left(A_{i}-\Sigma\right)R_{0}^{T}|s_{i+k},  \dots s_{j-1}, s_{j}]\newB{i+k}{j}^{T}U}} + \\
& \;\;\;\;\;\;\;\;\;\;\;\;\;\;\;\;\;\;\;\;\;\;\;\;\;\;\;\;\;\;\;\;\;\;\;\;\;\;\;\;\;\;\;\;\;\;\;\; \E\bbb{\Tr\bb{U^{T}\newB{i+k}{j}G\left(A_{i}-\Sigma\right)\bb{R - R_{0}}^{T}\newB{i+k}{j}^{T}U}}
}
Using Lemma \ref{lemma:bound_covariance} with $S := I$ we have, 
\ba{
    \|\E\bbb{\left(A_{i}-\Sigma\right)R_{0}^{T}|s_{i+k},\dots,s_j}\|_{2} &\leq \sum_{\ell=i+1}^{i+k-1}\eta_{\ell}\bb{\left|\lambda_{2}\bb{P}\right|^{\ell-i}\mathcal{V} + 8\gamma\eta_{i+1}^{2}\mathcal{M}\bb{\mathcal{M}+\lambda_{1}}} \notag \\
    &\leq \eta_{i+1}\mathcal{V}\frac{\left|\lambda_{2}\bb{P}\right|}{1 - \left|\lambda_{2}\bb{P}\right|} + 8\gamma\eta_{i+1}^{3}k_{i+1}\mathcal{M}\bb{\mathcal{M}+\lambda_{1}} \label{eq:bound_linear_term}
}
Therefore, 
\ba{
    & |\E\bbb{T_{2}}| \notag \\
    &\leq \|G\|_{2}\bb{\eta_{i+1}\frac{\mathcal{V}\left|\lambda_{2}\bb{P}\right|}{1 - \left|\lambda_{2}\bb{P}\right|} + 8\gamma\eta_{i+1}^{3}k_{i+1}\mathcal{M}\bb{\mathcal{M}+\lambda_{1}} + \eta_{i+1}^2k_{i+1}^2\mathcal{M}\bb{\mathcal{M}+\lambda_{1}}^{2}}\mathbb{E}\left[ \Tr\bb{U^{T}\newB{i+k}{j}\newB{i+k}{j}^{T}U}\right] \notag \\
    &= \eta_{i+1}\|G\|_{2}\bb{\frac{\mathcal{V}\left|\lambda_{2}\bb{P}\right|}{1 - \left|\lambda_{2}\bb{P}\right|} + 8\gamma\eta_{i+1}^{2}k_{i+1}\mathcal{M}\bb{\mathcal{M}+\lambda_{1}} + \eta_{i+1}k_{i+1}^2\mathcal{M}\bb{\mathcal{M}+\lambda_{1}}^{2}}\mathbb{E}\left[\Tr\bb{U^{T}\newB{i+k}{j}\newB{i+k}{j}^{T}U}\right] \label{eq:boundT2}
}
Similarly using Lemma \ref{lemma:bound_covariance} with $S := G$, 
\begin{align}
    \left|\mathbb{E}\left[T_{3}\right]\right| &\leq \eta_{i+1}\|G\|_{2}\bb{\frac{\mathcal{V}\left|\lambda_{2}\bb{P}\right|}{1 - \left|\lambda_{2}\bb{P}\right|} + 8\gamma\eta_{i+1}^{2}k_{i+1}\mathcal{M}\bb{\mathcal{M}+\lambda_{1}} + \eta_{i+1}k_{i+1}^2\mathcal{M}\bb{\mathcal{M}+\lambda_{1}}^{2}}\mathbb{E}\left[\Tr\bb{U^{T}\newB{i+k}{j}\newB{i+k}{j}^{T}U}\right] \label{eq:boundT3}
\end{align}
Finally, 
\ba{
    \left|\mathbb{E}\left[T_{4}\right]\right| &\leq  \mathcal{M}\|G\|_{2}\|R\|_{2}^{2}\mathbb{E}\left[\Tr\bb{U^{T}\newB{i+k}{j}\newB{i+k}{j}^{T}U}\right] \notag \\
    &\leq \bb{1+\epsilon}^{2}\eta_{i+1}^{2}k_{i+1}^{2}\mathcal{M}\bb{\mathcal{M}+\lambda_{1}}^{2}\|G\|_{2}\mathbb{E}\left[\Tr\bb{U^{T}\newB{i+k}{j}\newB{i+k}{j}^{T}U}\right] \text{ using Lemma } \ref{lemma:etakproduct} \label{eq:boundT4}
}
Therefore, using Eqs~\ref{eq:boundT1}, \ref{eq:boundT2}, \ref{eq:boundT3}, \ref{eq:boundT4} along with \ref{eq:decompose_error}, we have
\bas {
    & \left|\mathbb{E}\left[\Tr\bb{U^{T}\newB{i+1}{j}G\bb{A_i-\Sigma}\newB{i+1}{j}^{T}U}\right]\right| \\ 
    & \leq \eta_{i+1}\|G\|_{2}\bb{\frac{2\mathcal{V}\left|\lambda_{2}\bb{P}\right|}{1 - \left|\lambda_{2}\bb{P}\right|} + \eta_{i+1}\mathcal{M}\bb{2\gamma +  16\gamma\eta_{i+1} k_{i+1}\bb{\mathcal{M}+\lambda_{1}} + \bb{2+\bb{1+\epsilon}^{2}}k_{i+1}^{2}\bb{\mathcal{M}+\lambda_{1}}^{2}}}\\ 
    & \;\;\;\;\;\;\;\;\;\;\;\;\;\;\;\;\;\;\;\;\;\;\;\;\;\;\;\;\;\;\;\;\;\;\;\;\;\;\;\;\;\;\;\;\;\;\;\;\;\;\;\;\;\;\;\;\;\;\;\;\;\;\;\;\;\;\;\;\;\;\;\;\;\;\;\;\;\;\;\;\;\;\;\;\;\;\;\;\;\;\;\;\;\;\;\;\;\;\;\;\;\;\;\;\;\;\;\;\;\;\;\;\times \mathbb{E}\left[\Tr\bb{U^{T}\newB{i+k}{j}\newB{i+k}{j}^{T}U} \right] \\ 
    & \leq \eta_{i+1}\|G\|_{2}\bb{\frac{2\mathcal{V}\left|\lambda_{2}\bb{P}\right|}{1 - \left|\lambda_{2}\bb{P}\right|} + \eta_{i+1}\mathcal{M}\bb{2\gamma\bb{1+8\epsilon} + \bb{2+\bb{1+\epsilon}^{2}}k_{i+1}^{2}\bb{\mathcal{M}+\lambda_{1}}^{2}}}\\ 
    & \;\;\;\;\;\;\;\;\;\;\;\;\;\;\;\;\;\;\;\;\;\;\;\;\;\;\;\;\;\;\;\;\;\;\;\;\;\;\;\;\;\;\;\;\;\;\;\;\;\;\;\;\;\;\;\;\;\;\;\;\;\;\;\;\;\;\;\;\;\;\;\;\;\;\;\;\;\;\;\;\;\;\;\;\;\;\;\;\;\;\;\;\;\;\;\;\;\;\;\;\;\;\;\;\;\;\;\;\;\;\;\;\times \mathbb{E}\left[\Tr\bb{U^{T}\newB{i+k}{j}\newB{i+k}{j}^{T}U} \right]
}
where in the last line we used $\eta_{i+1} k_{i+1}\bb{\mathcal{M}+\lambda_{1}} \leq \epsilon$.
Hence proved.
\end{proof}

\begin{lemma}
\label{lemma:trace_square_bound}
    Let $\forall i \in [n], \eta_{i}k_{i}\bb{\mathcal{M}+\lambda_{1}} \leq \epsilon, \epsilon \in \bb{0,1}$ and $\eta_{i}$ forms a non-increasing sequence. Set $k_{i} := \tau_{\text{mix}}\bb{\gamma\eta_{i}^{2}}, \gamma \in (0,1]$. Then for constant matrices $U \in \mathbb{R}^{d \times d'}$, $G \in \mathbb{R}^{d \times d}$, $i \leq j \leq n, \; j-i \geq k_{i}$, we have 
    \begin{align*}
        & \left|\mathbb{E}\left[\Tr\bb{U^{T}\newB{i+1}{j}G\left(A_{i}-\Sigma\right)^{2}\newB{i+1}{j}^{T}U}\right]\right| \\ 
        & \;\;\;\;\;\;\;\;\;\;\;\; \leq \bb{\mathcal{V} + \eta_{i+1}\mathcal{M}^{2}\bb{2\gamma\eta_{i+1} + \bb{1+\epsilon}\bb{2+\epsilon\bb{1+\epsilon}}k_{i+1}\bb{\mathcal{M}+\lambda_{1}}}}\left\|G\right\|_{2}\mathbb{E}\left[\Tr\bb{U^{T}\newB{i+k_{i+1}}{j}\newB{i+k_{i+1}}{j}^{T}U} \right]
    \end{align*}
    where $\newB{i}{j}$ is defined in \ref{definition:Bji}.
\end{lemma}
\begin{proof}
For convenience of notation, we denote $k_{i+1}:= k$. Let $\newB{i+1}{j} = \newB{i+k}{j}\left(I + R\right)$, then 
\begin{align*}
    & \mathbb{E}\left[\Tr\bb{U^{T}\newB{i+1}{j}G\left(A_{i}-\Sigma\right)^{2}\newB{i+1}{j}^{T}U}\right] = \\
    & \;\;\;\;\; \mathbb{E}\left[\underbrace{\Tr\bb{U^{T}\newB{i+k}{j}G\left(A_{i}-\Sigma\right)^{2}\newB{i+k}{j}^{T}U}}_{T_{1}}\right] + \mathbb{E}\left[\underbrace{\Tr\bb{U^{T}\newB{i+k}{j}G\left(A_{i}-\Sigma\right)^{2}R^{T}\newB{i+k}{j}^{T}U}}_{T_{2}}\right] + \\
    & \;\;\;\;\; \mathbb{E}\left[\underbrace{\Tr\bb{U^{T}\newB{i+k}{j}RG\left(A_{i}-\Sigma\right)^{2}\newB{i+k}{j}^{T}U}}_{T_{3}}\right] + \mathbb{E}\left[\underbrace{\Tr\bb{U^{T}\newB{i+k}{j}RG\left(A_{i}-\Sigma\right)^{2}R^{T}\newB{i+k}{j}^{T}U}}_{T_{4}}\right]
\end{align*}
We will now bound each of the terms $\mathbb{E}\left[T_{1}\right], \mathbb{E}\left[T_{2}\right], \mathbb{E}\left[T_{3}\right]$ and $\mathbb{E}\left[T_{4}\right]$. 
\\
\\ 
Since $\left\|\E_{\pi}\bbb{\bb{A_{t}-\Sigma}^{2}}\right\|_{2} \leq \mathcal{V}$, therefore
\begin{align*}
    \mathbb{E}\left[T_{1}\right] &= \mathbb{E}\left[\Tr\bb{U^{T}\newB{i+k}{j}G\left(A_{i}-\Sigma\right)^{2}\newB{i+k}{j}^{T}U}\right] \\
    &= \mathbb{E}\left[\mathbb{E}\left[\Tr\bb{U^{T}\newB{i+k}{j}G\left(A_{i}-\Sigma\right)^{2}\newB{i+k}{j}^{T}U} \bigg| s_{i+k},  \dots s_{j-1}, s_{j} \right]\right] \\
    &= \mathbb{E}\left[\Tr\bb{U^{T}\newB{i+k}{j}G\; \mathbb{E}\left[\left(A_{i}-\Sigma\right)^{2}\bigg| s_{i+k},  \dots s_{j-1}, s_{j} \right]\newB{i+k}{j}^{T}U} \right] \\
    &= \mathbb{E}\left[\Tr\bb{U^{T}\newB{i+k}{j}G\; \mathbb{E}\left[\left(A_{i}-\Sigma\right)^{2}\bigg| s_{i+k} \right]\newB{i+k}{j}^{T}U} \right] \text{ using Lemma \ref{lemma:reverse_Markov} } \\ 
    &\stackrel{(i)}\leq  \bb{\mathcal{V} + 2d_{\text{mix}}\bb{k}\mathcal{M}^{2}}\left\|G\right\|_{2}\mathbb{E}\left[\Tr\bb{U^{T}\newB{i+k}{j}\; \newB{i+k}{j}^{T}U} \right]
\end{align*}
where in $(i)$, we used similar steps as \ref{eq:mixing_time_bound} to get 
\ba{
    \left\|\mathbb{E}\left[\left(A_{i}-\Sigma\right)^{2}\bigg| s_{i+k} \right]\right\|_{2} \leq \left\|\mathbb{E}_{\pi}\left[\left(A_{i}-\Sigma\right)^{2}\right]\right\|_{2} + 2d_{\text{mix}}\bb{k}\mathcal{M}^{2} \label{eq:mixing_time_bound_square}
}
Next, using Lemma \ref{lemma:etakproduct} we have that 
\ba{\label{eq:r}
\left\|R\right\|_{2} \leq \bb{1+\epsilon}k_{i+1}\eta_{i+1}\bb{\mathcal{M}+\lambda_{1}}. 
}
Therefore, 
\bas{
    \E\bbb{T_{2}} &= \mathbb{E}\left[\Tr\bb{U^{T}\newB{i+k}{j}G\left(A_{i}-\Sigma\right)^{2}R^{T}\newB{i+k}{j}^{T}U}\right] \\ 
                 &\leq \bb{1+\epsilon}k_{i+1}\eta_{i+1}\mathcal{M}^{2}\bb{\mathcal{M}+\lambda_{1}}\left\|G\right\|_{2}\mathbb{E}\left[\Tr\bb{U^{T}\newB{i+k}{j}\newB{i+k}{j}^{T}U}\right]
}
Similarly, 
\bas{
    \E\bbb{T_{3}} &= \mathbb{E}\left[\Tr\bb{U^{T}\newB{i+k}{j}RG\left(A_{i}-\Sigma\right)^{2}\newB{i+k}{j}^{T}U}\right] \\ 
                 &\leq \bb{1+\epsilon}k_{i+1}\eta_{i+1}\mathcal{M}^{2}\bb{\mathcal{M}+\lambda_{1}}\left\|G\right\|_{2}\mathbb{E}\left[\Tr\bb{U^{T}\newB{i+k}{j}\newB{i+k}{j}^{T}U}\right]
}
Finally, using the bound on $\|R\|_2$ from Eq~\ref{eq:r}, we have:
\bas{
    \E\bbb{T_{4}} &= \mathbb{E}\left[\Tr\bb{U^{T}\newB{i+k}{j}RG\left(A_{i}-\Sigma\right)^{2}R^{T}\newB{i+k}{j}^{T}U}\right] \\  
    &\leq \bb{1+\epsilon}^{2}k_{i+1}^{2}\eta_{i+1}^{2}\mathcal{M}^{2}\bb{\mathcal{M}+\lambda_{1}}^{2}\left\|G\right\|_{2}\mathbb{E}\left[\Tr\bb{U^{T}\newB{i+k}{j}\newB{i+k}{j}^{T}U}\right] \\
    &\leq \epsilon\bb{1+\epsilon}^{2}k_{i+1}\eta_{i+1}\mathcal{M}^{2}\bb{\mathcal{M}+\lambda_{1}}\left\|G\right\|_{2}\mathbb{E}\left[\Tr\bb{U^{T}\newB{i+k}{j}\newB{i+k}{j}^{T}U}\right] \text{ using } \forall i, \; \eta_{i}k_{i}\bb{M+\lambda_{1}} \leq c
}
Therefore, 
\begin{align*}
        & \left|\mathbb{E}\left[\Tr\bb{U^{T}\newB{i+1}{j}G\left(A_{i}-\Sigma\right)^{2}\newB{i+1}{j}^{T}U}\right]\right| \\ 
        & \stackrel{(i)}\leq \bb{\mathcal{V} + \eta_{i+1}\bb{2\gamma\eta_{i+1}\mathcal{M}^{2} + \bb{1+\epsilon}\bb{2+\epsilon\bb{1+\epsilon}}k_{i+1}\mathcal{M}^{2}\bb{\mathcal{M}+\lambda_{1}}}}\left\|G\right\|_{2}\mathbb{E}\left[\Tr\bb{U^{T}\newB{i+k}{j}\newB{i+k}{j}^{T}U} \right]  \\ 
        &= \bb{\mathcal{V} + \eta_{i+1}\mathcal{M}^{2}\bb{2\gamma\eta_{i+1} + \bb{1+\epsilon}\bb{2+\epsilon\bb{1+\epsilon}}k_{i+1}\bb{\mathcal{M}+\lambda_{1}}}}\left\|G\right\|_{2}\mathbb{E}\left[\Tr\bb{U^{T}\newB{i+k}{j}\newB{i+k}{j}^{T}U} \right]
    \end{align*}
where in $\bb{i}$, we used $d_{\text{mix}}\bb{k} = d_{\text{mix}}\bb{k_{i+1}} \leq \gamma\eta_{i+1}^{2}$. Hence proved.
\end{proof}

\begin{lemma}
    \label{lemma:bound_Pt}
    Let $\forall i \in [n], \eta_{i}k_{i}\bb{\mathcal{M}+\lambda_{1}} \leq \epsilon, \epsilon \in \bb{0,1}$ and step-sizes $\eta_{i}$ forms a non-increasing sequence. Further, let the step-sizes follow a slow-decay property, i.e, $\forall i,\eta_{i} \leq \eta_{i-k_{i}} \leq 2\eta_{i}$. Set $k_{i} := \tau_{\text{mix}}\bb{\gamma\eta_{i}^{2}}, \gamma \in (0,1]$. Let $G \in \mathbb{R}^{d \times d}$ be a constant positive semi-definite matrix, and $P_{t} := \Tr\bb{\B_{t-1}\B_{t-1}^TG(A_t-\Sigma)}$, then, 
    \bas {
    & \E\bbb{P_{t}} \leq \eta_{t-k_{t}}\bb{\frac{2\mathcal{V}\left|\lambda_{2}\bb{P}\right|}{1 - \left|\lambda_{2}\bb{P}\right|} + \eta_{t-k_{t}}\mathcal{M}\bb{2\gamma\bb{1+8\epsilon} + \bb{2+\bb{1+\epsilon}^{2}}k_{t}^{2}\bb{\mathcal{M}+\lambda_{1}}^{2}}}\left\|G\right\|_{2}\mathbb{E}\left[\Tr\bb{\B_{t-k_{t}}\B_{t-k_{t}}^T}\right]
    }
    where $\B_{t}$ is defined in \ref{definition:Bt}.
\end{lemma}
\begin{proof} Let $\B_{t} = \bb{I + R}\B_{t-k_{t}}$ with $\|R\|_{2} \leq r$. Then,
\bas{
\mathbb{E}\left[P_{t}\right] &= 
\mathbb{E}\left[\underbrace{\Tr\bb{\B_{t-k_{t}}\B_{t-k_{t}}^TG(A_t-\Sigma)}}_{P_{t,1}}\right]+\mathbb{E}\left[\underbrace{\Tr\bb{\B_{t-k_{t}}\B_{t-k_{t}}^TR^{T}G(A_t-\Sigma)}}_{P_{t,2}}\right]\\
&\;\;\;\;\;\; +\mathbb{E}\left[\underbrace{\Tr\bb{\B_{t-k_{t}}\B_{t-k_{t}}^TG(A_t-\Sigma)R}}_{P_{t,3}}\right] + \mathbb{E}\left[\underbrace{\Tr\bb{\B_{t-k_{t}}\B_{t-k_{t}}^TR^{T}G(A_t-\Sigma)R}}_{P_{t,4}}\right]
}
Let's consider each of the terms above. Using Von-Neumann's trace inequality and \ref{eq:boundT2}, we have,
\bas {
    \mathbb{E}\left[P_{t,1}\right] &= \mathbb{E}\left[\Tr\bb{\B_{t-k_{t}}\B_{t-k_{t}}^T\E\bbb{G(A_t-\Sigma)|s_{1},s_{2},\dots,s_{t-k_{t}}}}\right] \\
    &\leq \mathbb{E}\left[\Tr\bb{\B_{t-k_{t}}\B_{t-k_{t}}^TG\E\bbb{(A_t-\Sigma)|s_{t-k_{t}}}}\right] \\
    &\leq \left\|G\E\bbb{(A_t-\Sigma)|s_{t-k_{t}}}\right\|_{2}\mathbb{E}\left[\Tr\bb{\B_{t-k_{t}}\B_{t-k_{t}}^T}\right] \\
    &\leq 2\mathcal{M}d_{\text{mix}}\bb{k_{t}}\left\|G\right\|_{2}\mathbb{E}\left[\Tr\bb{\B_{t-k_{t}}\B_{t-k_{t}}^T}\right] \text{ using } \ref{eq:mixing_time_bound} \\
    &\leq 2\gamma\eta_{t}^{2}\mathcal{M}\left\|G\right\|_{2}\mathbb{E}\left[\Tr\bb{\B_{t-k_{t}}\B_{t-k_{t}}^T}\right] 
    \\
    \\
    \E\bbb{P_{t,2}} &= \E\bbb{\Tr\bb{\B_{t-k_{t}}\B_{t-k_{t}}^T,\E\bbb{R^{T}G(A_t-\Sigma)U|s_{1},s_{2},\dots,s_{t-k_{t}}}}} \\
    &\leq \left\|\E\bbb{R^{T}G(A_t-\Sigma)|s_{1},s_{2},\dots,s_{t-k_{t}}}\right\|_{2}\mathbb{E}\left[\Tr\bb{\B_{t-k_{t}}\B_{t-k_{t}}^T}\right] \\
    &= \left\|\E\bbb{R^{T}G(A_t-\Sigma)|s_{t-k_{t}}}\right\|_{2}\mathbb{E}\left[\Tr\bb{\B_{t-k_{t}}\B_{t-k_{t}}^T}\right] \\
    &\leq \eta_{t-k_{t}}\left\|G\right\|_{2}\bb{\frac{\mathcal{V}\left|\lambda_{2}\bb{P}\right|}{1 - \left|\lambda_{2}\bb{P}\right|} + 8\gamma\eta_{t-k_{t}}^{2}k_{t}\mathcal{M}\bb{\mathcal{M}+\lambda_{1}} + \eta_{t-k_{t}}k_{t}^2\mathcal{M}\bb{\mathcal{M}+\lambda_{1}}^{2}}\mathbb{E}\left[\Tr\bb{\B_{t-k_{t}}\B_{t-k_{t}}^T}\right]
    \text{ using } \ref{eq:boundT2}
}
\bas{
    \E\bbb{P_{t,3}} &\leq \eta_{t-k_{t}}\left\|G\right\|_{2}\bb{\frac{\mathcal{V}\left|\lambda_{2}\bb{P}\right|}{1 - \left|\lambda_{2}\bb{P}\right|} + 8\gamma\eta_{t-k_{t}}^{2}k_{t}\mathcal{M}\bb{\mathcal{M}+\lambda_{1}} + \eta_{t-k_{t}}k_{t}^2\mathcal{M}\bb{\mathcal{M}+\lambda_{1}}^{2}}\mathbb{E}\left[\Tr\bb{\B_{t-k_{t}}\B_{t-k_{t}}^T}\right] \\ 
    & \;\;\;\;\; \text{ using similar steps as }  \E\bbb{P_{t,2}}
    \\
    \\
    \E\bbb{P_{t,4}}  &= \mathbb{E}\left[\Tr\bb{\B_{t-k_{t}}\B_{t-k_{t}}^TR^{T}G(A_t-\Sigma)R}\right] \\
                    &\leq r^{2}\mathcal{M}\left\|G\right\|_{2}\mathbb{E}\left[\Tr\bb{\B_{t-k_{t}}\B_{t-k_{t}}^T}\right] \\
                    &\leq \bb{1+\epsilon}^{2}\eta_{t-k_{t}+1}^{2}k_{t}^{2}\mathcal{M}\bb{\mathcal{M}+\lambda_{1}}^{2}\left\|G\right\|_{2}\mathbb{E}\left[\Tr\bb{\B_{t-k_{t}}\B_{t-k_{t}}^T}\right] \text{ using Lemma } \ref{lemma:etakproduct} \\ 
                    &\leq \bb{1+\epsilon}^{2}\eta_{t-k_{t}}^{2}k_{t}^{2}\mathcal{M}\bb{\mathcal{M}+\lambda_{1}}^{2}\left\|G\right\|_{2}\mathbb{E}\left[\Tr\bb{\B_{t-k_{t}}\B_{t-k_{t}}^T}\right]
}
Therefore we have,
\bas{
    & \E\bbb{P_{t}}  \\
    & \leq \eta_{t-k_{t}}\bb{\frac{2\mathcal{V}\left|\lambda_{2}\bb{P}\right|}{1 - \left|\lambda_{2}\bb{P}\right|} + \mathcal{M}\bb{2\gamma\eta_{t} +  16\gamma\eta_{t-k_{t}}^{2}k_{t}\bb{\mathcal{M}+\lambda_{1}} + \bb{2+\bb{1+\epsilon}^{2}}\eta_{t-k_{t}}k_{t}^{2}\bb{\mathcal{M}+\lambda_{1}}^{2}}}\left\|G\right\|_{2} \\ 
    & \qquad\qquad\qquad
    \times \mathbb{E}\left[\Tr\bb{\B_{t-k_{t}}\B_{t-k_{t}}^T}\right] \\ 
    & \stackrel{\bb{i}}\leq \eta_{t-k_{t}}\bb{\frac{2\mathcal{V}\left|\lambda_{2}\bb{P}\right|}{1 - \left|\lambda_{2}\bb{P}\right|} + \eta_{t-k_{t}}\mathcal{M}\bb{2\gamma +  16\gamma\eta_{t} k_{t}\bb{\mathcal{M}+\lambda_{1}} + \bb{2+\bb{1+\epsilon}^{2}}k_{t}^{2}\bb{\mathcal{M}+\lambda_{1}}^{2}}}\left\|G\right\|_{2} \\ 
    & \qquad\qquad\qquad
    \times \mathbb{E}\left[\Tr\bb{\B_{t-k_{t}}\B_{t-k_{t}}^T}\right] \\ 
    & \stackrel{\bb{ii}}\leq \eta_{t-k_{t}}\bb{\frac{2\mathcal{V}\left|\lambda_{2}\bb{P}\right|}{1 - \left|\lambda_{2}\bb{P}\right|} + \eta_{t-k_{t}}\mathcal{M}\bb{2\gamma\bb{1+8\epsilon} + \bb{2+\bb{1+\epsilon}^{2}}k_{t}^{2}\bb{\mathcal{M}+\lambda_{1}}^{2}}}\left\|G\right\|_{2}\mathbb{E}\left[\Tr\bb{\B_{t-k_{t}}\B_{t-k_{t}}^T}\right]
}
where in $\bb{i}$ we used $2\eta_{t-k_{t}} \leq \eta_{t} \leq \eta_{t-k_{t}}$ along with $\eta_{t}k_{t}\bb{\mathcal{M}+\lambda_{1}} \leq \epsilon$ in  $\bb{ii}$. Hence proved.
\end{proof}

\begin{lemma}
    \label{lemma:bound_Qt}
    Let $\forall i \in [n], \eta_{i}k_{i}\bb{\mathcal{M}+\lambda_{1}} \leq \epsilon, \epsilon \in \bb{0,1}$ and $\eta_{i}$ forms a non-increasing sequence. Set $k_{i} := \tau_{\text{mix}}\bb{\gamma\eta_{i}^{2}}, \gamma \in (0,1]$. Let $U \in \mathbb{R}^{d \times d}$ be a constant matrix and $Q_{t} := \Tr\bb{\B_{t-1}\B_{t-1}^T(A_t-\Sigma)U(A_t-\Sigma)}$. Further, let the decay of the step-sizes be slow such that $\forall i, \; \eta_{i} \leq \eta_{i-k_{i}} \leq 2\eta_{i}$. Then
    \bas {
        & \E\bbb{Q_{t}} \leq \bb{\mathcal{V} + \eta_{t-k_{t}+1}\mathcal{M}^{2}\bb{2\gamma\eta_{t}  + 2\bb{1+\epsilon}\bb{1+\epsilon\bb{1+\epsilon}}k_{t}\bb{\mathcal{M}+\lambda_{1}}}}\left\|U\right\|_{2}\E\bbb{\Tr\bb{B_{t-k_{t}}B_{t-k_{t}}^T}}
    }
    where $B_{t}$ is defined in \ref{definition:Bt}.
\end{lemma}
\begin{proof} Let $\B_{t} = \bb{I + R}\B_{t-k_{t}}$ with $\|R\|_{2} \leq r$. Then,
\bas{
\mathbb{E}\left[Q_{t}\right] &= 
\mathbb{E}\left[\underbrace{\Tr\bb{\B_{t-k_{t}}\B_{t-k_{t}}^T(A_t-\Sigma)U(A_t-\Sigma)}}_{Q_{t,1}}\right]+\mathbb{E}\left[\underbrace{\Tr\bb{\B_{t-k_{t}}\B_{t-k_{t}}^TR^{T}(A_t-\Sigma)U(A_t-\Sigma)}}_{Q_{t,2}}\right]\\
&\;\;\;\;\;\; +\mathbb{E}\left[\underbrace{\Tr\bb{R\B_{t-k_{t}}\B_{t-k_{t}}^T(A_t-\Sigma)U(A_t-\Sigma)}}_{Q_{t,3}}\right] + \mathbb{E}\left[\underbrace{\Tr\bb{R\B_{t-k_{t}}\B_{t-k_{t}}^TR^{T}(A_t-\Sigma)U(A_t-\Sigma)}}_{Q_{t,4}}\right]
}
Let's consider each of the terms above. Using Von-Neumann's trace inequality and noting that $\left\|\E_{\pi}\bbb{\bb{A_{t}-\Sigma}^{2}}\right\|_{2} \leq \mathcal{V}$, we have
\bas {
    \mathbb{E}\left[Q_{t,1}\right] &= \mathbb{E}\left[\Tr\bb{\B_{t-k_{t}}\B_{t-k_{t}}^T\E\bbb{(A_t-\Sigma)U(A_t-\Sigma)|s_{1},s_{2},\dots,s_{t-k_{t}}}}\right] \\
    &= \mathbb{E}\left[\Tr\bb{\B_{t-k_{t}}\B_{t-k_{t}}^T\E\bbb{(A_t-\Sigma)U(A_t-\Sigma)|s_{t-k_{t}}}}\right] \\
    &\leq \left\|\E\bbb{(A_t-\Sigma)U(A_t-\Sigma)|s_{t-k_{t}}}\right\|_{2}\mathbb{E}\left[\Tr\bb{\B_{t-k_{t}}\B_{t-k_{t}}^T}\right] 
    \\ 
    &\leq \left\|U\right\|_{2}\left\|\E\bbb{(A_t-\Sigma)^{2}|s_{t-k_{t}}}\right\|_{2}\mathbb{E}\left[\Tr\bb{\B_{t-k_{t}}\B_{t-k_{t}}^T}\right] \text{ using } \ref{eq:mixing_time_bound_square} \\ 
    &\leq \left\|U\right\|_{2}\bb{\mathcal{V} + 2d_{\text{mix}}\bb{k_{t}}\mathcal{M}^{2}}\mathbb{E}\left[\Tr\bb{\B_{t-k_{t}}\B_{t-k_{t}}^T}\right] \\ 
    & \leq \left\|U\right\|_{2}\bb{\mathcal{V} + 2\gamma\eta_{t}^{2}\mathcal{M}^{2}}\mathbb{E}\left[\Tr\bb{\B_{t-k_{t}}\B_{t-k_{t}}^T}\right] 
    }
\bas{
    \E\bbb{Q_{t,2}} &= \E\bbb{\Tr\bb{\B_{t-k_{t}}\B_{t-k_{t}}^T\E\bbb{R^{T}(A_t-\Sigma)U(A_t-\Sigma)|s_{1},s_{2},\dots,s_{t-k_{t}}}}} \\
                    &\leq \left\|\E\bbb{R^{T}(A_t-\Sigma)U(A_t-\Sigma)|s_{t-k_{t}}}\right\|_{2}\E\bbb{\Tr\bb{\B_{t-k_{t}}\B_{t-k_{t}}^T}} \\
                    &\leq \bb{1+\epsilon}\eta_{t-k_{t}+1}k_{t}\mathcal{M}^{2}\bb{\mathcal{M}+\lambda_{1}}\left\|U\right\|_{2}\E\bbb{\Tr\bb{\B_{t-k_{t}}\B_{t-k_{t}}^T}} \;\; \text{ using Lemma } \ref{lemma:etakproduct}
    \\ \\
    \E\bbb{Q_{t,3}} &\leq 
    \bb{1+\epsilon}\eta_{t-k_{t}+1}k_{t}\mathcal{M}^{2}\bb{\mathcal{M}+\lambda_{1}}\left\|U\right\|_{2}\E\bbb{\Tr\bb{\B_{t-k_{t}}\B_{t-k_{t}}^T}}
    \text{ using a similar argument as } Q_{t,2} 
    \\ 
    \\
    \E\bbb{Q_{t,4}} &= \mathbb{E}\left[\Tr\bb{R\B_{t-k_{t}}\B_{t-k_{t}}^TR^{T}(A_t-\Sigma)U(A_t-\Sigma)}\right] \\
                    &= \mathbb{E}\left[\Tr\bb{\B_{t-k_{t}}\B_{t-k_{t}}^TR^{T}(A_t-\Sigma)U(A_t-\Sigma)R}\right] \\
                    &\leq r^{2}\left\|U\right\|_{2}\mathcal{M}^{2}\mathbb{E}\left[\Tr\bb{\B_{t-k_{t}}\B_{t-k_{t}}^T}\right] \\
                    & \leq \bb{1+\epsilon}^{2}\eta_{t-k_{t}+1}^{2}k_{t}^{2}\mathcal{M}^{2}\bb{\mathcal{M}+\lambda_{1}}^{2}\left\|U\right\|_{2}\E\bbb{\Tr\bb{\B_{t-k_{t}}\B_{t-k_{t}}^T}} \text{ using Lemma } \ref{lemma:etakproduct} 
}
Therefore, we have
\bas{
    & \E\bbb{Q_{t}} \\ 
    &\leq \bb{\mathcal{V} + \eta_{t-k_{t}+1}\bb{2\gamma\eta_{t} \mathcal{M}^{2} + 2\bb{1+\epsilon}k_{t}\mathcal{M}^{2}\bb{\mathcal{M}+\lambda_{1}} + \bb{1+\epsilon}^{2}\eta_{t-k_{t}+1}k_{t}^{2}\mathcal{M}^{2}\bb{\mathcal{M}+\lambda_{1}}^{2}}}\left\|U\right\|_{2}\E\bbb{\Tr\bb{\B_{t-k_{t}}\B_{t-k_{t}}^T}} \\ 
    &\stackrel{(i)}\leq \bb{\mathcal{V} + \eta_{t-k_{t}+1}\mathcal{M}^{2}\bb{2\gamma\eta_{t}  + 2\bb{1+\epsilon}k_{t}\bb{\mathcal{M}+\lambda_{1}} + 2\epsilon\bb{1+\epsilon}^{2}k_{t}\bb{\mathcal{M}+\lambda_{1}}}}\left\|U\right\|_{2}\E\bbb{\Tr\bb{\B_{t-k_{t}}\B_{t-k_{t}}^T}} \\ 
    &= \bb{\mathcal{V} + \eta_{t-k_{t}+1}\mathcal{M}^{2}\bb{2\gamma\eta_{t}  + 2\bb{1+\epsilon}\bb{1+\epsilon\bb{1+\epsilon}}k_{t}\bb{\mathcal{M}+\lambda_{1}}}}\left\|U\right\|_{2}\E\bbb{\Tr\bb{\B_{t-k_{t}}\B_{t-k_{t}}^T}} 
}
In $\bb{i}$, we used the slow-decay assumption on $\eta_{i}$ mentioned in the lemma statement along with $\eta_{i}k_{i}\bb{\mathcal{M}+\lambda_{1}} \leq \epsilon$. Hence proved.
\end{proof}

\begin{lemma} \textit{(Learning Rate Schedule)}
    \label{lemma:learning_rate_schedule}
    Fix any $\delta \in \bb{0,1}$. Set $k_{i} := \tau_{\text{mix}}\bb{\eta_{i}^{2}}$. Suppose the step sizes are set such that
    \bas{
        \eta_{i} = \frac{\alpha}{\bb{\lambda_{1}-\lambda_{2}}\bb{\beta + i}}
    }
    Define the linear function 
    \bas{
        \forall i \in [n], \; f\bb{i} := \frac{1}{\eta_{i}} =  \frac{\bb{\lambda_{1}-\lambda_{2}}\bb{\beta + i}}{\alpha},
    }
    With $\epsilon := \frac{1}{100}$ and $\xi_{k,t}, \zeta_{k,t}, \mathcal{V}',\overline{\mathcal{V}_{k,t}}$ defined in \ref{assumptions:important_theorems}, set $\alpha > 2, \; f\bb{0} \geq e, \; m := 200$ and 
    \bas{ 
        & \beta := 600\max\left\{ \frac{\tau_{\text{mix}}\log\bb{f\bb{0}}\bb{\mathcal{M}+\lambda_{1}}\alpha}{\lambda_{1}-\lambda_{2}}, \; \frac{5\tau_{\text{mix}}\log\bb{f\bb{0}}\bb{\mathcal{M}+\lambda_{1}}^{2}\alpha^{2}}{3\bb{\lambda_{1}-\lambda_{2}}^{2}\log\bb{1+\frac{\delta}{m}}}, \frac{\bb{\mathcal{V}' + 5\lambda_{1}^{2}}\alpha^{2}}{300\bb{\lambda_{1}-\lambda_{2}}^{2}\log\bb{1 + \frac{\delta}{m}}}\right\}
    }
    then  we have
    \begin{enumerate}
        \item $\eta_{i}k_{i}\bb{\mathcal{M}+\lambda_{1}} \leq \epsilon \label{result_1}$
        \item $\forall i, \; \eta_{i} \leq \eta_{i-k_{i}} \leq \bb{1+2\epsilon}\eta_{i} \leq 2\eta_{i}$ (slow-decay)
        \item $\sum\limits_{i=1}^{n}\bb{\overline{\mathcal{V}_{k,i}} + \zeta_{k,i} + 4\lambda_{1}^{2}}\eta_{i}^{2}  \leq \log\bb{1+\frac{\delta}{m}} $
        \item $\sum\limits_{i=1}^{n}\bb{\mathcal{V}'+\xi_{k,i}}\eta_{i-k_{i}}^{2}\exp\left(-\sum\limits_{j=i+1}^{n}2\eta_j\left( \lambda_1-\lambda_{2}\right)\right) \leq $
        \bas{
            & \bb{\frac{2\bb{1+10\epsilon}\alpha^{2}}{2\alpha - 1}}\frac{\mathcal{V}'}{\bb{\lambda_{1}-\lambda_{2}}^{2}}\frac{1}{n} +  \bb{\frac{24\bb{1+10\epsilon}\alpha^{3}}{\bb{\alpha - 1}}}\frac{\mathcal{M}\bb{\mathcal{M}+\lambda_{1}}^{2}}{\bb{\lambda_{1}-\lambda_{2}}^{3}}\frac{k_{n}^{2}}{n^{2}} 
        }
    \end{enumerate}
\end{lemma}

\begin{proof}

We use the following inequalities - 
\begin{align}
    & \sum_{j=i}^{t}\eta_{j}^{2} \leq \frac{\alpha^{2}}{\left(\lambda_{1}-\lambda_{2}\right)^{2}\left(\beta+i-1\right)} \;\;\;\; \left(\text{Using } \frac{1}{x+1} \leq \sum_{i=1}^{\infty}\frac{1}{(x+i)^{2}} \leq \frac{1}{x}\right) \label{ineq:eta2_upper_bound} \\
    & \sum_{j=i}^{t}\eta_{j} \geq \frac{\alpha}{\left(\lambda_{1}-\lambda_{2}\right)}\log\left( \frac{t+\beta+1}{i+\beta}\right) \label{ineq:eta_lower_bound} \\
    & \sum_{j=i}^{t}\eta_{j} \leq  \frac{\alpha}{\left(\lambda_{1}-\lambda_{2}\right)}\log\left(\frac{t+\beta}{i+ \beta - 1}\right) \label{ineq:eta_upper_bound} \\
    & \sum_{j=i}^{t}(j+\beta)^{\ell} \leq \frac{(t+\beta+1)^{\ell+1}-(i+\beta)^{\ell+1}}{\ell+1} \leq \frac{(t+\beta+1)^{\ell+1}}{\ell+1} \; \forall \; \ell > 0 \label{ineq:sum_exponentials_upper_bound}
\end{align}
For the first result, we observe that $f(x) = \frac{\log\bb{x}}{x}$ is a decreasing function of $x$ for $x \geq e$. Using properties of the mixing time (see Section \ref{section:preliminaries} in the manuscript), we have
\ba{
    k_{i} := \tau_{\text{mix}}\bb{ \eta_{i}^{2}} \leq \frac{2\tau_{\text{mix}}}{\log\bb{2}}\log\bb{\frac{1}{\eta_{i}^{2}}} = \frac{4\tau_{\text{mix}}}{\log\bb{2}}\log\bb{\frac{\bb{\beta+i}\bb{\lambda_{1}-\lambda_{2}}}{\alpha}} = \frac{4\tau_{\text{mix}}}{\log\bb{2}}\log\bb{f\bb{i}} \label{eq:ki_definition}
}
for $\eta_{i} < 1$. For $i \geq 0$  
\bas{
    f\bb{i} \geq f\bb{0} = \frac{\beta\bb{\lambda_{1}-\lambda_{2}}}{\alpha} \geq e
}
Therefore,
\bas {
     \eta_{i}k_{i}\bb{\mathcal{M}+\lambda_{1}} &\leq \frac{4\tau_{\text{mix}}\bb{\mathcal{M}+\lambda_{1}}}{\log\bb{2}} \frac{\alpha}{\bb{\beta + i}\bb{\lambda_{1}-\lambda_{2}}}\log\bb{\frac{\bb{\beta+i}\bb{\lambda_{1}-\lambda_{2}}}{\alpha}} \\
     &= \frac{4\tau_{\text{mix}}\bb{\mathcal{M}+\lambda_{1}}}{\log\bb{2}} \frac{\log\bb{f\bb{i}}}{f\bb{i}} \\
     &\leq \frac{4\tau_{\text{mix}}\bb{\mathcal{M}+\lambda_{1}}}{\log\bb{2}} \frac{\log\bb{f\bb{0}}}{f\bb{0}}
}
From the assumptions mentioned in the Lemma statement, we have 
\ba{
    \frac{\log\bb{f\bb{0}}}{f\bb{0}} < \frac{\epsilon\log\bb{2}}{4\tau_{\text{mix}}\bb{\mathcal{M}+\lambda_{1}}} = \frac{\log\bb{2}}{400\tau_{\text{mix}}\bb{\mathcal{M}+\lambda_{1}}} \label{eq:condition1_learning_rate}
}
Therefore,
\ba{
    \forall \; i, \eta_{i}k_{i}\bb{\mathcal{M}+\lambda_{1}} \leq \epsilon \label{condition:1}
}
For the second result, we note that $\forall i \in \bbb{n}$,
\bas{
    \frac{\eta_{i-k_{i}}}{\eta_{i}} &= \frac{\beta + i}{\beta + i - k_{i}} \\
                                    &= 1 + \frac{k_{i}}{\beta + i - k_{i}} \\
                                    &= 1 + \frac{1}{\frac{\beta + i}{k_{i}} - 1}
}
Consider the fraction $\frac{\beta + i}{k_{i}}$. We can simplify it as : 
\bas {
    \frac{\beta + i}{k_{i}} &\geq \frac{\log\bb{2}}{4\tau_{\text{mix}}}\frac{\beta+i}{\log\bb{\frac{\bb{\beta+i}\bb{\lambda_{1}-\lambda_{2}}}{\alpha}}} \\
    &= \frac{\alpha\log\bb{2}}{4\tau_{\text{mix}}\bb{\lambda_{1}-\lambda_{2}}}\frac{f\bb{i}}{\log\bb{f\bb{i}}} \\
    &\geq \frac{\alpha\log\bb{2}}{4\tau_{\text{mix}}\bb{\lambda_{1}-\lambda_{2}}}\frac{f\bb{0}}{\log\bb{f\bb{0}}} \\ 
    &\geq \frac{1}{\epsilon} \text{ from } \ref{eq:condition1_learning_rate}
}
where we used the fact that $\frac{x}{\log\bb{x}}$ is an increasing function for $x \geq e$. Therefore, we have that 
\bas{
    \frac{\eta_{i-k_{i}}}{\eta_{i}} &\leq 1 + \frac{1}{\frac{1}{\epsilon} - 1} \\ 
    &= \frac{1}{1-\epsilon} \\ 
    &\leq 1+2\epsilon \text{ for } \epsilon \in \bb{0,0.1}
}
For the third result, we note that 
\bas{
    \zeta_{k,t} &:= 40k_{t+1}\bb{\mathcal{M}+\lambda_{1}}^{2}, \\ 
    \xi_{k,t} &:= 2\eta_{t}\mathcal{M}\bbb{3 + 9k_{t+1}^{2}\bb{\mathcal{M}+\lambda_{1}}^{2}} \\ 
    &\leq 24\eta_{t}\mathcal{M}\bbb{k_{t+1}^{2}\bb{\mathcal{M}+\lambda_{1}}^{2}} \text{ since } \bb{\mathcal{M}+\lambda_{1}} \geq 1 \; \text{WLOG} \\ 
    &\leq 24\epsilon\bb{1+\epsilon} k_{t+1}\bb{\mathcal{M}+\lambda_{1}}^{2} \text{since } \eta_{t} \leq \bb{1+2\epsilon}\eta_{t+1} \text{ and } \eta_{t+1}k_{t+1}\bb{\mathcal{M}+\lambda_{1}} \leq \epsilon 
} 
Therefore, 
\ba{
    \sum_{i=1}^{n}\bb{\overline{\mathcal{V}_{k,i}} + \zeta_{k,i}}\eta_{i}^{2}
    &= \bb{\mathcal{V}' + 5\lambda_{1}^{2}}\sum_{i=1}^{n}\eta_{i}^{2} + 41\bb{\mathcal{M}+\lambda_{1}}^{2}\sum_{i=1}^{n} \eta_{i}^{2}k_{i+1} \notag  \\ 
    & \stackrel{\bb{i}}\leq  \bb{\mathcal{V}' + 5\lambda_{1}^{2}}\underbrace{\sum_{i=1}^{n}\eta_{i}^{2}}_{T_{1}} + 45\bb{\mathcal{M}+\lambda_{1}}^{2}\underbrace{\sum_{i=1}^{n} \eta_{i+1}^{2}k_{i+1}}_{T_{2}} \label{condition:3}
}
where $\bb{i}$ follows from the slow decay property of $\eta_{i}$.
\\ 
\\
For $T_{1}$, using \ref{ineq:eta2_upper_bound} we have, 
\ba{
    T_{1} \leq \frac{\alpha^{2}}{\left(\lambda_{1}-\lambda_{2}\right)^{2}\beta} \label{eq:T1_bound}
}
For $T_{2}$, substituting the value of $k_{i}$ from \ref{eq:ki_definition}
for $\eta_{i} < 1$ we have,
\ba{
     T_{2} := \sum_{i=1}^{n}\eta_{i+1}^{2}k_{i+1} &\leq \frac{4\tau_{\text{mix}}}{\log\bb{2}}\sum_{i=1}^{n}\bb{\frac{\alpha}{\bb{\lambda_{1}-\lambda_{2}}\bb{\beta + i + 1}}}^{2}\log\bb{\frac{\bb{\lambda_{1}-\lambda_{2}}\bb{\beta + i + 1}}{\alpha}} \\ 
     &= \frac{4\tau_{\text{mix}}}{\log\bb{2}}\sum_{i=1}^{n}\frac{\log\bb{f\bb{i+1}}}{f\bb{i+1}^{2}} \label{eq:T2_bound}
}
Note that $f\bb{i}$ is a linear function of $i$ and $\forall i\; f\bb{i+1} - f\bb{i} = \frac{\lambda_{1}-\lambda_{2}}{\alpha}$. We observe that $g(x) = \frac{\log\bb{x}}{x^{2}}$ is a decreasing function of $x$ for $x \geq e^{\frac{1}{2}} \sim 1.65$. Therefore, 
\bas{
    \bb{\frac{\lambda_{1}-\lambda_{2}}{\alpha}}\sum_{i=1}^{n}\frac{\log\bb{f\bb{i+1}}}{f\bb{i+1}^{2}} \leq \int_{f\bb{1}}^{f\bb{n+1}} \frac{\log\bb{x}}{x^{2}} \,dx
}
Substituting in \ref{eq:T2_bound} we have, 
\bas {
    T_{2} &\leq \frac{4\tau_{\text{mix}}}{\log\bb{2}}\bb{\frac{\alpha}{\lambda_{1}-\lambda_{2}}}\int_{f\bb{1}}^{f\bb{n+1}} \frac{\log\bb{x}}{x^{2}} \,dx \\
    &= \frac{4\tau_{\text{mix}}}{\log\bb{2}}\bb{\frac{\alpha}{\lambda_{1}-\lambda_{2}}} \bb{-\bb{\frac{\log\bb{x}}{x} + \frac{1}{x}}\Bigg|_{f\bb{1}}^{f\bb{n}}} \\
    &\leq \frac{4\tau_{\text{mix}}}{\log\bb{2}}\bb{\frac{\alpha}{\lambda_{1}-\lambda_{2}}} \bb{\frac{\log\bb{f\bb{1}}}{f\bb{1}} + \frac{1}{f\bb{1}}} \\
    &\leq \frac{8\tau_{\text{mix}}}{\log\bb{2}}\bb{\frac{\alpha}{\lambda_{1}-\lambda_{2}}}\bb{\frac{\log\bb{f\bb{1}}}{f\bb{1}}} \\
    &\leq \frac{8\tau_{\text{mix}}}{\log\bb{2}}\bb{\frac{\alpha}{\lambda_{1}-\lambda_{2}}}\bb{\frac{\log\bb{f\bb{0}}}{f\bb{0}}} \text{ since } \frac{\log\bb{x}}{x} \text{ is a decreasing function of } x \text{ for } x \geq e
}
Putting everything together in \ref{condition:3} and using the bounds on $\beta, f\bb{0}$ mentioned in the lemma statement, we have, 
\bas{
    \sum_{i=1}^{n}\bb{\overline{\mathcal{V}_{k,i}} + \zeta_{k,i}}\eta_{i}^{2} &\leq 460\bb{\mathcal{M}+\lambda_{1}}^{2}\tau_{\text{mix}}\bb{\frac{\alpha}{\lambda_{1}-\lambda_{2}}}\frac{\log\bb{f\bb{0}}}{f\bb{0}} + \frac{\alpha^{2}}{\left(\lambda_{1}-\lambda_{2}\right)^{2}\beta}\bb{\mathcal{V}' + 5\lambda_{1}^{2}} \\
    &= 460\tau_{\text{mix}}\log\bb{f\bb{0}}\frac{\alpha^{2}}{\bb{\lambda_{1}-\lambda_{2}}^{2}\beta}\bb{\mathcal{M}+\lambda_{1}}^{2} + \frac{\alpha^{2}}{\left(\lambda_{1}-\lambda_{2}\right)^{2}\beta}\bb{\mathcal{V}' + 5\lambda_{1}^{2}} \\ 
    &\leq \log\bb{1+\frac{\delta}{m}}
}
Finally, for the last result we first note that 
\bas{
    \xi_{k,t} &:= 2\eta_{t}\mathcal{M}\bbb{3 + 9k_{t+1}^{2}\bb{\mathcal{M}+\lambda_{1}}^{2}} \\ 
    &\leq 24\eta_{t}\mathcal{M}\bbb{k_{t+1}^{2}\bb{\mathcal{M}+\lambda_{1}}^{2}} \text{ since } \bb{\mathcal{M}+\lambda_{1}} \geq 1 \; \text{WLOG}
}
Therefore,
\ba{
& \sum_{i=1}^{n}\bb{\mathcal{V}'+\xi_{k,i}}\eta_{i-k_{i}}^{2}\exp\left(-\sum_{j=i+1}^{n}2\eta_j\left( \lambda_1-\lambda_{2}\right)\right) \notag 
\\ & \;\;\;\;\;\;\;\;\;\;\;\;\;\;\; \leq \bb{1+2\epsilon}^{2}\sum_{i=1}^{n}\bb{\mathcal{V}'+\xi_{k,i}}\eta_{i}^{2}\exp\left(-\sum_{j=i+1}^{n}2\eta_j\left( \lambda_1-\lambda_{2}\right)\right) \notag \\
& \;\;\;\;\;\;\;\;\;\;\;\;\;\;\;\leq \bb{1+5\epsilon}\sum_{i=1}^{n}\bb{\mathcal{V}'+\xi_{k,i}}\eta_{i}^{2}\exp\left(-\sum_{j=i+1}^{n}2\eta_j\left( \lambda_1-\lambda_{2}\right)\right) \text{ since } \epsilon \in \bb{0,0.1} \notag \\
& \;\;\;\;\;\;\;\;\;\;\;\;\;\;\; = \bb{1+5\epsilon}\bbb{\sum_{i=1}^{n}\mathcal{V}'\eta_{i}^{2}\exp\left(-\sum_{j=i+1}^{n}2\eta_j\left( \lambda_1-\lambda_{2}\right)\right) + \sum_{i=1}^{n}\xi_{k,i}\eta_{i}^{2}\exp\left(-\sum_{j=i+1}^{n}2\eta_j\left( \lambda_1-\lambda_{2}\right)\right)}
}
Let's define 
\bas{
    & g\bb{i} := \exp\left(-\sum_{j=i+1}^{n}2\eta_j\left( \lambda_1-\lambda_{2}\right)\right), \;\; T_{3} := \sum_{i=1}^{n}\eta_{i}^{2}g\bb{i}, \;\; T_{4} := \sum_{i=1}^{n}\eta_{i}^{3}g\bb{i}, \;\; T_{5} := \sum_{i=1}^{n}\eta_{i}^{3}k_{i}^{2}g\bb{i},
}
Note that since $k_{n} \geq k_{i}$,
\bas{
    T_{5} &= \sum_{i=1}^{n}\eta_{i}^{3}k_{i}^{2}g\bb{i} \leq k_{n}^{2}\sum_{i=1}^{n}\eta_{i}^{3}g\bb{i} = k_{n}^{2}T_{4}
}
Then, 
\ba{
\sum_{i=1}^{n}\bb{\mathcal{V}'+\xi_{k,i}}\eta_{i-k_{i}}^{2}\exp\left(-\sum_{j=i+1}^{n}2\eta_j\left( \lambda_1-\lambda_{2}\right)\right) &\leq \bb{1+5\epsilon}\bbb{\mathcal{V'}T_{3} +  24\mathcal{M}\bb{\mathcal{M}+\lambda_{1}}^{2}T_{5}} \notag \\ 
&\leq \bb{1+5\epsilon}\bbb{\mathcal{V'}T_{3} + 24\mathcal{M}\bb{\mathcal{M}+\lambda_{1}}^{2}k_{n}^{2}T_{4}} \label{condition:4}
}
Using \ref{ineq:eta_lower_bound}, $g\bb{i} \leq \bb{\frac{i+\beta+1}{n+\beta+1}}^{2\alpha}$. Noting that $\bb{\frac{\beta+1}{\beta}}^{2} \leq \bb{\frac{\beta+1}{\beta}}^{3} \leq 2$, we have
\ba {
    T_{3} &:= \sum_{i=1}^{n}\eta_{i}^{2} \exp\bb{-2\sum_{j=i+1}^{n}\eta_{j}\bb{\lambda_{1}-\lambda_{2}}} \notag \\
          &=\bb{\frac{\alpha}{\lambda_{1}-\lambda_{2}}}^{2}\sum_{i=1}^{n}\frac{1}{\bb{\beta+i}^{2}}\bb{\frac{i+\beta+1}{n+\beta+1}}^{2\alpha} \notag \\
          &\leq \bb{\frac{\alpha}{\lambda_{1}-\lambda_{2}}}^{2}\bb{\frac{\beta+1}{\beta}}^{2}\sum_{i=1}^{n}\frac{1}{\bb{\beta+i+1}^{2}}\bb{\frac{i+\beta+1}{n+\beta+1}}^{2\alpha} \notag \\
          &= \bb{\frac{\alpha}{\lambda_{1}-\lambda_{2}}}^{2}\bb{\frac{\beta+1}{\beta}}^{2}\sum_{i=1}^{n}\frac{1}{\bb{\beta+i+1}^{2}}\bb{\frac{i+\beta+1}{n+\beta+1}}^{2\alpha} \notag \\
          &\leq 2\bb{\frac{\alpha}{\lambda_{1}-\lambda_{2}}}^{2}\frac{1}{\bb{n+\beta+1}^{2\alpha}}\sum_{i=1}^{n}\bb{i+\beta+1}^{2\alpha-2} \notag \\
          &\leq \frac{2}{2\alpha - 1}\bb{\frac{\alpha}{\lambda_{1}-\lambda_{2}}}^{2}\frac{1}{\bb{n+\beta+2}}\bb{\frac{n+\beta+2}{n+\beta+1}}^{2\alpha} \text{ using } \ref{ineq:sum_exponentials_upper_bound} \notag \\ 
          &= \frac{2}{2\alpha - 1}\bb{\frac{\alpha}{\lambda_{1}-\lambda_{2}}}^{2}\frac{1}{\bb{n+\beta+2}}\bb{1 + \frac{1}{n+\beta + 1}}^{2\alpha}\label{eq:T_3_bound} 
}
and similarly,
\ba {
    T_{4} &:= \sum_{i=1}^{n}\eta_{i}^{3} \exp\bb{-2\sum_{j=i+1}^{n}\eta_{j}\bb{\lambda_{1}-\lambda_{2}}} \notag \\
          &=\bb{\frac{\alpha}{\lambda_{1}-\lambda_{2}}}^{3}\sum_{i=1}^{n}\frac{1}{\bb{\beta+i}^{3}}\bb{\frac{i+\beta+1}{n+\beta+1}}^{2\alpha} \notag \\
          &\leq \bb{\frac{\alpha}{\lambda_{1}-\lambda_{2}}}^{3}\bb{\frac{\beta+1}{\beta}}^{3}\sum_{i=1}^{n}\frac{1}{\bb{\beta+i+1}^{3}}\bb{\frac{i+\beta+1}{n+\beta+1}}^{2\alpha} \notag \\
          &= \bb{\frac{\alpha}{\lambda_{1}-\lambda_{2}}}^{3}\bb{\frac{\beta+1}{\beta}}^{3}\sum_{i=1}^{n}\frac{1}{\bb{\beta+i+1}^{2}}\bb{\frac{i+\beta+1}{n+\beta+1}}^{2\alpha} \notag \\
          &\leq 2\bb{\frac{\alpha}{\lambda_{1}-\lambda_{2}}}^{3}\frac{1}{\bb{n+\beta+1}^{2\alpha}}\sum_{i=1}^{n}\bb{i+\beta+1}^{2\alpha-3} \notag
}
\ba{
          &\leq \frac{1}{\alpha - 1}\bb{\frac{\alpha}{\lambda_{1}-\lambda_{2}}}^{3}\frac{1}{\bb{n+\beta+2}^{2}}\bb{\frac{n+\beta+2}{n+\beta+1}}^{2\alpha} \text{ using } \ref{ineq:sum_exponentials_upper_bound} \notag \\ 
          &= \frac{1}{\alpha - 1}\bb{\frac{\alpha}{\lambda_{1}-\lambda_{2}}}^{3}\frac{1}{\bb{n+\beta+2}^{2}}\bb{1 + \frac{1}{n+\beta+1}}^{2\alpha}
          \label{eq:T_4_bound}
}
Using \ref{condition:1}, we have 
\ba{
    \frac{\alpha}{n+\beta+1} = \eta_{n}\bb{\lambda_{1}-\lambda_{2}} \leq \eta_{n}\lambda_{1} \leq \eta_{n}k_{n}\lambda_{1} \leq \epsilon \leq 0.1 \label{alpha_n_bound}
}
Therefore, using \cite{kozma}
\ba{
    \bb{1 + \frac{1}{n+\beta+1}}^{2\alpha} \stackrel{(i)}\leq \frac{1}{1 - \frac{2\alpha}{n+\beta+1}} \stackrel{(ii)}\leq 1 + \frac{4\alpha}{n+\beta+1} \leq 1 + 4\epsilon \label{n_power_alpha_bound}
}
where $(i)$ follows since $\frac{2\alpha}{n+\beta+1} < 1$ by \ref{alpha_n_bound} and $(ii)$ follows since $\frac{1}{1-x} \leq 1+2x \text{ for } x \in [0,\frac{1}{2}]$.

Using \ref{n_power_alpha_bound} with \ref{eq:T_3_bound}, we have
\ba{
    T_{3} &\leq \frac{2}{2\alpha - 1}\bb{\frac{\alpha}{\lambda_{1}-\lambda_{2}}}^{2}\frac{1}{\bb{n+\beta+2}}\bb{1 + \frac{4\alpha}{n+\beta+1}} \notag \\ 
    &\leq \frac{2\bb{1+4\epsilon}}{2\alpha - 1}\bb{\frac{\alpha}{\lambda_{1}-\lambda_{2}}}^{2}\frac{1}{\bb{n+\beta+2}} \label{eq:T_3_bound_modified}
}
Using \ref{n_power_alpha_bound} with \ref{eq:T_4_bound}, we have
\ba{
    T_{4} &\leq \frac{1+4\epsilon}{\alpha - 1}\bb{\frac{\alpha}{\lambda_{1}-\lambda_{2}}}^{3}\frac{1}{\bb{n+\beta+2}^{2}} \label{eq:T_4_bound_modified}
}
Let 
\bas{
    C_{1} := \frac{2\bb{1+10\epsilon}\alpha^{2}}{2\alpha - 1}, C_{2} := \frac{24\bb{1+10\epsilon}\alpha^{3}}{\bb{\alpha - 1}},
}
Putting together \ref{eq:T_3_bound_modified}, \ref{eq:T_4_bound_modified} in \ref{condition:4} and using the definition of $k_{i}$ in \ref{eq:ki_definition} we have 

\bas{
    \bb{1+5\epsilon}\mathcal{V}'T_{3} &\leq \frac{2\bb{1+5\epsilon}\bb{1+4\epsilon}}{2\alpha - 1}\bb{\frac{\alpha}{\lambda_{1}-\lambda_{2}}}^{2}\frac{\mathcal{V}'}{\bb{n+\beta+2}} \\ 
    &\leq \frac{2\bb{1+10\epsilon}\alpha^{2}}{2\alpha - 1}\frac{\mathcal{V}'}{\bb{\lambda_{1}-\lambda_{2}}^{2}}\frac{1}{n} \text{ since } \epsilon \leq 0.05
}
and similarly,
\bas{
  24\bb{1+5\epsilon}\mathcal{M}\bb{\mathcal{M}+\lambda_{1}}^{2}k_{n}^{2}T_{4} &\leq \frac{24\bb{1+5\epsilon}\bb{1+4\epsilon}\alpha^{3}}{\alpha - 1}\frac{\mathcal{M}\bb{\mathcal{M}+\lambda_{1}}^{2}}{\bb{\lambda_{1}-\lambda_{2}}^{3}}\frac{k_{n}^{2}}{n^{2}} 
}
Therefore from \ref{condition:4}, we have
\bas{
    & \sum_{i=1}^{n}\bb{\mathcal{V}'+\xi_{k,i}}\eta_{i-k_{i}}^{2}\exp\left(-\sum_{j=i+1}^{n}2\eta_j\left( \lambda_1-\lambda_{2}\right)\right) \leq C_{1}\frac{\mathcal{V}'}{\bb{\lambda_{1}-\lambda_{2}}^{2}}\frac{1}{n} +  C_{2}\frac{\mathcal{M}\bb{\mathcal{M}+\lambda_{1}}^{2}}{\bb{\lambda_{1}-\lambda_{2}}^{3}}\frac{k_{n}^{2}}{n^{2}}
}
Hence proved.
\end{proof}

\section{Proofs : Convergence Analysis of Oja's Algorithm for Markovian Data}
\label{appendix:proofs_of_bounding_convergence}

\setcounter{theorem}{1}

In this section, we present proofs of Theorems \ref{theorem:v1upperbound}, \ref{theorem:Vperpupperbound}, \ref{theorem:v1lowerbound} and \ref{theorem:v1squareupperbound}. We state versions of these theorems that are valid under more general conditions on the step sizes. Specifically, for the following, we only require a sequence of non-increasing step-sizes which satisfy, for $\epsilon := \frac{1}{100}, \forall i \in [n]$ -  
\begin{enumerate}[label=\textbf{C.\arabic*}]
    \item[\textbf{C.1}]$\eta_{i}k_{i}\bb{\mathcal{M}+\lambda_{1}} \leq \epsilon$
    \qquad\quad \ \ \textbf{C.2  }\ \ (Slow decay) $\eta_{i} \leq \eta_{i-k_{i}} \leq \bb{1+2\epsilon}\eta_{i} \leq 2\eta_{i}$
\end{enumerate}
The version of these theorems stated in the main manuscript are obtained by plugging in the step-sizes as $\eta_{i}:= \frac{\alpha}{\bb{\lambda_{1}-\lambda_{2}}\bb{\beta + i}}$ for the values of $\alpha, \beta$ provided in Lemma \ref{lemma:learning_rate_schedule}. Before starting with the proofs, we define the following scalar variables - 
\begin{align}  \label{assumptions:important_theorems}
    r &:= 2\bb{1+\epsilon}k_{n}\eta_{n}\left(\mathcal{M}+\lambda_{1}\right), \qquad \zeta_{k,t} := 40k_{t+1}\bb{\mathcal{M}+\lambda_{1}}^{2}\notag\\
    \psi_{k,t} &:= 6\mathcal{M}\bbb{1 + 3k_{t+1}^{2}\bb{\mathcal{M}+\lambda_{1}}^{2}}, \qquad \xi_{k,t} := \eta_{t-k_{t}}\psi_{k,t} \notag\\
    \mathcal{V}' &:= \frac{1 + \bb{3 + 4\epsilon}|\lambda_{2}\bb{P}|}{1 - \left|\lambda_{2}\bb{P}\right|}\mathcal{V},\ \  \qquad\qquad\overline{\mathcal{V}_{k,t}} := \mathcal{V}' + \lambda_{1}^{2} + \xi_{k,t}
\end{align}

The basic idea behind these proofs is illustrated in Figure \ref{fig:approx1}, where we are trying to approximate the matrix product by conditioning back in time just the right amount, to balance the tradeoff between the advantage of the mixing decay and the norm of the product of matrices.

\begin{figure}[H]
    \centering
    \includegraphics[width=0.6\textwidth]{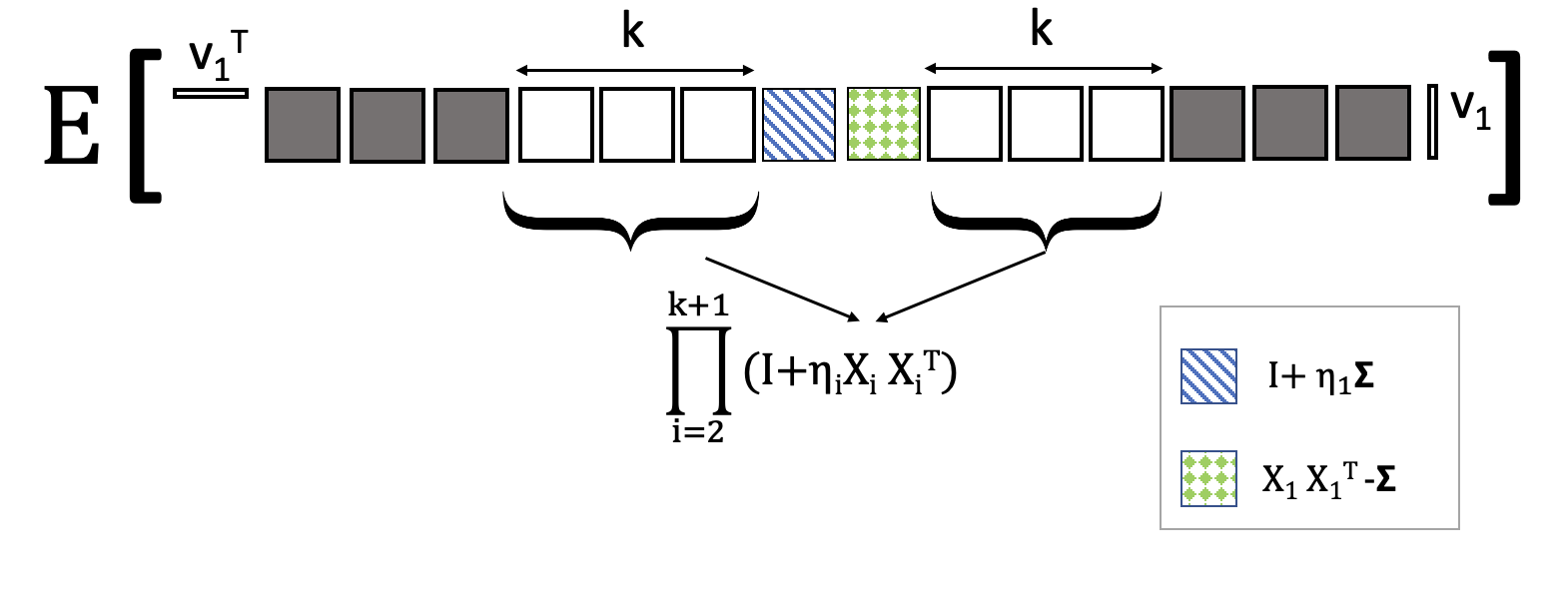}
    \caption{If we could replace the intermediate products (white matrices) by $I$, the conditional expectation of the noise matrix $X_1X_1^T-\Sigma$ conditioned on the grey matrices would be nearly zero.}
    \label{fig:approx1}
\end{figure}

\begin{theorem} (General Version)
\label{theorem::appendix:v1upperbound}
Under Assumptions~\ref{ass:model},~\ref{assumption:variance_bound} and~\ref{assumption:norm_bound}, for all $n>k_n$, and any decaying step-size schedule $\eta_i$ satisfying~\textbf{C.1} and~\textbf{C.2},  we have:
\begin{align*}
    & \mathbb{E}\left[v_{1}^{T}\newB{1}{n}\newB{1}{n}^{T}v_{1}\right] \leq \bb{1+r}^{2}\exp\left(\sum_{t=1}^{n-k_{n}}\bb{2\eta_{t}\lambda_{1} + \eta_{t}^{2}\bb{\mathcal{V}'+\lambda_1^2 + \xi_{k,t}}}\right)
\end{align*}
where $\newB{i}{j}$ is defined in \ref{definition:Bji}.
\end{theorem}

\begin{proof}
Define $\alpha_{n,t} := \mathbb{E}\left[\Tr\bb{v_{1}^{T}\newB{t}{n}\newB{t}{n}^{T}v_{1}}\right] = \mathbb{E}\left[v_{1}^{T}\newB{t}{n}\newB{t}{n}^{T}v_{1}\right], i \leq t \leq n $. Then, we have

\ba{
v_1^T\newB{t}{n}\newB{t}{n}^{T} v_1
&= v_1^T\newB{t+1}{n}(I+\eta_{t}\Sigma)^2\newB{t+1}{n}^{T} v_1\;+\;2\eta_{t} \underbrace{\bb{v_1^T\newB{t+1}{n}(I+\eta_{t}\Sigma)(A_{t}-\Sigma)\newB{t+1}{n}^{T} v_1}}_{P_{n,t}} \notag \\
&+\;\eta_{t}^2\underbrace{\bb{v_1^T\newB{t+1}{n}(A_t-\Sigma)^2\newB{t+1}{n}^{T} v_1}}_{Q_{n,t}} \label{eq:v1} \\
&\leq v_1^T\newB{t+1}{j}\newB{t+1}{j}^{T} v_1((1+\eta_{t}\lambda_1)^2) + \eta_{t}^2Q_{n,t} + 2\eta_{t}P_{n,t}\notag
}
Using Lemma \ref{lemma:trace_linear_bound} with $U = v_{1}, G = \bb{I + \eta_{t}\Sigma}, \gamma = 1$ and noting that $\E_{\pi}\bbb{A_{t}-\Sigma} = 0$, along with observing that $\alpha_{n,t+k_{t+1}} \leq \alpha_{n,t+k_{t}}$ from Lemma \ref{lemma:inner_product_monotonicity}, we have
\begin{align*}
    \left|\mathbb{E}\left[P_{n,t}\right]\right| &\leq \eta_{t+1}\bb{1+\eta_{t}\lambda_{1}}\bb{\frac{2\mathcal{V}\left|\lambda_{2}\bb{P}\right|}{1 - \left|\lambda_{2}\bb{P}\right|} + \eta_{t+1}\mathcal{M}\bb{2+16\epsilon + \bb{2+\bb{1+\epsilon}^{2}}k_{t+1}^{2}\bb{\mathcal{M}+\lambda_{1}}^{2}}}\alpha_{n,t+k_{{t}}}
\end{align*}
We note that $\forall i, \; k_{i} \geq 1$, therefore, using the assumption in \ref{assumptions:important_theorems}, $1+\eta_{t}\lambda_{1} \leq 1+\eta_{t}k_{t}\bb{\mathcal{M}+\lambda_{1}} \leq 1+\epsilon$. 
\\ 
\\ 
Next, using Lemma \ref{lemma:trace_square_bound} with $U = v_{1}, G = I,\gamma = 1$ and noting that $\left\|\E_{\pi}\bbb{\bb{A_{t}-\Sigma}^{2}}\right\|_{2} \leq \mathcal{V}$ along with observing that $\alpha_{n,t+k_{t+1}} \leq \alpha_{n,t+k_{t}}$ using Lemma \ref{lemma:inner_product_monotonicity}, we have
\bas{
    \left|\mathbb{E}\left[Q_{n,t}\right]\right| &\leq \bb{\mathcal{V} + \eta_{t+1}\mathcal{M}^{2}\bb{2\eta_{t+1} + \bb{1+\epsilon}\bb{2+\epsilon\bb{1+\epsilon}}k_{t+1}\bb{\mathcal{M}+\lambda_{1}}}}\alpha_{n,t+k_{{t}}} \\ 
    &\leq \bb{\mathcal{V} + 2\epsilon\eta_{t+1}\mathcal{M} + \eta_{t+1}\mathcal{M}^{2}\bb{ \bb{1+\epsilon}\bb{2+\epsilon\bb{1+\epsilon}}k_{t+1}\bb{\mathcal{M}+\lambda_{1}}}}\alpha_{n,t+k_{{t}}}
}
where in the last line, we used $\eta_{t+1}\mathcal{M} \leq \eta_{t+1}\bb{\mathcal{M}+\lambda_{1}} \leq \eta_{t+1}k_{t+1}\bb{\mathcal{M}+\lambda_{1}} \leq \epsilon$.
\\ 
\\ 
Then from \ref{eq:v1} for $n-k_{t} \geq t \geq 1$, 
\ba{
    \alpha_{n,t} &\leq \left(1+\eta_{t}\lambda_{1}\right)^{2}\alpha_{n,t+1} +  \bb{\frac{1 + \bb{3 + 4\epsilon}|\lambda_{2}\bb{P}|}{1 - \left|\lambda_{2}\bb{P}\right|}}\mathcal{V}\eta_{t}^{2}\alpha_{n,t+k_{t}} + C_{k,t}\eta_{t}^{3}\alpha_{n,t+k_{t}} \label{eq:bound_C_kt}
}
where $C_{k,t}$ is defined as 
\bas{
     C_{k,t} &:= \mathcal{M}\bbb{4\bb{1+\epsilon}\bb{1+8\epsilon}+2\epsilon + k_{t+1}\bb{\mathcal{M}+\lambda_{1}}\bb{\bb{1+\epsilon}\bb{2+\epsilon\bb{1+\epsilon}}\mathcal{M} + 2\bb{2+\bb{1+\epsilon}^{2}}k_{t+1}\bb{\mathcal{M}+\lambda_{1}}}} \\ 
     &\stackrel{\bb{i}}\leq \mathcal{M}\bbb{4\bb{1+\epsilon}\bb{1+8\epsilon}+2\epsilon + \bb{\bb{1+\epsilon}\bb{2+\epsilon\bb{1+\epsilon}} + 2\bb{2+\bb{1+\epsilon}^{2}}}k_{t+1}^{2}\bb{\mathcal{M}+\lambda_{1}}^{2}} \\ 
     &= \mathcal{M}\bbb{4 + 38\epsilon + 32\epsilon^{2} + \bb{6 + 2\epsilon + \bb{1+\epsilon}^{2}\bb{1+2\epsilon}}k_{t+1}^{2}\bb{\mathcal{M}+\lambda_{1}}^{2}}
}
where in $\bb{i}$ we used $\mathcal{M} \leq k_{t+1}\bb{\mathcal{M}+\lambda_{1}}$.
\\ 
\\ 
Then recalling the definition of $\xi_{k,t}$ in \ref{assumptions:important_theorems}, and noting that $\alpha_{n,t+k_{t}} \leq \alpha_{n,t+1}$ using Lemma \ref{lemma:inner_product_monotonicity} we have from \ref{eq:bound_C_kt},
\begin{align*}
    \alpha_{n,t} &\leq \left(1+\eta_{t}\lambda_{1}\right)^{2}\alpha_{n,t+1} +  \bb{\bb{\frac{1 + \bb{3 + 4\epsilon}|\lambda_{2}\bb{P}|}{1 - \left|\lambda_{2}\bb{P}\right|}}\mathcal{V} + \xi_{k,t}}\eta_{t}^{2}\alpha_{n,t+k_{t}} \\ 
    &= \bb{1 + 2\eta_{t}\lambda_{1} + \eta_{t}^{2}\bb{\bb{\frac{1 + \bb{3 + 4\epsilon}|\lambda_{2}\bb{P}|}{1 - \left|\lambda_{2}\bb{P}\right|}}\mathcal{V} + \lambda_{1}^{2} + \xi_{k,t}}}\alpha_{n,t+1}
\end{align*}
Therefore using this recursion, we have, 
\begin{align*}
    \alpha_{n,1} &\leq \alpha_{n,n-k_{n}+1}\exp\left(2\lambda_{1}\sum_{t=1}^{n-k_{n}}\eta_{t} + \sum_{t=1}^{n-k_{n}}\eta_{t}^{2}\bb{\bb{\frac{1 + \bb{3 + 4\epsilon}|\lambda_{2}\bb{P}|}{1 - \left|\lambda_{2}\bb{P}\right|}}\mathcal{V} + \lambda_{1}^{2} + \xi_{k,t}}\right)
\end{align*}
Let $\newB{n-k_{n}+1}{n} = I + R'$, where $\|R'\|\leq r$ a.s.
\begin{align*}
    \alpha_{n,n-k_{n}+1} & = \mathbb{E}\left[v_{1}^{T}\newB{n-k_{n}+1}{n}\newB{n-k_{n}+1}{n}^{T}v_1\right] \\
    &= \mathbb{E}\left[v_{1}^{T}v_{1}\right] + \mathbb{E}\left[v_{1}^{T}(R' + R^{\prime T})v_{1}\right] + \mathbb{E}\left[v_{1}^{T}R'R^{\prime T}v_{1}\right] \\
    & \leq 1+2r+r^{ 2} 
\end{align*}
Using Lemma \ref{lemma:etakproduct} we have 
\bas{
        r   & \leq \bb{1+\epsilon}k_{n}\eta_{n-k_{n}+1}\left(\mathcal{M}+\lambda_{1}\right) \\
        & \leq \bb{1+\epsilon}k_{n}\eta_{n-k_{n}}\left(\mathcal{M}+\lambda_{1}\right) \\
        & \leq 2\bb{1+\epsilon}k_{n}\eta_{n}\left(\mathcal{M}+\lambda_{1}\right) \;\; \text{since } \eta_{n-k_{n}} \leq 2\eta_{n}
}
Therefore, 
\begin{align}
    \alpha_{n,1} &\leq \bb{1+2r+r^{2}}\exp\left(2\lambda_{1}\sum_{t=1}^{n-k_{n}}\eta_{t} + \sum_{t=1}^{n-k_{n}}\eta_{t}^{2}\bb{\bb{\frac{1 + \bb{3 + 4\epsilon}|\lambda_{2}\bb{P}|}{1 - \left|\lambda_{2}\bb{P}\right|}}\mathcal{V} + \lambda_{1}^{2} + \xi_{k,t}}\right) \notag
\end{align}
Hence proved.
\end{proof}

\begin{theorem} (General Version)
\label{theorem:appendix:Vperpupperbound}
Let $\init := \min\left\{t : t \in [n], t-k_{t} \geq 0 \right\}$. Under Assumptions~\ref{ass:model},~\ref{assumption:variance_bound} and~\ref{assumption:norm_bound}, for all $n>\init$, and any decaying step-size $\eta_i$ satisfying~\textbf{C.1} and~\textbf{C.2},  we have,
\begin{align*}
\mathbb{E}\left[\Tr\left(V_{\perp}^{T}\B_{n}\B_{n}^{T}V_{\perp}\right)\right] &\leq \bb{1+5\epsilon}\exp\left(\sum_{i=\init+1}^n 2\eta_i\lambda_2+\bb{\mathcal{V}'+\lambda_1^2+\xi_{k,i}}\eta_{i-k_{i}}^2\right)\\
&\qquad\times\left(d + \sum_{i=\init+1}^{n}\bb{\mathcal{V}'+\xi_{k,i}}C_{k,i}'\eta_{i-k_{i}}^{2}\exp\left(\sum_{j=\init+1}^{i}2\eta_j\left( \lambda_1-\lambda_{2}\right)\right)\right)
\end{align*}
where $C_{k,t}' := \exp\bb{2\lambda_{1}\sum_{j=1}^{\init}\bb{\eta_{j}-\eta_{t-\init+j}}+\sum_{j=1}^{t-\init}\eta_{j}^{2}\bb{\overline{\mathcal{V}_{k,j}}-\overline{\mathcal{V}_{k,j+u}}}}$
and $B_{t}$ is defined in \ref{definition:Bt}.
\end{theorem}

\begin{proof}
For $ t \leq n$, let
\bas{
    & \alpha_{t} := \alpha_{t,1} = \mathbb{E}\left[v_{1}^{T}\B_{t}\B_{t}^{T}v_{1}\right] = \mathbb{E}\left[\Tr\left(v_{1}^{T}\B_{t}\B_{t}^{T}v_{1}\right)\right], \; \text{as defined in Theorem } \ref{theorem:v1upperbound} \\
    & \beta_{t} := \mathbb{E}\left[\Tr\left(V_{\perp}^{T}\B_{t}\B_{t}^{T}V_{\perp}\right)\right]
}
Note that $\alpha_{t} + \beta_{t} = \Tr\bb{\B_t\B_t^T}$ by definition. Then,
\bas{
\Tr\bb{\B_t\B_t^T\vp\vp^T} &= \Tr\bb{\B_{t-1}\B_{t-1}^T(I+\eta_t\Sigma)\vp\vp^T(I+\eta_t\Sigma)} + \eta_t\Tr\bb{\B_{t-1}^T(I+\eta_t\Sigma)\vp\vp^T(A_t-\Sigma)\B_{t-1}}\\
&\;\;\;\;\;\;\;\; + \eta_t\Tr\bb{\B_{t-1}^T(A_t-\Sigma)\vp\vp^T(I+\eta_t\Sigma)\B_{t-1}} +\eta_t^2\Tr\bb{\B_{t-1}\B_{t-1}^T(A_t-\Sigma)\vp\vp^T(A_t-\Sigma)}\\
&\leq (1+\eta_t\lambda_2)^2\Tr\bb{\B_{t-1}\B_{t-1}^T\vp\vp^T} + 2\eta_t\underbrace{\Tr\bb{\B_{t-1}\B_{t-1}^T(I+\eta_t\Sigma)\vp\vp^T(A_t-\Sigma)}}_{P_{t}}\\
&\;\;\;\;\;\;\;\; + \eta_t^2\underbrace{\Tr\bb{\B_{t-1}\B_{t-1}^T(A_t-\Sigma)\vp\vp^T(A_t-\Sigma)}}_{Q_{t}}
}
Let $\B_{t-1} = \left(I + R\right)\B_{t-k_{t}}$ with $\|R\|_{2} \leq r$. Using Lemma \ref{lemma:bound_Pt} with $G = \bb{I+\eta_t\Sigma}\vp\vp^{T}=\vp(I+\eta_t\Lambda_{\perp})\vp^T, \gamma = 1$, where $\Lambda_\perp$ is a $d-1\times d-1 $ diagonal matrix of eigenvalues $\lambda_2,\dots,\lambda_d$ of $\Sigma$, and noting that $\left\|\vp\vp^T\right\|_{2} = 1$, 
\bas{
    \E\bbb{P_{t}} &\leq \bb{1+\eta_{t}\lambda_{1}}\eta_{t-k_{t}}\bb{\frac{2\mathcal{V}\left|\lambda_{2}\bb{P}\right|}{1 - \left|\lambda_{2}\bb{P}\right|} + \eta_{t-k_{t}}\mathcal{M}\bb{2\bb{1+8\epsilon} + \bb{2+\bb{1+\epsilon}^{2}}k_{t}^{2}\bb{\mathcal{M}+\lambda_{1}}^{2}}}\bb{\alpha_{t-k_{t}}+\beta_{t-k_{t}}} \\ 
    &\leq \bb{1+\epsilon}\eta_{t-k_{t}}\bb{\frac{2\mathcal{V}\left|\lambda_{2}\bb{P}\right|}{1 - \left|\lambda_{2}\bb{P}\right|} + \eta_{t-k_{t}}\mathcal{M}\bb{2\bb{1+8\epsilon} + \bb{2+\bb{1+\epsilon}^{2}}k_{t}^{2}\bb{\mathcal{M}+\lambda_{1}}^{2}}}\bb{\alpha_{t-k_{t}}+\beta_{t-k_{t}}} 
}
where in the last line, we used $\eta_{t}\lambda_{1} \leq \eta_{t}k_{t}\bb{\mathcal{M} + \lambda_{1}} \leq \epsilon$.
\\
\\
Using Lemma \ref{lemma:bound_Qt} with $U = \vp\vp^{T}, \gamma = 1$, 
\bas{
    \E\bbb{Q_{t}} &\leq \bb{\mathcal{V} + \eta_{t-k_{t}+1}\mathcal{M}^{2}\bb{2\eta_{t}  + 2\bb{1+\epsilon}\bb{1+\epsilon\bb{1+\epsilon}}k_{t}\bb{\mathcal{M}+\lambda_{1}}}}\bb{\alpha_{t-k_{t}}+\beta_{t-k_{t}}} \\ 
    &\stackrel{(i)}\leq \bb{\mathcal{V} + 2\epsilon\eta_{t}\mathcal{M} + 2\eta_{t-k_{t}+1}\mathcal{M}^{2}\bb{\bb{1+\epsilon}\bb{1+\epsilon\bb{1+\epsilon}}k_{t}\bb{\mathcal{M}+\lambda_{1}}}}\bb{\alpha_{t-k_{t}}+\beta_{t-k_{t}}} \\ 
    &\stackrel{(ii)}\leq \bb{\mathcal{V} + 2\epsilon\eta_{t}\mathcal{M} + 2\eta_{t-k_{t}+1}\mathcal{M}\bb{\bb{1+\epsilon}\bb{1+\epsilon\bb{1+\epsilon}}k_{t}^{2}\bb{\mathcal{M}+\lambda_{1}}^{2}}}\bb{\alpha_{t-k_{t}}+\beta_{t-k_{t}}}
}
where in $(i)$ we used $\forall i,\; \eta_{i}\mathcal{M} \leq \eta_{i}k_{i}\bb{\mathcal{M}+\lambda_{1}} \leq \epsilon$ and in $(ii)$ we used $\mathcal{M} \leq k_{t}\bb{\mathcal{M}+\lambda_{1}}$. Putting everything together, we have, 
\bas {
    & \E\bbb{\Tr\bb{\B_t\B_t^T\vp\vp^T}} \\
    & \;\; \leq \bb{1 + \eta_{t}\lambda_{2}}^{2}\beta_{t-1} \\ & \;\;\;\;\;\;\;\; + 2\bb{1+\epsilon}\eta_{t}\eta_{t-k_{t}}\bb{\frac{2\mathcal{V}\left|\lambda_{2}\bb{P}\right|}{1 - \left|\lambda_{2}\bb{P}\right|} + \eta_{t-k_{t}}\mathcal{M}\bb{2\bb{1+8\epsilon} + \bb{2+\bb{1+\epsilon}^{2}}k_{t}^{2}\bb{\mathcal{M}+\lambda_{1}}^{2}}}\bb{\alpha_{t-k_{t}}+\beta_{t-k_{t}}} \\
    & \;\;\;\;\;\;\;\;\; + \eta_{t}^{2}\bb{\mathcal{V} + 2\epsilon\eta_{t}\mathcal{M} + 2\eta_{t-k_{t}+1}\mathcal{M}\bb{\bb{1+\epsilon}\bb{1+\epsilon\bb{1+\epsilon}}k_{t}^{2}\bb{\mathcal{M}+\lambda_{1}}^{2}}}\bb{\alpha_{t-k_{t}}+\beta_{t-k_{t}}} \\
    & \;\; \leq \bb{1 + \eta_{t}\lambda_{2}}^{2}\beta_{t-1} \\ & \;\;\;\;\;\;\;\; + 2\bb{1+\epsilon}\eta_{t-k_{t}}^{2}\bb{\frac{2\mathcal{V}\left|\lambda_{2}\bb{P}\right|}{1 - \left|\lambda_{2}\bb{P}\right|} + \eta_{t-k_{t}}\mathcal{M}\bb{2\bb{1+8\epsilon} + \bb{2+\bb{1+\epsilon}^{2}}k_{t}^{2}\bb{\mathcal{M}+\lambda_{1}}^{2}}}\bb{\alpha_{t-k_{t}}+\beta_{t-k_{t}}} \\
    & \;\;\;\;\;\;\;\;\; + \eta_{t-k_{t}}^{2}\bb{\mathcal{V} + 2\epsilon\eta_{t}\mathcal{M} + 2\eta_{t-k_{t}+1}\mathcal{M}\bb{\bb{1+\epsilon}\bb{1+\epsilon\bb{1+\epsilon}}k_{t}^{2}\bb{\mathcal{M}+\lambda_{1}}^{2}}}\bb{\alpha_{t-k_{t}}+\beta_{t-k_{t}}} \\
    & \;\; \leq \bb{1 + \eta_{t}\lambda_{2}}^{2}\beta_{t-1} + \eta_{t-k_{t}}^{2}\bb{\bb{\frac{1 + \bb{3 + 4\epsilon}|\lambda_{2}\bb{P}|}{1 - \left|\lambda_{2}\bb{P}\right|}}\mathcal{V} + \xi_{k,t}}\bb{\alpha_{t-k_{t}}+\beta_{t-k_{t}}}
}
where $\xi_{k,t}$ is as defined in \ref{assumptions:important_theorems}.
Therefore using Lemma \ref{lemma:inner_product_monotonicity_identity},
\ba{
    \E\bbb{\Tr\bb{\B_t\B_t^T\vp\vp^T}} &\leq \bb{1 + 2\eta_{t}\lambda_{2} + \eta_{t-k_{t}}^{2}\bb{\bb{\frac{1 + \bb{3 + 4\epsilon}|\lambda_{2}\bb{P}|}{1 - \left|\lambda_{2}\bb{P}\right|}}\mathcal{V} + \lambda_{2}^{2} + \xi_{k,t}}}\beta_{t-1} \notag \\ 
    & + \eta_{t-k_{t}}^{2}\bb{\bb{\frac{1 + \bb{3 + 4\epsilon}|\lambda_{2}\bb{P}|}{1 - \left|\lambda_{2}\bb{P}\right|}}\mathcal{V} + \xi_{k,t}}\alpha_{t-1} \label{eq:beta_recursion_1}
}
Let $\chi_{\epsilon} := 1 + 4\epsilon\bb{1+\epsilon}\bb{1 + \epsilon + \epsilon^{2}} \leq 1.05$. From Theorem \ref{theorem:v1upperbound} denoting
\ba{
    r_{k,t} := 1 + 4\bb{1+\epsilon}\eta_{t-1}k_{t-1}\left(\mathcal{M}+\lambda_{1}\right) + 4\bb{1+c}^{2}\eta_{t-1}^{2}k_{t-1}^{2}\left(\mathcal{M}+\lambda_{1}\right)^{2} \leq 1 + 4\epsilon\bb{1+\epsilon}\bb{1 + \epsilon + \epsilon^{2}} = \chi_{\epsilon}, \label{r_kt_bound}
}we have, 
\bas {
    \alpha_{t-1} \leq r_{k,t}\exp\left(2\lambda_{1}\sum_{i=1}^{t-k_{t}-1}\eta_{t} + \sum_{i=1}^{t-k_{t}-1}\eta_{i}^{2}\bb{\bb{\frac{1 + \bb{3 + 4\epsilon}|\lambda_{2}\bb{P}|}{1 - \left|\lambda_{2}\bb{P}\right|}}\mathcal{V} + \lambda_{1}^{2} + \xi_{k,i}}\right)
}
Now, we note the definition of $\overline{\mathcal{V}_{k,t}}$ and $\mathcal{V}'$ as mentioned in \ref{assumptions:important_theorems} - 
\bas{
    \overline{\mathcal{V}_{k,t}} &:= \bb{\frac{1 + \bb{3 + 4\epsilon}|\lambda_{2}\bb{P}|}{1 - \left|\lambda_{2}\bb{P}\right|}}\mathcal{V} + \lambda_{1}^{2} + \xi_{k,t} \\ 
    &= \mathcal{V}' + \lambda_{1}^{2} + \xi_{k,t}
}
Therefore using \ref{eq:beta_recursion_1}, 
\bas {
    \beta_{t} &\leq \bb{1 + 2\eta_{t}\lambda_{2} + \eta_{t-k_{t}}^{2}\overline{\mathcal{V}_{k,t}}}\beta_{t-1} + \eta_{t-k_{t}}^{2}r_{k,t}\bb{\mathcal{V}' + \xi_{k,t}}\exp\left(2\lambda_{1}\sum_{i=1}^{t-k_{t}-1}\eta_{i} + \sum_{i=1}^{t-k_{t}-1}\eta_{i}^{2}\overline{\mathcal{V}_{k,i}}\right)
}
Recursing on the above inequality for $\init < t \leq n$ where $\init = \min\left\{i : i \in [n], i-k_{i} \geq 0 \right\}$, we have, 
\bas{
\beta_n &\leq \beta_{\init}\exp\left(2\sum_{i=\init+1}^n\eta_i\lambda_2 + \sum_{i=\init+1}^{n}\overline{\mathcal{V}_{k,i}}\eta_{i-k_{i}}^2\right) \\
& \;\;\;\;\;\; + \sum_{i=\init+1}^{n}r_{k,i}\bb{\mathcal{V}'+\xi_{k,i}}\eta_{i-k_{i}}^2\exp\bb{\sum_{j=i+1}^{n}\left(2\eta_j\lambda_2+\overline{\mathcal{V}_{k,j}}\eta_{j-k_{j}}^2\right)}\exp\bb{\sum_{j=1}^{i-k_{i}}2\eta_j \lambda_1+\overline{\mathcal{V}_{k,j}}\eta_j^2} \\
&\leq \exp\left(\sum_{i=\init+1}^n 2\eta_i\lambda_2 + \overline{\mathcal{V}_{k,i}}\eta_{i-k_{i}}^2\right)   \\
& \;\;\;\;\;\; \times \left(\beta_{\init} + \sum_{i=\init+1}^{n}r_{k,i}\bb{\mathcal{V}'+\xi_{k,i}}\eta_{i-k_{i}}^{2}\exp\left(\sum_{j=1}^{i-k_{i}}(2\eta_j \lambda_1+\overline{\mathcal{V}_{k,j}}\eta_j^2) - \sum_{j=\init+1}^{i}\bb{2\eta_j \lambda_{2}+\overline{\mathcal{V}_{k,j}}\eta_{j-k_{j}}^2}\right)\right)
}
Now, since $k_{i}, k_{j} \geq k_{\init} = \init$, therefore, we have
\bas{
\beta_n &\leq \exp\left(\sum_{i=\init+1}^n 2\eta_i\lambda_2+\overline{\mathcal{V}_{k,i}}\eta_{i-k_{i}}^2\right) \times  \\
& \;\;\;\;\;\;\; \left(\beta_{\init} + \sum_{i=\init+1}^{n}r_{k,i}\bb{\mathcal{V}'+\xi_{k,i}}\eta_{i-k_{i}}^{2}\exp\left(\sum_{j=1}^{i-\init}(2\eta_j \lambda_1+\overline{\mathcal{V}_{k,j}}\eta_j^2) - \sum_{j=\init+1}^{i}\bb{2\eta_j \lambda_{2}+\overline{\mathcal{V}_{k,j}}\eta_{j-\init}^2}\right)\right)
}
Recall that $C_{k,i}' := \exp\bb{2\lambda_{1}\sum_{j=1}^{\init}\bb{\eta_{j}-\eta_{i-\init+j}}+\sum_{j=1}^{i-\init}\eta_{j}^{2}\bb{\overline{\mathcal{V}_{k,j}}-\overline{\mathcal{V}_{k,j+\init}}}}$ as defined in \ref{assumptions:important_theorems}. Therefore, 
\bas{
\beta_n &\leq \exp\left(\sum_{i=\init+1}^n 2\eta_i\lambda_2+\overline{\mathcal{V}_{k,i}}\eta_{i-k_{i}}^2\right) \times \\ 
& \;\;\;\;\;\;\;\;\;\;\;\;\;\;\;\; \left(\beta_{\init} + \sum_{i=\init+1}^{n}r_{k,i}\bb{\mathcal{V}'+\xi_{k,i}}C_{k,i}'\eta_{i-k_{i}}^{2}\exp\left(\sum_{j=\init+1}^{i}2\eta_j\left( \lambda_1-\lambda_{2}\right)\right)\right)
}

Let $\B_{\init} = I + R'$ with $\|R'\| \leq r'$ a.s. Using Lemma \ref{lemma:etakproduct} we have 
\bas{
r'  & \leq \bb{1+\epsilon}k_{u}\eta_{1}\left(\mathcal{M}+\lambda_{1}\right) \\
    & \leq \bb{1+\epsilon}k_{u}\eta_{0}\left(\mathcal{M}+\lambda_{1}\right) \\
    & \leq 2\bb{1+\epsilon}k_{u}\eta_{u}\left(\mathcal{M}+\lambda_{1}\right) \;\; \text{since } \eta_{0} = \eta_{u-k_{u}} \leq 2\eta_{u} \\
    &< 2\epsilon\bb{1+\epsilon} 
}
Therefore, 
\begin{align*}
    \beta_{\init} & = \mathbb{E}\left[\Tr\left(V_{\perp}^{T}\B_{\init}\B_{\init}^{T}V_{\perp}\right)\right] \\
    &= \mathbb{E}\left[\Tr\left(V_{\perp}^{T}V_{\perp}\right)\right] + \mathbb{E}\left[\Tr\left(V_{\perp}^{T}(R' + R^{\prime T})V_{\perp}\right)\right] + \mathbb{E}\left[\Tr\left(V_{\perp}^{T}R'R^{\prime T}V_{\perp}\right)\right] \\
    & \leq d\left(1 + 2r' + r^{\prime 2}\right) \\ 
    &\leq d\bb{1 + 4\epsilon\bb{1+\epsilon} + 4\epsilon^{2}\bb{1+\epsilon}^{2}} \\ 
    &= d\bb{1 + 4\epsilon\bb{1+\epsilon}\bb{1 + \epsilon + \epsilon^{2}}} \\ 
    &= \chi_{\epsilon}d
\end{align*}
The proof follows by noting that $r_{k,t} \leq \chi_{\epsilon}$ as shown in \ref{r_kt_bound}.
\end{proof}

\begin{theorem} (General Version)
\label{theorem:appendix:v1lowerbound}
Under Assumptions~\ref{ass:model},~\ref{assumption:variance_bound} and~\ref{assumption:norm_bound}, for all $n>k_n$, any decaying step-size $\eta_i$ satisfying~\textbf{C.1} and~\textbf{C.2},  we have:
\begin{align*}
& \mathbb{E}\left[v_{1}^{T}\newB{1}{n}\newB{1}{n}^{T}v_{1}\right]  \geq \bb{1-t}\exp\bb{\sum_{i=1}^{n-k_{n}}2\eta_{i}\lambda_{1} - \sum_{i=1}^{n-k_{n}}4\eta_{i}^{2}\lambda_{1}^{2}}
\end{align*}
where $t := 2r + s, s := 3\bb{1+r}^{2}\exp\bb{2\lambda_{1}^{2}\sum_{i=1}^{n}\eta_{i}^{2}}\sum_{t=1}^{n-k_{n}} W_{k,t}\eta_{t}^2\exp\left(\sum_{i=t+1}^{n-k_{n}}\eta_{i}^{2}\right),  W_{k,t} := \mathcal{V}' + \xi_{k,t}$
and $\newB{i}{j}$ has been defined in \ref{definition:Bji}.
\end{theorem}

\begin{proof}
We will start will expanding the quantity of interest using Eq~\ref{eq:v1}.
\ba{
    \alpha_{n,t} = \mathbb{E}\left[v_{1}^{T}\newB{t}{n}\newB{t}{n}^{T}v_{1}\right] \geq \mathbb{E}\bbb{v_{1}^{T}B_{n,t+1}\left(I + \eta_{t}\Sigma\right)^{2}B_{n,t+1}^{T}v_{1} + 2\eta_{t}P_{n,t}} \label{eq:alpha_lb_1}
}
where $P_{n,t}$ has been defined in Theorem \ref{theorem::appendix:v1upperbound}. Let's define 
\bas{
    & \mathcal{S}_{t} := \prod_{i=t}^{1}\left(I + \eta_{i}\Sigma\right)\prod_{i=1}^{t}\left(I + \eta_{i}\Sigma\right), \; \mathcal{S}_{0} = I \; \text{ and } \\
    & \delta_{n,t} := \mathbb{E}\left[v_1^T\newB{t+1}{n}\mathcal{S}_{t}\newB{t+1}{n}^{T} v_1\right]
}
Note that  $\delta_{n,0} = \alpha_{n,1}$.  First we bound $\delta_{n,n-k_{n}}$. Let $\newB{n-k_{n}}{n}=I+R'$. By Lemma~\ref{lemma:etakproduct} along with the slow-decay assumption on the step-sizes, we know that $\|R'\|_2\leq r := 2\bb{1+\epsilon}\eta_{n}k_{n}\bb{\mathcal{M}+\lambda_{1}}$ a.s. Then,
\bas{
& \delta_{n,n-k_{n}} - \prod_{i=1}^{n-k_{n}}\left(1+\eta_{i}\lambda_{1}\right)^{2} \geq -2\left|E[v_1^TR'\mathcal{S}_{n-k_{n}}v_1]\right| \geq -2r\prod_{i=1}^{n-k_{n}}(1+\eta_i\lambda_1)^2
}
Therefore,
\ba{
    \delta_{n,n-k_{n}} &\geq \prod_{i=1}^{n-k_{n}}\left(1+\eta_{i}\lambda_{1}\right)^{2}(1-2r) \notag \\ 
                       &= (1-2r)\left\|\mathcal{S}_{n-k_{n}}\right\|_{2} \label{eq:delta_n_minus_kn_bound}
}
Now using \ref{eq:alpha_lb_1}, we have 
\begin{align*}
    \delta_{n,t-1} &\geq \delta_{n,t} + 2\eta_{t}\mathbb{E}\left[\underbrace{v_{1}^{T}\newB{t+1}{n}\left(I + \eta_{t}\Sigma\right)\mathcal{S}_{t-1}\left(A_{t}-\Sigma\right)\newB{t+1}{n}^{T}v_{1}}_{U_{t}}\right]
\end{align*}
First, observe that $\mathcal{S}_{t-1}=U\Lambda U^T$, where $U$ denotes a matrix of eigenvectors of $\Sigma$, and $\Lambda$ is a PSD diagonal matrix. Since $I+\eta_{t}\Sigma=U\Lambda'U^T$ for some other PSD diagonal matrix $\Lambda'$, the product will also be PSD.
\\ \\
By using Lemma \ref{lemma:trace_linear_bound} with $U = v_{1}, G = \bb{I + \eta_{t}\Sigma}\mathcal{S}_{t-1}, \gamma = 1$ and noting that $\E_{\pi}\bbb{A_{t}-\Sigma} = 0$, we have
\begin{align*}
    \left|\mathbb{E}\left[U_{t}\right]\right| 
    &\leq \bb{1+\eta_{t}\lambda_{1}}\eta_{t+1}\left\|\mathcal{S}_{t-1}\right\|_{2}\left(\frac{2\mathcal{V}\left|\lambda_{2}\bb{P}\right|}{1 - \left|\lambda_{2}\bb{P}\right|} \right.\\
    &\left.+ \eta_{t+1}\mathcal{M}\bb{2\bb{1+8\epsilon} + \bb{2+\bb{1+\epsilon}^{2}}k_{t+1}^{2}\bb{\mathcal{M}+\lambda_{1}}^{2}}\right)\alpha_{n,t+k_{t+1}} \\ 
    &\leq \bb{1+\epsilon}\eta_{t+1}\left\|\mathcal{S}_{t-1}\right\|_{2}W_{k,t}\alpha_{n,t+1}
\end{align*}
where $W_{k,t} = \mathcal{V}' + \xi_{k,t}$. Therefore, 
\bas{
    \delta_{n,t-1} &\geq  \delta_{n,t} - 2\bb{1+\epsilon}W_{k,t}\eta_{t}^2\alpha_{n,t+1}\left\|\mathcal{S}_{t-1}\right\|_{2} \text{ for $t\leq n-k_{n}$}
}
Let 
\bas{
    \mathcal{V}' := \bb{\frac{1 + \bb{3 + 4\epsilon}|\lambda_{2}\bb{P}|}{1 - \left|\lambda_{2}\bb{P}\right|}}\mathcal{V}
}
as defined in \ref{assumptions:important_theorems}. Unwinding the recursion for $t\leq n-k_{n}$, we have,
\bas{
\delta_{n,0} &\geq \delta_{n,n-k_{n}} - 2\bb{1+\epsilon}\sum_{t=1}^{n-k_{n}} W_{k,t}\eta_{t}^2\alpha_{n,t+1}\left\|\mathcal{S}_{t-1}\right\|_{2} \\
&\geq (1-2r)\left\|\mathcal{S}_{n-k_{n}}\right\|_{2} \\
&- 2\bb{1+\epsilon}\bb{1 + r}^{2}\sum_{t=1}^{n-k_{n}} W_{k,t}\eta_{t}^2\exp\left(2\lambda_{1}\sum_{i=t+1}^{n-k_{n}}\eta_{i} + \sum_{i=t+1}^{n-k_{n}}\eta_{i}^{2}\bb{\mathcal{V}' + \lambda_{1}^{2} + C_{k,i}}\right)\left\|\mathcal{S}_{t-1}\right\|_{2}
}
where second step followed from Theorem \ref{theorem:v1upperbound} and \ref{eq:delta_n_minus_kn_bound}. \\ \\ 
Using the inequalities $\forall x \in \mathbb{R}, \; 1+x \leq e^{x}$ and $\forall x \in \mathbb{R}, \; x \geq 0,  1+x \geq e^{x - x^{2}}$, \;\; $\forall t$ we have, 
\bas{
    & \left\|\mathcal{S}_{t}\right\|_{2} = \prod_{i=1}^{t}(1+\eta_{i}\lambda_{1})^{2} \leq \exp\bb{2\lambda_{1}\sum_{i=1}^{t}\eta_{i}} \text{, and } \\ 
    & \left\|\mathcal{S}_{t}\right\|_{2} = \prod_{i=1}^{t}(1+\eta_{i}\lambda_{1})^{2} \geq \exp\bb{2\lambda_{1}\sum_{i=1}^{t}\eta_{i} - 4\lambda_{1}^{2}\sum_{i=1}^{t}\eta_{i}^{2}} 
}
Therefore denoting $\theta_{\epsilon} := 2\bb{1+\epsilon}\exp\bb{2\lambda_{1}^{2}\sum_{i=1}^{n}\eta_{i}^{2}}$, we have 
\bas{
    &\delta_{n,0} \\
    &\geq \exp\bb{2\lambda_{1}\sum_{i=1}^{n-k_{n}}\eta_{i} - 4\lambda_{1}^{2}\sum_{i=1}^{n-k_{n}}\eta_{i}^{2}}\bbb{\bb{1-2r} - \theta_{\epsilon}\bb{1+r}^{2}\sum_{t=1}^{n-k_{n}} W_{k,t}\eta_{t}^2\exp\left(\sum_{i=t+1}^{n-k_{n}}\eta_{i}^{2}\bb{\mathcal{V}' + \lambda_{1}^{2} + C_{k,i}}\right)} \\ 
    &\geq \exp\bb{2\lambda_{1}\sum_{i=1}^{n-k_{n}}\eta_{i} - 4\lambda_{1}^{2}\sum_{i=1}^{n-k_{n}}\eta_{i}^{2}}\bbb{\bb{1-2r} - \theta_{\epsilon}\bb{1+r}^{2}\sum_{t=1}^{n-k_{n}} W_{k,t}\eta_{t}^2\exp\left(\sum_{i=t+1}^{n-k_{n}}\eta_{i}^{2}\bb{\mathcal{V}' + \lambda_{1}^{2} + C_{k,i}}\right)} \\ 
    &\geq \exp\bb{2\lambda_{1}\sum_{i=1}^{n-k_{n}}\eta_{i} - 4\lambda_{1}^{2}\sum_{i=1}^{n-k_{n}}\eta_{i}^{2}}\bbb{1 - \bb{2r + \theta_{\epsilon}\bb{1+r}^{2}\sum_{t=1}^{n-k_{n}} W_{k,t}\eta_{t}^2\exp\left(\sum_{i=t+1}^{n-k_{n}}\eta_{i}^{2}\overline{\mathcal{V}_{k,i}}\right)}}
}
where $\overline{\mathcal{V}_{k,i}}$ is defined in \ref{assumptions:important_theorems}. Hence proved.
\end{proof}

\begin{theorem} (General Version)
\label{theorem:appendix:v1squareupperbound}
Under Assumptions~\ref{ass:model},~\ref{assumption:variance_bound} and~\ref{assumption:norm_bound}, for all $n>k_n$, and decaying step-size $\eta_i$ satisfying~\textbf{C.1} and~\textbf{C.2}, we have:
\begin{align*}
& \mathbb{E}\left[\left(v_{1}^{T}\newB{1}{n}\newB{1}{n}^{T}v_{1}\right)^{2}\right] \leq \bb{1+r}^{4}\exp\left(\sum_{i=1}^{n-k_{n}}4\eta_{i}\lambda_{1} + \sum_{i=1}^{n-k_{n}}\eta_{i}^{2}  \zeta_{k,i} \right)
\end{align*}
where 
$\newB{i}{j}$ has been defined in \ref{definition:Bji}.
\end{theorem} 

\begin{proof}
Define $Q_{n,t} := v_1^T\newB{t+1}{n}(A_t-\Sigma)^2\newB{t+1}{n}^{T} v_1$, and $P_{n,t}:=v_1^T\newB{t+1}{n}(I+\eta_t\Sigma)(A_t-\Sigma)\newB{t+1}{n}^{T} v_1$.
Using \ref{eq:v1}, we have, for $n \geq t \geq 1$,
\bas{
0\leq v_1^T\newB{t}{n}\newB{t}{n}^{T}v_1&=v_1^T\newB{t+1}{n}(I+\eta_t\Sigma)^2\newB{t+1}{n}^{T} v_1+\eta_t^2Q_{n,t}+2\eta_tP_{n,t}\\&\leq v_1^T\newB{t+1}{j}\newB{t+1}{j}^{T} v_1(1+\eta_{t}\lambda_1)^2 + \eta_{t}^2\mathcal{M}^{2}\bb{v_1^T\newB{t+1}{n}\newB{t+1}{n}^{T} v_1} + 2\eta_{t}P_{n,t}\\
&\leq v_1^T\newB{t+1}{j}\newB{t+1}{j}^{T} v_1\underbrace{\left((1+\eta_{t}\lambda_1)^2+\eta_{t}^2\mathcal{M}^{2}\right)}_{c_{t}}+2\eta_{t}P_{n,t}
}
Thus, we have -
\ba{
\kappa_{n,t} := \mathbb{E}\left[(v_1^T\newB{t}{n}\newB{t}{n}^{T}v_1)^2\right] &\leq \E\bbb{\bb{c_tv_1^T\newB{t+1}{n}\newB{t+1}{n}^{T} v_1 + 2\eta_{t}P_{n,t}}^{2}} \notag \\
&\leq c_t^2 \kappa_{n,t+1}+4\eta_t^2\mathbb{E}\left[P_{n,t}^{2}\right]+4c_{t}\eta_{t}\mathbb{E}\left[\bb{v_1^T\newB{t+1}{n}\newB{t+1}{n}^{T}v_1}P_{n,t}\right] \label{eq:kappa_recursion}
}
Note that,
\bas{
\mathbb{E}\left[P_{n,t}^2\right] &\leq \mathbb{E}\bbb{\bb{v_1^T\newB{t+1}{n}(I+\eta_t\Sigma)(A_t-\Sigma)\newB{t+1}{n}^{T} v_1}^2}\\
&\leq (1+\eta_{t}\lambda_1)^2\mathcal{M}^{2}\E
\bbb{\bb{v_1^T\newB{t+1}{n}\newB{t+1}{n}^{T} v_1}^2}\\
&=(1+\eta_{t}\lambda_1)^2\mathcal{M}^{2}\kappa_{n,t+1}
}
Now we work on the cross-term. For the convenience of notation, let's denote $k:= k_{t+1}$ unless otherwise specified. Let $\newB{t+1}{n} = \newB{t+k}{n}\left(I + R\right)$ with,
\bas{
    \|R\|_{2} \leq (1+c)\eta_{t+1}k(\M+\lambda_{1})=:r_t \leq \epsilon\bb{1+\epsilon}
}
Using Lemma \ref{lemma:etakproduct}, we have
\ba{\label{eq:Y1}
|\underbrace{v_1^T\newB{t+1}{n}\newB{t+1}{n}^{T}v_1-v_1^T\newB{t+k}{n}\newB{t+k}{n}^{T}v_1}_{Y_1}| &= |v_1^T\newB{t+k}{n}(R+R^T+RR^T)\newB{t+k}{n}^{T}v_1|\notag\\
&\leq |v_{1}^{T}\newB{t+k}{n}\newB{t+k}{n}^{T}v_{1}|\left(2r_t + r_t^{2}\right)
}
We will also bound 
\ba{\label{eq:Y2}
&|\underbrace{v_1^T\newB{t+1}{n}(I+\eta_t\Sigma)(A_t-\Sigma)\newB{t+1}{n}^{T}v_1-v_1^T\newB{t+k}{n}(I+\eta_t\Sigma)(A_t-\Sigma)\newB{t+k}{n}^{T}v_1}_{Y_2}|\notag\\
&\;\;\;\; =|v_1^T\newB{t+k}{n}R(I+\eta_t\Sigma)(A_t-\Sigma)(I+R^{T})\newB{t+k}{n}^{T}v_1+
v_1^T\newB{t+k}{n}(I+\eta_t\Sigma)(A_t-\Sigma)R^{T}\newB{t+k}{n}^{T}v_1| \notag \\
&\;\;\;\; \leq (2r_t+r_t^2)\left(1+\eta_{t}\lambda_{1}\right)\mathcal{M}|v_{1}^{T}\newB{t+k}{n}\newB{t+k}{n}^{T}v_{1}|
}
So, now we have:
\bas{
&\mathbb{E}\left[\bb{v_1^T\newB{t+1}{n}\newB{t+1}{n}^{T}v_1P_{n,t}}\right] \\
&=\mathbb{E}\left[(v_1^T\newB{t+1}{n}\newB{t+1}{n}^{T}v_1)(v_1^T\newB{t+1}{n}(I+\eta_t\Sigma)(A_t-\Sigma)\newB{t+1}{n}^{T}v_1)\right]\\
&=\mathbb{E}\left[(Y_1+v_1^T\newB{t+k}{n}\newB{t+k}{n}^{T}v_1)(Y_2+v_1^T\newB{t+k}{n}(I+\eta_t\Sigma)(A_t-\Sigma)\newB{t+k}{n}^{T}v_1)\right]\\
&=\underbrace{\mathbb{E}\left[Y_1Y_2\right]}_{T_{1}}+\underbrace{\mathbb{E}\left[Y_1v_1^T\newB{t+k}{n}(I+\eta_t\Sigma)(A_t-\Sigma)\newB{t+k}{n}^{T}v_1\right]}_{T_{2}}+\underbrace{\mathbb{E}\left[Y_2v_1^T\newB{t+k}{n}\newB{t+k}{n}^{T}v_1\right]}_{T_{3}}\\
&\;\; +\underbrace{\mathbb{E}\left[(v_1^T\newB{t+k}{n}\newB{t+k}{n}^{T}v_1)(v_1^T\newB{t+k}{n}(I+\eta_t\Sigma)(A_t-\Sigma)\newB{t+k}{n}^{T}v_1)\right]}_{T_{4}}
}
Lets start with the last term, $T_{4}$. Using Lemma \ref{lemma:inner_product_monotonicity} we have,
\bas{
|T_{4}| &\leq \left|\mathbb{E}\left[(v_1^T\newB{t+k}{n}\newB{t+k}{n}^{T}v_1)(v_1^T\newB{t+k}{n}(I+\eta_t\Sigma)\mathbb{E}\left[(A_t-\Sigma)|s_{t+k}\right]\newB{t+k}{n}^{T}v_1)\right]\right| \\&\leq 2(1+\eta_t\lambda_1)\mathcal{M}d_{\text{mix}}\bb{k}\kappa_{n,t+k} \\
&\leq 2\eta_{t+1}^2(1+\eta_t\lambda_1)\mathcal{M}\kappa_{n,t+k} \\
&\leq 2\eta_{t+1}^2(1+\eta_t\lambda_1)\mathcal{M}\kappa_{n,t+1}
}
Using Eqs~\ref{eq:Y1} and~\ref{eq:Y2} the first three terms can be bounded as:
\bas{
|T_{1}| \leq \mathbb{E}\left[|Y_1Y_2|\right] &\leq \bb{2r_{t}+r_{t}^{2}}^{2}\left(1+\eta_{t}\lambda_{1}\right)\mathcal{M}\kappa_{n,t+k} \\
&\leq \bb{2r_{t}+r_{t}^{2}}^{2}\left(1+\eta_{t}\lambda_{1}\right)\mathcal{M}\kappa_{n,t+1} \text{ using Lemma } \ref{lemma:inner_product_monotonicity} \\
&= \bb{2+r_{t}}^{2}r_{t}^{2}\left(1+\eta_{t}\lambda_{1}\right)\mathcal{M}\kappa_{n,t+1} \\
&\leq \bb{1+\epsilon}^{2}\bb{2+\epsilon\bb{1+\epsilon}}^{2}\left(1+\eta_{t}\lambda_{1}\right)\eta_{t+1}^{2}k_{t+1}^{2}\mathcal{M}\bb{\mathcal{M}+\lambda_{1}}^{2}\kappa_{n,t+1} \\
&\leq \bb{1+\epsilon}^{3}\bb{2+\epsilon+\epsilon^{2}}^{2}\eta_{t+1}^{2}k_{t+1}^{2}\mathcal{M}\bb{\mathcal{M}+\lambda_{1}}^{2}\kappa_{n,t+1} \text{ since } \eta_{t}\lambda_{1} \leq \epsilon
}
\bas{
|T_{2}| &\leq \mathbb{E}\left[|Y_1v_1^T\newB{t+k}{n}(I+\eta_t\Sigma)(A_t-\Sigma)\newB{t+k}{n}^{T}v_1|\right] \\&\leq \bb{2+r_{t}}r_{t}\bb{1+\eta_{t}\lambda_{1}}\mathcal{M}\kappa_{n,t+k} \\
&\leq \bb{2+r_{t}}r_{t}\bb{1+\eta_{t}\lambda_{1}}\mathcal{M}\kappa_{n,t+1} \text{ using Lemma } \ref{lemma:inner_product_monotonicity} \\
&\leq \bb{2+\epsilon + \epsilon^{2}}\bb{1+\epsilon}\bb{1+\eta_{t}\lambda_{1}}\eta_{t+1}k_{t+1}\bb{\mathcal{M}+\lambda_{1}}\mathcal{M}\kappa_{n,t+1} \\
&\leq \bb{1+\epsilon}^{2}\bb{2+\epsilon + \epsilon^{2}}\eta_{t+1}k_{t+1}\mathcal{M}\bb{\mathcal{M}+\lambda_{1}}\kappa_{n,t+1}
}and similarly,
\bas{
|T_{3}| &\leq \mathbb{E}\left[Y_2v_1^T\newB{t+k}{n}\newB{t+k}{n}^{T}v_1\right] \\
&\leq r_{t}\bb{2+r_{t}}\bb{1+\eta_{t}\lambda_{1}}\mathcal{M}\kappa_{n,t+k} \\
&\leq \bb{1+\epsilon}\bb{2+\epsilon+\epsilon^{2}}\bb{1+\eta_{t}\lambda_{1}}\eta_{t+1}k_{t+1}\mathcal{M}\bb{\mathcal{M}+\lambda_{1}}\kappa_{n,t+k} \\
&\leq \bb{1+\epsilon}^{2}\bb{2+\epsilon+\epsilon^{2}}\eta_{t+1}k_{t+1}\mathcal{M}\bb{\mathcal{M}+\lambda_{1}}\kappa_{n,t+k} \\ 
&\leq \bb{1+\epsilon}^{2}\bb{2+\epsilon+\epsilon^{2}}\eta_{t+1}k_{t+1}\mathcal{M}\bb{\mathcal{M}+\lambda_{1}}\kappa_{n,t+1} \text{ using Lemma } \ref{lemma:inner_product_monotonicity}
}
Note that
\bas {
    c_{t} &:= (1+\eta_{t}\lambda_1)^2+\eta_{t}^2\mathcal{M}^{2} \leq 1+2\epsilon+2\epsilon^2 \text{, and } \\ 
    c_{t}^{2} &= \bb{1 + 2\eta_{t}\lambda_{1} + \eta_{t}^{2}\bb{\mathcal{M}^{2}+\lambda_{1}^{2}}}^{2} \\ 
    &= 1 + 4\eta_{t}^{2}\lambda_{1}^{2} + \eta_{t}^{4}\bb{\mathcal{M}^{2}+\lambda_{1}^{2}}^{2} + 4\eta_{t}\lambda_{1} + 4\eta_{t}^{3}\lambda_{1}\bb{\mathcal{M}^{2}+\lambda_{1}^{2}} + 2\eta_{t}^{2}\bb{\mathcal{M}^{2}+\lambda_{1}^{2}} \\ 
    &\leq 1 + 4\eta_{t}\lambda_{1} + \eta_{t}^{2}\bb{2\mathcal{M}^{2}+6\lambda_{1}^{2}+\epsilon\mathcal{M}+\epsilon\lambda_{1}} + 4\eta_{t}^{3}\lambda_{1}\bb{\mathcal{M}^{2}+\lambda_{1}^{2}} \\ 
    &\leq 1 + 4\eta_{t}\lambda_{1} + 6\eta_{t}^{2}\bb{\mathcal{M}+\lambda_{1}}^{2} + 4\eta_{t}^{3}\lambda_{1}\bb{\mathcal{M}+\lambda_{1}}^{2}
}
Define 
\bas{
    \phi_{\epsilon} &:= \bb{1+\epsilon}\bb{2+\epsilon+\epsilon^{2}} \\ 
    \omega_{\epsilon} &:= 1+2\epsilon+2\epsilon^{2} \\ 
    \zeta_{k,t} &:= \bb{10 + 8\bb{1 + \epsilon} + 4\bb{1 + 2\epsilon}\phi_{\epsilon}}\phi_{\epsilon}c_{t}k_{t+1}\bb{\mathcal{M}+\lambda_{1}}^{2}
}
Putting everything together in Eq~\ref{eq:kappa_recursion}, for $t \leq n-k_{t+1}$ we have, 
\bas{
    &\frac{\kappa_{n,t}}{\kappa_{n,t+1}}\\
    &\leq c_{t}^{2} + 4\eta_{t}^{2}\bb{1+\eta_{t}\lambda_{1}}^{2}\mathcal{M}^{2} + 4\bb{1+\epsilon}c_{t}\eta_{t}\mathcal{M}\bb{2\phi_{\epsilon}\eta_{t+1}k_{t+1}\bb{\mathcal{M}+\lambda_{1}} + \bb{2 + \phi_{\epsilon}^{2}k_{t+1}^{2}\bb{\mathcal{M}+\lambda_{1}}^{2}}\eta_{t+1}^{2}} \\
    &\leq c_{t}^{2} + 4\eta_{t}^{2}\bb{1+\eta_{t}\lambda_{1}}^{2}\mathcal{M}^{2} + 4\bb{1 + \epsilon}c_{t}\eta_{t}\mathcal{M}\bb{2\phi_{\epsilon}\eta_{t}k_{t+1}\bb{\mathcal{M}+\lambda_{1}} + \bb{2 + \phi_{\epsilon}^{2}k_{t+1}^{2}\bb{\mathcal{M}+\lambda_{1}}^{2}}\eta_{t}^{2}} \\
    &= c_{t}^{2} + 4\eta_{t}^{2}\bbb{\mathcal{M}^{2} + 2\phi_{\epsilon}\bb{1 + \epsilon}c_{t}\mathcal{M}\bb{\mathcal{M}+\lambda_{1}}k_{t+1}} + 4\eta_{t}^{3}\bbb{\bb{1 + 2\epsilon}\lambda_{1} + \bb{1 + \epsilon}c_{t}\mathcal{M}\bb{2 + \phi_{\epsilon}^{2}k_{t+1}^{2}\bb{\mathcal{M}+\lambda_{1}}^{2}}} \\ 
    &\leq c_{t}^{2} + 4\eta_{t}^{2}\bbb{2 + 2\phi_{\epsilon}\bb{1 + \epsilon}c_{t}k_{t+1}}\mathcal{M}\bb{\mathcal{M}+\lambda_{1}} + 4\bb{1 + 2\epsilon}\eta_{t}^{3}\bbb{\lambda_{1} + c_{t}\mathcal{M}\bb{2 + \phi_{\epsilon}^{2}k_{t+1}^{2}\bb{\mathcal{M}+\lambda_{1}}^{2}}} \\ 
    &\leq 1 + 4\eta_{t}\lambda_{1} + \eta_{t}^{2}\bbb{10 + 8\phi_{\epsilon}\bb{2 + \epsilon}c_{t}k_{t+1}}\bb{\mathcal{M}+\lambda_{1}}^{2} + 4\bb{1 + 2\epsilon}\eta_{t}^{3}\bbb{\lambda_{1}+ 2c_{t}\mathcal{M} + c_{t} \phi_{\epsilon}^{2}k_{t+1}^{2}\bb{\mathcal{M}+\lambda_{1}}^{3}}
}
\bas{
    &\leq \exp\bb{4\eta_{t}\lambda_{1} + \eta_{t}^{2}\bb{10 + 8\phi_{\epsilon}\bb{1 + \epsilon}c_{t}k_{t+1}}\bb{\mathcal{M}+\lambda_{1}}^{2} + 4\bb{1 + 2\epsilon}\eta_{t}^{3}\bb{\lambda_{1}+ c_{t}\mathcal{M} + 2c_{t} \phi_{\epsilon}^{2}k_{t+1}^{2}\bb{\mathcal{M}+\lambda_{1}}^{3}}} \\
    &\leq \exp\bb{4\eta_{t}\lambda_{1} + \eta_{t}^{2}\bb{10 + 8\phi_{\epsilon}\bb{1 + \epsilon}c_{t}k_{t+1}}\bb{\mathcal{M}+\lambda_{1}}^{2} + 4\epsilon\bb{1 + 2\epsilon}\eta_{t}^{2}\bb{2 c_{t} + c_{t} \phi_{\epsilon}^{2}k_{t+1}\bb{\mathcal{M}+\lambda_{1}}^{2}}} \\
    &\leq \exp\bb{4\eta_{t}\lambda_{1} + \eta_{t}^{2}\bb{8\epsilon\bb{1 + 2\epsilon}\omega_{\epsilon} + \bb{10 + \bb{8\bb{1 + \epsilon} + 4\epsilon\bb{1 + 2\epsilon}\phi_{\epsilon}}\phi_{\epsilon}\omega_{\epsilon}k_{t+1}}\bb{\mathcal{M}+\lambda_{1}}^{2}}} \\ 
    &\leq \exp\bb{4\eta_{t}\lambda_{1} + \eta_{t}^{2}\bb{1 + \bb{10 + 20k_{t+1}}\bb{\mathcal{M}+\lambda_{1}}^{2}}} \\ 
    &\leq \exp\bb{4\eta_{t}\lambda_{1} + \eta_{t}^{2}\bb{1 + \bb{10 + 20k_{t+1}}\bb{\mathcal{M}+\lambda_{1}}^{2}}} \\ 
    &\leq \exp\bb{4\eta_{t}\lambda_{1} + 40\eta_{t}^{2}k_{t+1}\bb{\mathcal{M}+\lambda_{1}}^{2}} \text{ since } \bb{\mathcal{M}+\lambda_{1}}, k_{t+1} \geq 1 
}
Recall our definition of $k := k_{t+1}$. We can use the above recursion for $1 \leq t \leq n - k_{t+1}$. We note that $t = n-k_{n}$ satisfies the conditions. Therefore, 
\bas{
\kappa_{n,1} &\leq \exp\left(\sum_{i=1}^{n-k_{n}}4\eta_{i}\lambda_{1} + \sum_{i=1}^{n-k_{n}}\eta_{i}^{2}\zeta_{k,i} \right)\kappa_{n,n-k_{n}+1}
}
Let $B_{n,n-k_{n}+1} = I + R'$,  with $\|R'\|_2\leq r$ a.s.
\begin{align*}
    \kappa_{n,n-k_{n}+1} & = \mathbb{E}\left[\left(v_{1}^{T}B_{n,n-k_{n}+1}B_{n,n-k_{n}+1}^{T}v_1\right)^{2}\right] \\
    &= \mathbb{E}\left[\left(v_{1}^{T}v_{1} + v_{1}^{T}(R' + R^{\prime T})v_{1} + v_{1}^{T}R'R^{\prime T}v_{1}\right)^{2}\right] \\
    &\leq \left(1+ 2r+r^{ 2}\right)^{2}\mathbb{E}\left[\left(v_{1}^{T}v_{1}\right)^{2}\right]
\end{align*}
Using Lemma \ref{lemma:etakproduct}, we have 
\bas{
   r   & \leq \bb{1+\epsilon}k_{n+1}\eta_{n-k_{n+1}}\left(\mathcal{M}+\lambda_{1}\right) \\
         & \leq \bb{1+\epsilon}k_{n}\eta_{n-k_{n}}\left(\mathcal{M}+\lambda_{1}\right) \\
         & \leq 2\bb{1+\epsilon}k_{n}\eta_{n}\left(\mathcal{M}+\lambda_{1}\right) \;\; \text{since } \eta_{n-k_{n}} \leq 2\eta_{n}
}
which completes our proof.
\end{proof}

\section{Main Results : Details and Proofs}
\label{appendix:proofs_of_main_results}

\subsection{Proof of Theorem \ref{theorem:oja_convergence_rate}}
\begin{lemma}
\label{lemma:additional_assumptions}
This lemma proves conditions required later in the proof. 
Let the step-sizes be set according to Lemma \ref{lemma:learning_rate_schedule} and $m := 200$. Define
\bas{
    & r := 2\bb{1+\epsilon}\eta_{n}k_{n}\left(\mathcal{M}+\lambda_{1}\right), \\ 
    & s := 3(1+r)^{2}\sum_{t=1}^{n-k_{n}-1} W_{k,t}\eta_{t}^2\exp\left( \sum_{i=t+1}^{n-k_{n}-1}\overline{\mathcal{V}_{k,i}}\eta_i^2\right) 
}
where $W_{k,t}$ is defined in Theorem \ref{theorem:v1lowerbound}, $\overline{\mathcal{V}_{k,i}}$ is defined in \ref{assumptions:important_theorems} and $\alpha, \beta, f\bb{.}, \delta$ are defined in Lemma \ref{lemma:learning_rate_schedule}. Then for sufficiently large number of samples $n$,  such that 
\bas{
    \frac{n}{\log\bb{f\bb{n}}} > \frac{\beta}{\log\bb{f\bb{0}}}
} we have
\begin{enumerate}
    \item $2r + s \leq \frac{1}{2}$ (\ref{assumption:lower_bound_n_1})
    \item $r = 2\bb{1+\epsilon}\eta_{n}k_{n}\left(\mathcal{M}+\lambda_{1}\right) <\frac{1}{50}\frac{\delta/m}{1+\delta/m}$ (\ref{assumption:r_delta_condition})
\end{enumerate}
\end{lemma}
\begin{proof}
    For (1), using Lemma \ref{lemma:learning_rate_schedule}-(3), we note that 
    \ba{
        s &\leq 3(1+r)^{2}\sum_{t=1}^{n-k_{n}-1} W_{k,t}\eta_{t}^2\exp\left(\sum_{i=t+1}^{n-k_{n}-1}\overline{\mathcal{V}_{k,i}}\eta_i^2\right) \notag \\
        &\leq 3(1+r)^{2}\sum_{t=1}^{n-k_{n}-1} W_{k,t}\eta_{t}^2\bb{1+\frac{\delta}{m}} \notag \\ 
        &\leq \frac{3(1+r)^{2}}{100}\bb{1+\frac{\delta}{m}}\log\bb{1+\frac{\delta}{m}} \label{eq:epsilon_delta_bound} \\ 
        &\leq \frac{3(1+r)^{2}\log\bb{2}}{50} \text{ since } \frac{\delta}{m} < 1 \notag
    }
    Therefore, 
    \ba{
        2r + s &\leq 2r + \frac{3\bb{1+r}^{2}}{25} \notag \\ 
                      &= \frac{3}{25} + \frac{56}{25}r + \frac{3}{25}r^{2} \label{eq:2r_eps_bound}
    }
    Setting $\frac{3}{25} + \frac{56}{25}r + \frac{3}{25}r^{2} \leq \frac{1}{2}$, we have, 
    \bas{
        & \;\;\;\;\;\;\;\; \frac{3}{25} + \frac{56}{25}r + \frac{3}{25}r^{2} \leq \frac{1}{2} \\ 
        & \implies 6r^{2} + 112r - 19 \leq 0
    }
    which holds for $r \in \bbb{0,\frac{1}{10}}$. 
    \\ \\
    For $\bb{2}$, using Lemma \ref{lemma:learning_rate_schedule} and substituting the value of $k_{i} := \tau_{\text{mix}}\bb{\eta_{i}^{2}} \leq \frac{2\tau_{\text{mix}}}{\log\bb{2}}\log\bb{\frac{1}{\eta_{i}^{2}}}$ for $\eta_{i} < 1$, we note that
    \bas{
        r &\leq \frac{8\bb{1+\epsilon}\tau_{\text{mix}}\bb{\mathcal{M}+\lambda_{1}}}{\log\bb{2}}\frac{\alpha}{\bb{\lambda_{1}-\lambda_{2}}\bb{\beta + n}}\log\bb{\frac{\bb{\lambda_{1}-\lambda_{2}}\bb{\beta + n}}{\alpha}} \\ 
        &= \frac{8\bb{1+\epsilon}\tau_{\text{mix}}\bb{\mathcal{M}+\lambda_{1}}}{\log\bb{2}}\frac{\log\bb{\frac{\bb{\lambda_{1}-\lambda_{2}}\bb{\beta + n}}{\alpha}}}{\frac{\bb{\lambda_{1}-\lambda_{2}}\bb{\beta + n}}{\alpha}} \\ 
        &= \frac{8\bb{1+\epsilon}\tau_{\text{mix}}\bb{\mathcal{M}+\lambda_{1}}}{\log\bb{2}}\frac{\log\bb{f\bb{n}}}{f\bb{n}}
    }
    Therefore (2) holds for sufficiently large $n$, i.e,
    \ba{
        & \;\;\;\;\;\;\;\;\;\; \frac{f\bb{n}}{\log\bb{f\bb{n}}}  \geq \frac{400\bb{1 + \frac{\delta}{m}}\bb{1+\epsilon}\tau_{\text{mix}}\bb{\mathcal{M}+\lambda_{1}}}{\log\bb{2}\frac{\delta}{m}} \notag
    }
    This is satisfied if 
    \ba{
         \frac{n}{\log\bb{f\bb{n}}} \geq \frac{400\tau_{\text{mix}}\bb{1 + \frac{\delta}{m}}\bb{1+\epsilon}}{\log\bb{2}}\frac{\bb{\mathcal{M}+\lambda_{1}}\alpha}{\bb{\lambda_{1}-\lambda_{2}}\frac{\delta}{m}}
        \label{eq:lnn_n_condition}
    }
    From Lemma \ref{lemma:learning_rate_schedule}, we have 
    \bas{
        \frac{\beta}{\log\bb{f\bb{0}}} &\geq \frac{600\tau_{\text{mix}}\bb{1+2\epsilon}^{2}\bb{\mathcal{M}+\lambda_{1}}^{2}\alpha^{2}}{\bb{\lambda_{1}-\lambda_{2}}^{2}\log\bb{1+\frac{\delta}{m}}} \stackrel{\bb{i}}\geq \frac{400\tau_{\text{mix}}\bb{1 + \frac{\delta}{m}}\bb{1+\epsilon}}{\log\bb{2}}\frac{\bb{\mathcal{M}+\lambda_{1}}\alpha}{\bb{\lambda_{1}-\lambda_{2}}\frac{\delta}{m}}
    }
    where $\bb{i}$ follows since $\frac{\mathcal{M}+\lambda_{1}}{\lambda_{1}-\lambda_{2}} > 1, \alpha > 2$ and $\log\bb{1+x} \leq x  \; \forall x$. Therefore, $\frac{n}{\log\bb{f\bb{n}}} > \frac{\beta}{\log\bb{f\bb{0}}}$ suffices. Further, we note that (2) implies (1) for m = 200, $\delta \leq 1$. Therefore, the condition on $n$ is sufficient for both results. Hence proved.
\end{proof}

\begin{lemma}
\label{lemma:u_bound}
Let \bas{
    \init &:= \min\left\{i : i \in [n], i-k_{i} \geq 0 \right\}
}
where $k_{i}$ is defined in Lemma \ref{lemma:learning_rate_schedule}. Then, 
    \bas{u \leq \lfloor \beta \rfloor  \leq \beta}
\end{lemma}
\begin{proof}
    Using the definition of $k_{i}$ mentioned in Lemma \ref{lemma:learning_rate_schedule}, we have 
    \bas{
        k_{i} := \tau_{\text{mix}}\bb{\eta_{i}^{2}} &\leq \frac{2\tau_{\text{mix}}}{\log\bb{2}}\log\bb{\frac{1}{\eta_{i}^{2}}} \\ 
        &= \frac{4\tau_{\text{mix}}}{\log\bb{2}}\log\bb{\frac{\bb{\lambda_{1}-\lambda_{2}}\bb{\beta+i}}{\alpha}}
    }
    Therefore, 
    \bas{
        \lfloor \beta \rfloor - k_{\lfloor \beta \rfloor} &\geq 
        \lfloor \beta \rfloor - \frac{4\tau_{\text{mix}}}{\log\bb{2}}\log\bb{\frac{\beta + \lfloor \beta \rfloor} {\frac{\alpha}{\lambda_{1}-\lambda_{2}}}} \\
        &\geq \frac{\beta}{2} - \frac{4\tau_{\text{mix}}}{\log\bb{2}}\log\bb{\frac{2\beta}{\frac{\alpha}{\lambda_{1}-\lambda_{2}}}} \text{ since } \beta > 1 \\ 
        &= \beta\bbb{\frac{1}{2} - \frac{4\tau_{\text{mix}}}{\log\bb{2}}\frac{\log\bb{2f\bb{0}}}{\beta}}, \text{ where } f\bb{.} \text{ is defined in Lemma } \ref{lemma:learning_rate_schedule}
    }
    Now, from Lemma \ref{lemma:learning_rate_schedule}, we know that $f\bb{0} > e$. Therefore, $\log\bb{2f\bb{0}} \leq 2\log\bb{f\bb{0}}$. Then,
    \bas{
        \lfloor \beta \rfloor - k_{\lfloor \beta \rfloor} &\geq \beta\bbb{\frac{1}{2} - \frac{8\tau_{\text{mix}}}{\log\bb{2}}\frac{\log\bb{f\bb{0}}}{\beta}}
    }
    Again, from the conditions in Lemma \ref{lemma:learning_rate_schedule}, we know that 
    \bas{
        \frac{\log\bb{f\bb{0}}}{\beta} \leq \frac{\epsilon}{6\tau_{\text{mix}}}\frac{\lambda_{1}-\lambda_{2}}{\bb{\mathcal{M}+\lambda_{1}}\alpha} \leq \frac{1}{120\tau_{\text{mix}}} \text{ since } \alpha > 2, \frac{\lambda_{1}-\lambda_{2}}{\mathcal{M}+\lambda_{1}} \leq 1, \epsilon \leq \frac{1}{100}
    }
    Therefore, 
    \bas{
        \lfloor \beta \rfloor - k_{\lfloor \beta \rfloor} \geq \beta\bb{\frac{1}{2}-\frac{8}{120\log\bb{2}}} \geq 0
    }
    Hence proved.
\end{proof}

\subsubsection{Numerator}

Using Theorem \ref{theorem:appendix:Vperpupperbound} and Markov's Inequality, we have with probability atleast $\bb{1-\delta}$
\begin{align*}
& \Tr\left(V_{\perp}^{T}B_{n}B_{n}^{T}V_{\perp}\right) \leq \\
& \;\;\;\; 1.05\frac{\exp\left(\sum_{i=\init+1}^n 2\eta_i\lambda_2+\overline{\mathcal{V}_{k,i}}\eta_{i-k_{i}}^2\right)}{\delta}\left(d + \sum_{i=\init+1}^{n}\bb{\mathcal{V}'+\xi_{k,i}}C_{k,i}'\eta_{i-k_{i}}^{2}\exp\left(\sum_{j=\init+1}^{i}2\eta_j\left( \lambda_1-\lambda_{2}\right)\right)\right)
\end{align*}

\subsubsection{Denominator}

Using Chebyshev's Inequality we have, with probability atleast $\bb{1-\delta}$
\ba{
    v_{1}^{T}\B_{n}\B_{n}^{T}v_{1} &\geq \mathbb{E}\left[v_{1}^{T}\B_{n}\B_{n}^{T}v_{1}\right]\bb{1 - \sqrt{\frac{1}{\delta}}\sqrt{\frac{\mathbb{E}\left[\bb{v_{1}^{T}\B_{n}\B_{n}^{T}v_{1}}^{2}\right]}{\mathbb{E}\left[v_{1}^{T}\B_{n}\B_{n}^{T}v_{1}\right]^{2}} - 1}} \label{eq:chebyshev_denominator}
}
Let $r := 2\bb{1+\epsilon}\eta_{n}k_{n}\left(\mathcal{M}+\lambda_{1}\right) \leq \etakub$. Using Theorem \ref{theorem:appendix:Vperpupperbound}, we have
\bas{
\mathbb{E}\left[\bb{v_{1}^{T}\B_{n}\B_{n}^{T}v_{1}}^{2}\right] \leq \bb{1+r}^{4}\exp\left(\sum_{i=1}^{n-k_{n}}4\eta_{i}\lambda_{1} + \sum_{i=1}^{n-k_{n}}\eta_{i}^{2}\zeta_{k,t} \right)
}
Using Theorem \ref{theorem:appendix:v1lowerbound}, we have
\begin{align*}
& \mathbb{E}\left[v_{1}^{T}\newB{1}{n}\newB{1}{n}^{T}v_{1}\right] \geq \exp\bb{2\lambda_{1}\sum_{i=1}^{n-k_{n}}\eta_{i} - 4\lambda_{1}^{2}\sum_{i=1}^{n-k_{n}}\eta_{i}^{2}}\bbb{1 - \bb{2r + 3\bb{1+r}^{2}\sum_{t=1}^{n-k_{n}} W_{k,t}\eta_{t}^2\exp\left(\sum_{i=t+1}^{n-k_{n}}\eta_{i}^{2}\overline{\mathcal{V}_{k,i}}\right)}}
\end{align*}
Let 
\bas{
    s := 3\bb{1+r}^{2}\sum_{t=1}^{n-k_{n}} W_{k,t}\eta_{t}^2\exp\left(\sum_{i=t+1}^{n-k_{n}}\eta_{i}^{2}\overline{\mathcal{V}_{k,i}}\right)
}
Then,
\bas{
\frac{\mathbb{E}\left[\bb{v_{1}^{T}\B_{n}\B_{n}^{T}v_{1}}^{2}\right]}{\mathbb{E}\left[v_{1}^{T}\B_{n}\B_{n}^{T}v_{1}\right]^{2}} &\leq \frac{\bb{1+r}^{4}}{\bb{1 - 2r - s}^{2}}\exp\left( \sum_{i=1}^{n-k_{n}}\eta_{i}^{2}\bb{\zeta_{k,i} + 4\lambda_{1}^{2}}\right)
}
By Lemma \ref{lemma:additional_assumptions}, we have that
\begin{align}
    2r + s \leq \frac{1}{2} \label{assumption:lower_bound_n_1}.
\end{align}
Then, using
\bas{
& \frac{1}{\bb{1-x}^{2}} \leq 1 + 6x \text{ for } x \in \left[0,\frac{1}{2}\right] \text{ and, } \bb{1+x}^{4} \leq 1 + 5x \text{ for } x \in \bbb{0,\frac{1}{10}}
}
we have, 
\bas{
\frac{\mathbb{E}\left[\bb{v_{1}^{T}\B_{n}\B_{n}^{T}v_{1}}^{2}\right]}{\mathbb{E}\left[v_{1}^{T}\B_{n}\B_{n}^{T}v_{1}\right]^{2}} &\leq \bb{1 + \mydiv{\csq}{4} r}\bb{1 + 12r + 6s}\exp\left( \sum_{i=1}^{n-k_{n}}\eta_{i}^{2}\bb{\zeta_{k,i} + 4\lambda_{1}^{2}} \right) \\
&\leq \bb{1 + \add{\mydiv{\csq}{4}}{12}r + 6s + \mult{\mydiv{\csq}{4}}{12}r^{2} + \mult{\mydiv{\csq}{4}}{6}rs}\exp\left( \sum_{i=1}^{n-k_{n}}\eta_{i}^{2}\bb{\zeta_{k,i} + 4\lambda_{1}^{2}} \right)\\
&\leq \bb{1 + \add{\mydiv{\csq}{4}}{17}r + 12s}\exp\left( \sum_{i=1}^{n-k_{n}}\eta_{i}^{2}\bb{\zeta_{k,i} + 4\lambda_{1}^{2}} \right) \text{ since } r \leq \etakub
}
By Lemma \ref{lemma:learning_rate_schedule}-(3), we have that
\begin{align}
    & \exp\left( \sum_{i=1}^{n-k_{n}}\eta_{i}^{2}\bb{\zeta_{k,i} + 4\lambda_{1}^{2}} \right) \leq 1+\frac{\delta}{m} \label{assumption:eta2_assumption2}
\end{align}
By \ref{eq:epsilon_delta_bound}, we have that
\ba{
    12s  &\leq \frac{48(1+r)^{2}}{100}\bb{1+\frac{\delta}{m}}^{2}\log\bb{1+\frac{\delta}{m}} \notag \\ 
        &\leq \frac{3}{5}\bb{1+\frac{\delta}{m}}^{2}\log\bb{1+\frac{\delta}{m}} \text{ since } r \leq \frac{1}{10} \label{eq:epsilon_delta_bound_1}
}
By Lemma \ref{lemma:additional_assumptions}, we have that 
\begin{align}
    r = 2\bb{1+\epsilon}\eta_{n}k_{n}\left(\mathcal{M}+\lambda_{1}\right) <\frac{1}{50}\frac{\delta/m}{1+\delta/m} \label{assumption:r_delta_condition}
\end{align} 
Then, 
\bas{
\frac{\mathbb{E}\left[\bb{v_{1}^{T}\B_{n}\B_{n}^{T}v_{1}}^{2}\right]}{\mathbb{E}\left[v_{1}^{T}\B_{n}\B_{n}^{T}v_{1}\right]^{2}} &\leq \bb{1+22r+12s}\bb{1+\frac{\delta}{m}} \\ 
&= 1 + \frac{\delta}{m} + 22r\bb{1+\frac{\delta}{m}} + 12s\bb{1+\frac{\delta}{m}} \\ 
&\leq 1 + \frac{\delta}{m} + \frac{22}{50}\frac{\delta}{m} + \frac{3}{5}\bb{1+\frac{\delta}{m}}^{3}\log\bb{1+\frac{\delta}{m}} \\ 
&\leq 1 + \frac{\delta}{m} + \frac{22}{50}\frac{\delta}{m} + \frac{7}{10}\log\bb{1+\frac{\delta}{m}} \text{ since } \delta \leq 1, m = 200 \\ 
&\leq 1 + \frac{\delta}{m} + \frac{22}{50}\frac{\delta}{m} + \frac{7}{10}\frac{\delta}{m} \text{ since }  \forall x, \; \log\bb{1+x} \leq x \\ 
&\leq 1 + 3\frac{\delta}{m}
}
Then setting $m = 200$, from \ref{eq:chebyshev_denominator} we have 
\bas{
v_{1}^{T}\B_{n}\B_{n}^{T}v_{1} &\geq \exp\bb{\sum_{i=1}^{n-k_{n}}2\eta_{i}\lambda_{1}-4\eta_{i}^{2}\lambda_{1}^{2}}\bb{1 - 2r - s}\bb{1 - \sqrt{\frac{1}{\delta}}\sqrt{\frac{3\delta}{m}}} \\
&\geq \exp\bb{\sum_{i=1}^{n-k_{n}}2\eta_{i}\lambda_{1}-4\eta_{i}^{2}\lambda_{1}^{2}}\bb{1 - \frac{1}{25}\frac{\delta/m}{1+\delta/m} - \frac{1}{20}\bb{1+\frac{\delta}{m}}^{2}\log\bb{1+\frac{\delta}{m}}}\bb{1 - \sqrt{\frac{3}{m}}}  \\ 
&\geq \frac{5}{6}\exp\bb{\sum_{i=1}^{n-k_{n}}2\eta_{i}\lambda_{1}-4\eta_{i}^{2}\lambda_{1}^{2}} \text{ since } \delta \leq 1 \text{ and } m = 200
}
The second inequality uses Eqs~\ref{eq:epsilon_delta_bound_1},~\ref{assumption:r_delta_condition}.
\subsubsection{Fraction}
Now that we have established this result let's calculate the fraction. Let the step-sizes be set according to Lemma \ref{lemma:learning_rate_schedule}. Define 
\bas {
    & \mathcal{S} := \exp\bb{\sum_{i=u+1}^{n}\overline{\mathcal{V}_{k,i}}\eta_{i-k_{i}}^2 + \sum_{i=1}^{n-k_{n}}4\lambda_{1}^{2}\eta_{i}^{2}} \\
    & Q_{\init} := \exp\bb{2\lambda_{1}\bb{\sum_{j=1}^{\init}\eta_{j} - \sum_{j=n-k_{n}+1}^{n}\eta_{j}}} \\
    & \mathcal{R}_{k,t} := \frac{\exp\bb{\sum_{j=1}^{t-\init}\eta_{j}^{2}\bb{\overline{\mathcal{V}_{k,j}}-\overline{\mathcal{V}_{k,j+u}}}}\exp\bb{2\lambda_{1}\sum_{j=n-k_{n}+1}^{n}\eta_{j}}}{\exp\bb{2\lambda_{1}\sum_{j=1}^{\init}\eta_{t-\init+j}}}
}
Then, recall that
\bas {
    \init &:= \min\left\{i : i \in [n], i-k_{i} \geq 0 \right\} \\ 
    \xi_{k,t} &:= 6\eta_{t-k_{t}}\mathcal{M}\bbb{1 + 3k_{t+1}^{2}\bb{\mathcal{M}+\lambda_{1}}^{2}} \\
    \mathcal{V}' &:= \bb{\frac{1 + \bb{3 + 4\epsilon}|\lambda_{2}\bb{P}|}{1 - \left|\lambda_{2}\bb{P}\right|}}\mathcal{V} \\ 
    \overline{\mathcal{V}_{k,t}} &:= \mathcal{V}' + \lambda_{1}^{2} + \xi_{k,t} \\
    C_{k,t}' &:= \exp\bb{2\lambda_{1}\sum_{j=1}^{\init}\bb{\eta_{j}-\eta_{t-\init+j}}+\sum_{j=1}^{t-\init}\eta_{j}^{2}\bb{\overline{\mathcal{V}_{k,j}}-\overline{\mathcal{V}_{k,j+u}}}} = Q_{\init}\mathcal{R}_{k,t}
}
Therefore, 
\ba{
& \frac{\Tr\left(V_{\perp}^{T}\B_{n}\B_{n}^{T}V_{\perp}\right)}{v_{1}^{T}\B_{n}\B_{n}^{T}v_{1}} \notag \\
& \leq \frac{1.3}{\delta}\frac{\exp\left(\sum_{i=\init+1}^n 2\eta_i\lambda_2+\overline{\mathcal{V}_{k,i}}\eta_{i-k_{i}}^2\right)}{\exp\bb{\sum_{i=1}^{n-k_{n}}2\eta_{i}\lambda_{1}-4\eta_{i}^{2}\lambda_{1}^{2}}}\left(d + \sum_{i=\init+1}^{n}\bb{\mathcal{V}'+\xi_{k,i}}C_{k,i}'\eta_{i-k_{i}}^{2}\exp\left(\sum_{j=\init+1}^{i}2\eta_j\left( \lambda_1-\lambda_{2}\right)\right)\right) \notag \\ 
&\leq \frac{1.3}{\delta}\frac{\mathcal{S}}{Q_{u}}\exp\bb{\sum_{i=\init+1}^{n}2\eta_{i}\bb{\lambda_{2}-\lambda_{1}}}\left(d + \sum_{i=\init+1}^{n}\bb{\mathcal{V}'+\xi_{k,i}}C_{k,i}'\eta_{i-k_{i}}^{2}\exp\left(\sum_{j=\init+1}^{i}2\eta_j\left( \lambda_1-\lambda_{2}\right)\right)\right) \notag \\ 
&\leq \frac{1.3}{\delta}\mathcal{S}\left(\underbrace{\frac{d\exp\bb{\sum_{i=\init+1}^{n}2\eta_{i}\bb{\lambda_{2}-\lambda_{1}}}}{Q_{u}}}_{X_{1}} + \underbrace{\sum_{i=\init+1}^{n}\bb{\mathcal{V}'+\xi_{k,i}}\mathcal{R}_{k,i}\eta_{i-k_{i}}^{2}\exp\left(-\sum_{j=i+1}^{n}2\eta_j\left( \lambda_1-\lambda_{2}\right)\right)}_{X_{2}}\right) \label{X_1_X_2_decomposition}
}
For $X_{1}$,  we have
\bas{
X_{1} &\leq \frac{d\exp\bb{\sum_{i=\init+1}^{n}2\eta_{i}\bb{\lambda_{2}-\lambda_{1}}}}{Q_{u}} \\
    &= \frac{d\exp\bb{\sum_{i=\init+1}^{n}2\eta_{i}\bb{\lambda_{2}-\lambda_{1}}}}{\exp\bb{2\lambda_{1}\bb{\sum_{j=1}^{\init}\eta_{j} - \sum_{j=n-k_{n}+1}^{n}\eta_{j}}}} \\
    &\leq \frac{d\exp\bb{\sum_{i=\init+1}^{n}2\eta_{i}\bb{\lambda_{2}-\lambda_{1}}}}{\exp\bb{-2\lambda_{1}\bb{\sum_{j=n-k_{n}+1}^{n}\eta_{j}}}} \\
    &\leq d\exp\bb{\sum_{i=u+1}^{n}2\eta_{i}\bb{\lambda_{2}-\lambda_{1}}}\exp\bb{2\lambda_{1}\bb{\sum_{j=n-k_{n}+1}^{n}\eta_{j}}}
}
Note that
\bas{
    \exp\bb{2\lambda_{1}\sum_{j=n-k_{n}+1}^{n}\eta_{j}} &\leq \exp\bb{2\bb{1+2\epsilon}\lambda_{1}k_{n}\eta_{n-k_{n}+1}} \text{ using monotonicity of } \eta_{i} \\ 
    &\leq \exp\bb{4\bb{1+2\epsilon}\lambda_{1}k_{n}\eta_{n}} \text{ using slow-decay of } \eta_{i} \\ 
    &\leq 1+2\frac{\delta}{m} \text{ using Lemma } \ref{lemma:additional_assumptions} \text{ along with } e^{x} \leq 1 + x + x^{2} \text{ for } x \in \bb{0,1}
}
Therefore, using \ref{ineq:eta_lower_bound}
\bas{
    X_{1} &\leq d\bb{1+\frac{2\delta}{m}}\bb{\frac{\beta + \init}{n}}^{2\alpha}
}

Next, for $X_{2}$, we first have
\bas{
    \mathcal{R}_{k,t} &:= \frac{\exp\bb{\sum_{j=1}^{t-\init}\eta_{j}^{2}\bb{\overline{\mathcal{V}_{k,j}}-\overline{\mathcal{V}_{k,j+u}}}}\exp\bb{2\lambda_{1}\sum_{j=n-k_{n}+1}^{n}\eta_{j}}}{\exp\bb{2\lambda_{1}\sum_{j=1}^{\init}\eta_{t-\init+j}}} \\ 
    &\leq \exp\bb{\sum_{j=1}^{t-\init}\eta_{j}^{2}\overline{\mathcal{V}_{k,j}}}\exp\bb{2\lambda_{1}\sum_{j=n-k_{n}+1}^{n}\eta_{j}} \\ 
    &\leq \bb{1+\frac{2\delta}{m}}^{2} \text{ using Lemmas } \ref{lemma:learning_rate_schedule}-(3),  
 \ref{lemma:additional_assumptions} \text{ and } e^{x} \leq 1 + x + x^{2} \text{ for } x \in \bb{0,1}
}
Now, using \ref{lemma:learning_rate_schedule}-(4) we have, 
\bas{
    & \sum_{i=1}^{n}\overline{\mathcal{V}_{k,i}}\eta_{i-k_{i}}^{2}\exp\left(-\sum_{j=i+1}^{n}2\eta_j\left( \lambda_1-\lambda_{2}\right)\right) \leq \\
    & \;\; \bb{\frac{2\bb{1+10\epsilon}\alpha^{2}}{2\alpha - 1}}\frac{\mathcal{V}'}{\bb{\lambda_{1}-\lambda_{2}}^{2}}\frac{1}{n} +  \bb{\frac{800\bb{1+10\epsilon}\alpha^{3}}{\bb{\alpha - 1}}}\frac{\mathcal{M}\bb{\mathcal{M}+\lambda_{1}}^{2}\tau_{\text{mix}}^{2}}{\bb{\lambda_{1}-\lambda_{2}}^{3}}\frac{\log^{2}\bb{\frac{\bb{\beta+n}\bb{\lambda_{1}-\lambda_{2}}}{\alpha}}}{n^{2}}
}
Then,
\bas{
    X_{2} &\leq \bb{1+\frac{2\delta}{m}}^{2}\bbb{\underbrace{\bb{\frac{2\bb{1+10\epsilon}\alpha^{2}}{2\alpha - 1}}}_{C_{1}}\frac{\mathcal{V}'}{\bb{\lambda_{1}-\lambda_{2}}^{2}}\frac{1}{n} +  \underbrace{\bb{\frac{24\bb{1+10\epsilon}\alpha^{3}}{\bb{\alpha - 1}}}}_{C_{2}}\frac{\mathcal{M}\bb{\mathcal{M}+\lambda_{1}}^{2}}{\bb{\lambda_{1}-\lambda_{2}}^{3}}\frac{k_n^{2}}{n^{2}}}
}
Therefore substituting in \ref{X_1_X_2_decomposition},
\ba {
    &\frac{\Tr\left(V_{\perp}^{T}\B_{n}\B_{n}^{T}V_{\perp}\right)}{v_{1}^{T}\B_{n}\B_{n}^{T}v_{1}} \leq \frac{1.3\mathcal{S}}{\delta }\bb{1+\frac{2\delta}{m}}^{2}\bbb{d\bb{\frac{\beta+\init}{n}}^{2\alpha} + \frac{C_{1}\mathcal{V}'}{\bb{\lambda_{1}-\lambda_{2}}^{2}}\frac{1}{n} + \frac{C_{2}\mathcal{M}\bb{\mathcal{M}+\lambda_{1}}^{2}}{\bb{\lambda_{1}-\lambda_{2}}^{3}}\frac{k_n^{2}}{n^{2}}} \label{eq:fraction_bound}
}

\begin{proof}[Proof of Theorem \ref{theorem:oja_convergence_rate}]
    To complete our proof, we bound $\mathcal{S}$ to simplify \ref{eq:fraction_bound}.
    We note that under the learning rate schedule presented in Lemma \ref{lemma:learning_rate_schedule}-(3), 
    \bas{
        \mathcal{S}  &\leq \bb{1 + \frac{\delta}{m}}
    }
    Therefore, 
    \bas{
    \frac{\Tr\left(V_{\perp}^{T}\B_{n}\B_{n}^{T}V_{\perp}\right)}{v_{1}^{T}\B_{n}\B_{n}^{T}v_{1}} &\leq \frac{1.3}{\delta }\bb{1+\frac{2\delta}{m}}^{3}\bbb{d\bb{\frac{\beta+\init}{n}}^{2\alpha} + \frac{C_{1}\mathcal{V}'}{\bb{\lambda_{1}-\lambda_{2}}^{2}}\frac{1}{n} + \frac{C_{2}\mathcal{M}\bb{\mathcal{M}+\lambda_{1}}^{2}}{\bb{\lambda_{1}-\lambda_{2}}^{3}}\frac{k_n^{2}}{n^{2}}} \\  
    &\leq \frac{1.4}{\delta}\bbb{d\bb{\frac{\beta+\init}{n}}^{2\alpha} + \frac{C_{1}\mathcal{V}'}{\bb{\lambda_{1}-\lambda_{2}}^{2}}\frac{1}{n} + \frac{C_{2}\mathcal{M}\bb{\mathcal{M}+\lambda_{1}}^{2}}{\bb{\lambda_{1}-\lambda_{2}}^{3}}\frac{k_n^{2}}{n^{2}}}
    }
    Using lemma \ref{lemma:u_bound}, we have that $u \leq \beta$. Then, using Lemma 3.1 from \cite{jain2016streaming} completes our proof.
\end{proof}

\subsection{Proof of Corollary \ref{cor:data-drop}}

\begin{proof}[Proof of Corollary \ref{cor:data-drop}]
We note that the downsampled data stream can be considered to be drawn from a Markov chain with transition kernel $P^{k}\bb{.,.}$ since each data-point is $k$ steps away from the previous one. We will denote the parameters of this transformed chain by  $\tilde{y}$ when the corresponding parameter is $y$ under the original chain. For example, $\tilde{\tau}_{\text{mix}}$ is the mixing time of the new chain.

Note that this modified transition matrix has the same stationary distribution $\pi$. It is also reversible. This can be seen by considering the diagonal matrix of stationary distribution probabilities $\Pi$, where $\Pi_{ii}=\pi_i$. For a reversible Markov Chain, we have $\Pi P=P\Pi$. However, that also implies $\Pi P^2 = (\Pi P) P= (P\Pi) P=P(\Pi P)=P^2\Pi$. This same technique works for $P^k$ yielding $\Pi P^k=P^k\Pi$.

Using standard results on Markov chains~\cite{levin2017Markov},
\ba{
        \frac{|\lambda_{2}\bb{P}|}{1-|\lambda_{2}\bb{P}|}\log\bb{\frac{1}{2\epsilon}} \leq  \tau_{\text{mix}}(\epsilon) \leq \frac{1}{1-|\lambda_{2}\bb{P}|}\log\bb{\frac{1}{\epsilon \pi_{\text{min}}}},\label{eq:mixing_time_eigengap}
}
where $\pi_{\min}:=\min_i \pi_i$. 
Therefore, as noted in the theorem statement, we substitute the modified parameters in the bound we have proven for Theorem~\ref{theorem:oja_convergence_rate}.

First, we will show that the mixing time for this new chain is $\Theta(1)$. We will use $k := \tau_{\text{mix}}\bb{\eta_{n}^{2}}$.
So by definition the $d_{\text{mix}}(k)\leq \eta_n^2$ using the definition of $d_{\text{mix}}$ in Section~\ref{section:preliminaries}.
Hence $d_{\text{mix}}(k)\leq 1/4$ using conditions on the learning rate schedule imposed in Theorem~\ref{theorem:oja_convergence_rate}. Therefore, in the transformed chain, the ``new'' $\tilde{\tau}_{\text{mix}}$ is $\Theta(1)$.

We also have:
\bas{
k\leq \frac{2\tau_{\text{mix}}}{\log\bb{2}}\log\bb{\frac{1}{\eta_{n}^{2}}} \leq \underbrace{\frac{2\log (4/\pi_{\text{min}})}{\log 2}}_{C}\frac{1}{1-\left|\lambda_{2}\bb{P}\right|}\log\bb{n}
}
We see that $C>1$.
Next, we note that for the transition kernel $P^{k}\bb{.,.}$, the second-largest absolute eigenvalue is given as $\left|\lambda_{2}\bb{P}\right|^{k}$. 
Consider the function $f\bb{x} := x^{\frac{1}{1-x}}$ for $x \in \bb{0,1}$. Then, 
\bas{
    f'\bb{x} &= f\bb{x}\bb{\frac{1-x-x\log\bb{x}}{x\bb{1-x}^{2}}} > 0
}
Therefore, $f\bb{x} < \lim_{x \rightarrow 1}f\bb{x} = \frac{1}{e} < 1$. which implies $\left|\lambda_{2}\bb{P}\right|^{k} \stackrel{(i)}\leq \bb{\frac{1}{e}}^{C\log\bb{n}} < \frac{1}{e}$. Here $(i)$ follows if $C > 1, n > 3$, which is true. Therefore, 
\bas{
    \widetilde{\mathcal{V}}' := \bb{\frac{1 + \bb{3 + 4\epsilon}|\lambda_{2}\bb{P}|^{k}}{1 - \left|\lambda_{2}\bb{P}\right|^{k}}}\mathcal{V} \leq 5\mathcal{V}
}
This also implies that the mixing time for the new Markov chain for sub-sampled data is $\Theta\bb{1}$. The bound then follows by substituting $n$ to be $\frac{n}{k} = n_{k}  = \Theta\bb{\frac{n}{C\tau_{\text{mix}}\log\bb{n}}}$ and setting the $\tau_{\text{mix}}$ in the original expression of Theorem \ref{theorem:oja_convergence_rate} to a constant.
\end{proof}

\section{Additional Experiments}
\label{appendix:additional_experiments}

In this section, we provide additional experiments to support the results established in Section \ref{section:main_results} of the manuscript. 
We present experiments with distributions that have nonzero mean vectors at each state, but zero mean with respect to the stationary distribution. This means that the $Z_i$'s are not necessarily zero-mean with respect to each state distribution $D\bb{s}$. To normalize the data-points, we estimate the mean $\mu$ and covariance matrix $\Sigma$ empirically from a much larger independently generated dataset.   

We experiment with two different settings here - Figure \ref{fig:Bernoulli_extended} contains the results for each state distribution being $D(s):=$ Bernoulli($p_s$) with $p_{s}\sim \mathcal{U}\bb{0,0.05}$ being fixed for each dataset. Figure \ref{fig:Uniform_extended} provides results for each state distribution being $D\bb{s} := \mathcal{U}\bb{0,\ell_{s}}$ with $\ell_{s} \sim \mathcal{U}\bb{0,10}$ being selected at the start of each random run. We observe that these experiments depict similar trends to those shown in the main manuscript, which validates our results for the case of non-zero state means. Furthermore, the Bernoulli data, being sparse compared to the Uniform one, seems to exhibit a clearer difference between data downsampling and the traditional Oja's algorithm.
To provide clear plots demonstrating the relative behavior of the algorithms considered in this paper, we have shown the averaged $\sin^{2}$ errors in Figures \ref{fig:Bernoulli_extended} and \ref{fig:Uniform_extended}. In Figure~\ref{fig:Bernoulli_random_runs} we show six random runs where we fixed the $p_s,s\in\Omega$ for each state for all runs. These figures clearly show that in general, Downsampled Oja has a worse performance than Oja's algorithm, which has a similar performance as the offline algorithm. It also shows that the Downsampled algorithm has the most variability, whereas Oja's algorithm on the whole dataset has much less variability, and finally, and not surprisingly, the offline algorithm has the least variability. Similar qualitative trends can be observed for the other settings.
\begin{figure}[!hbt]
    \centering
    \begin{subfigure}[b]{0.3\textwidth}
        \centering
        \includegraphics[width=\textwidth]{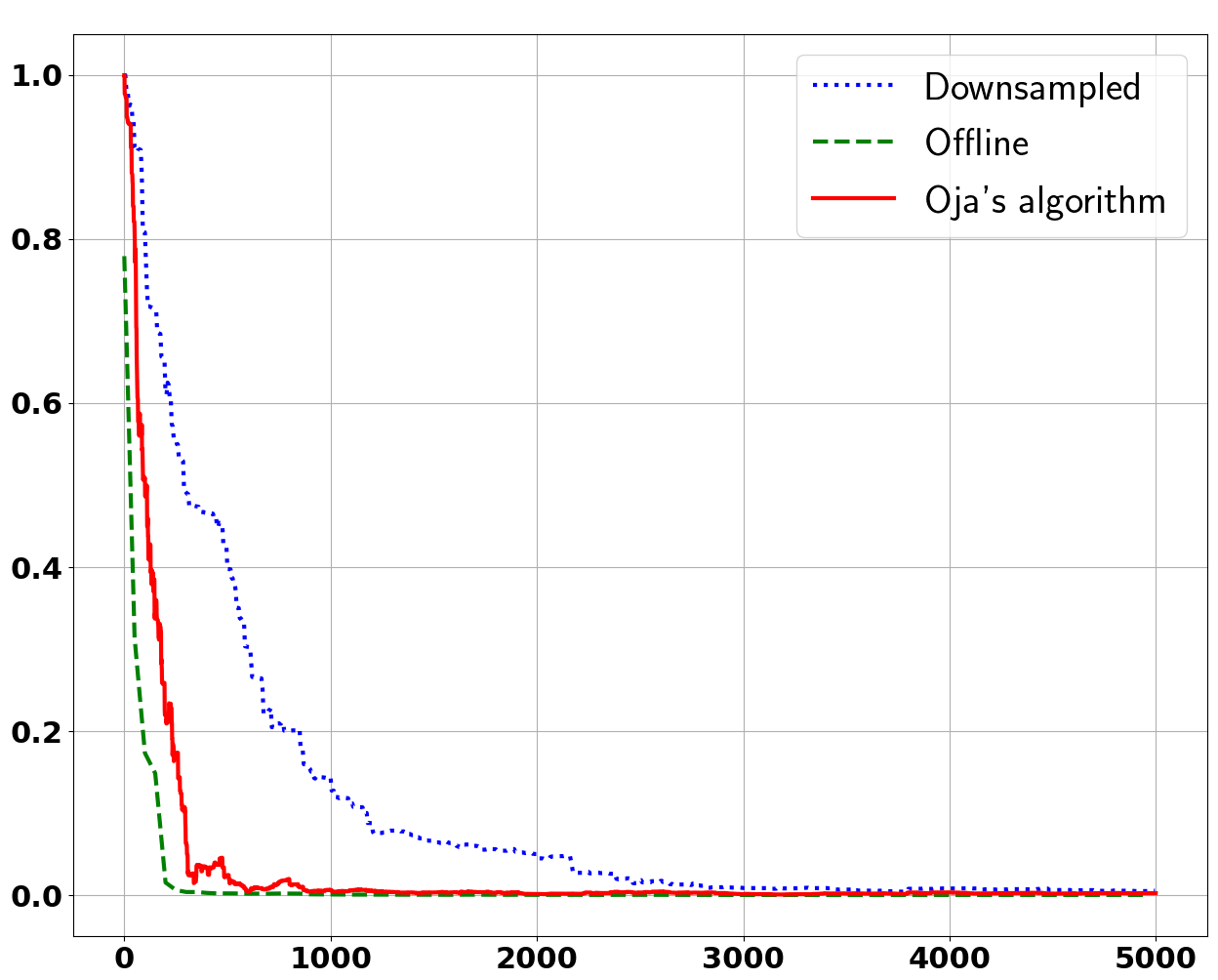}
        \caption{\label{fig:bernoulli_algo_comparison}}
    \end{subfigure}
    \begin{subfigure}[b]{0.3\textwidth}
        \centering
        \includegraphics[width=\textwidth]{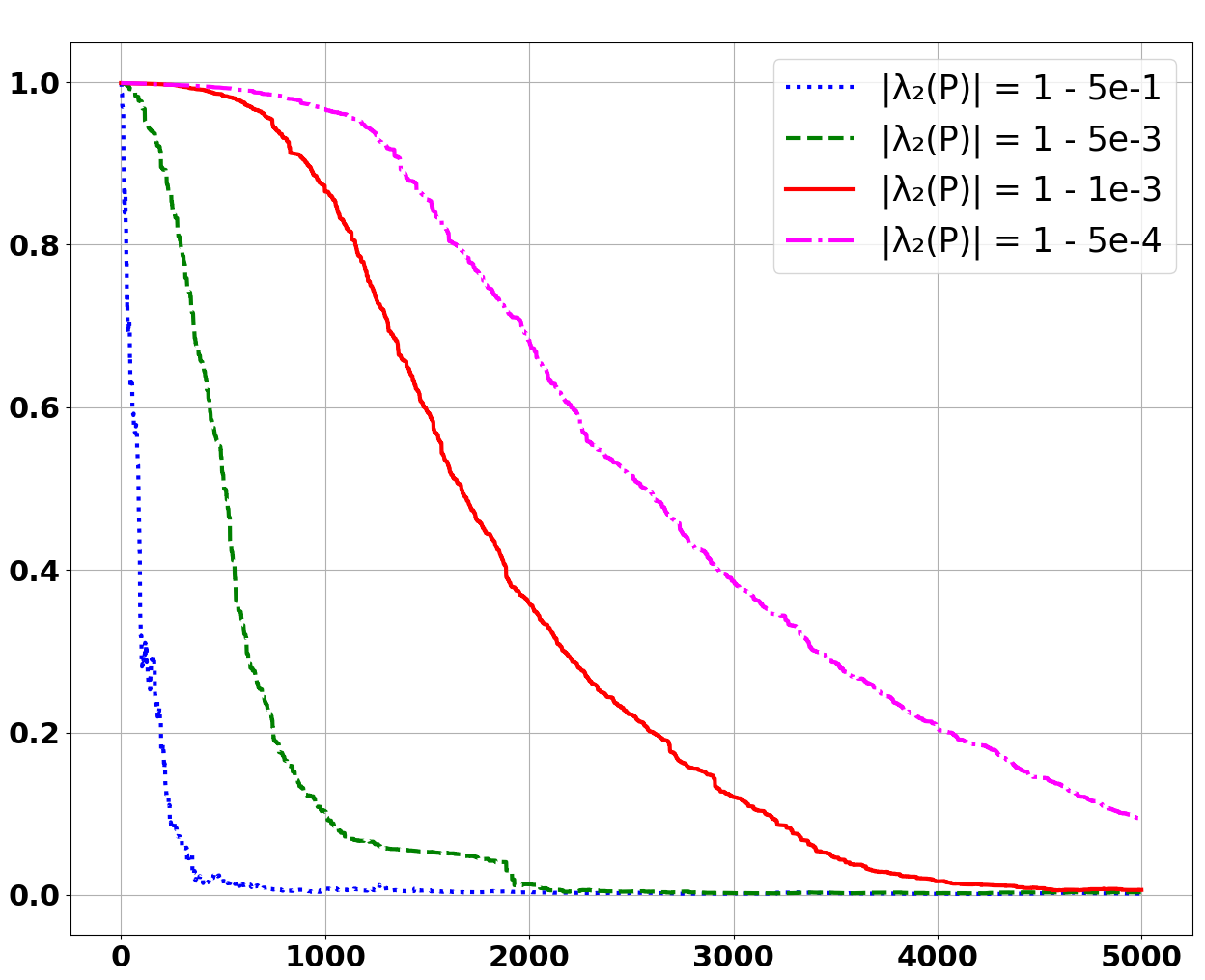}
        \caption{\label{fig:bernoulli_eigengap_comparison}}
    \end{subfigure}
    \begin{subfigure}[b]{0.3\textwidth}
        \centering
        \includegraphics[width=\textwidth]{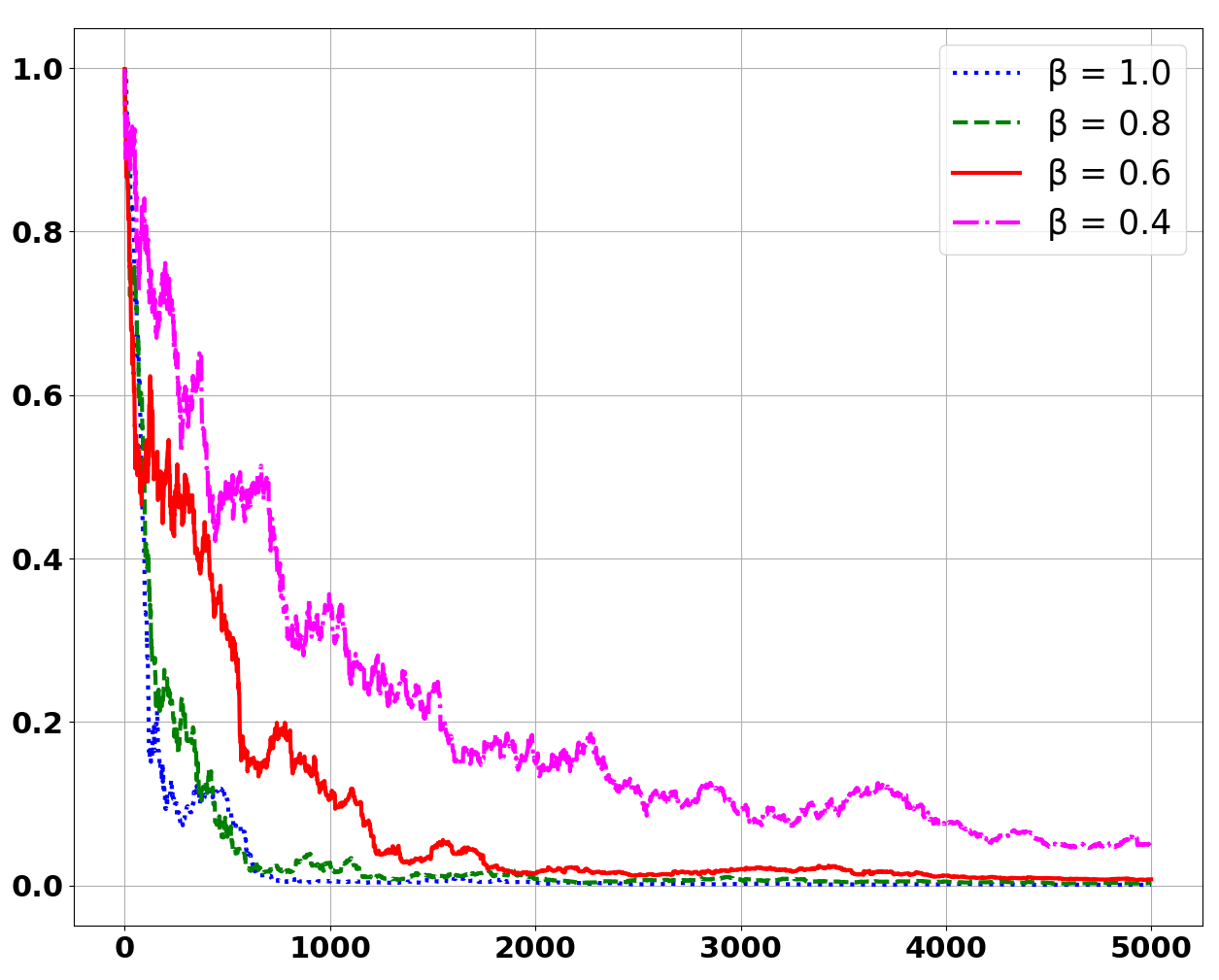}
        \caption{\label{fig:bernoulli_beta_comparison}}
    \end{subfigure}
    \caption{Experiments with Bernoulli data. \ref{fig:bernoulli_algo_comparison} compares the three different algorithms, \ref{fig:bernoulli_eigengap_comparison} shows effect of changing the eigengap of the transition_matrix and \ref{fig:bernoulli_beta_comparison} records the variation in performance on changing the eigengap of the data covariance matrix.}
    \label{fig:Bernoulli_extended}
\end{figure}

\begin{figure}[!hbt]
    \centering
    \begin{subfigure}[b]{0.3\textwidth}
        \centering
        \includegraphics[width=\textwidth]{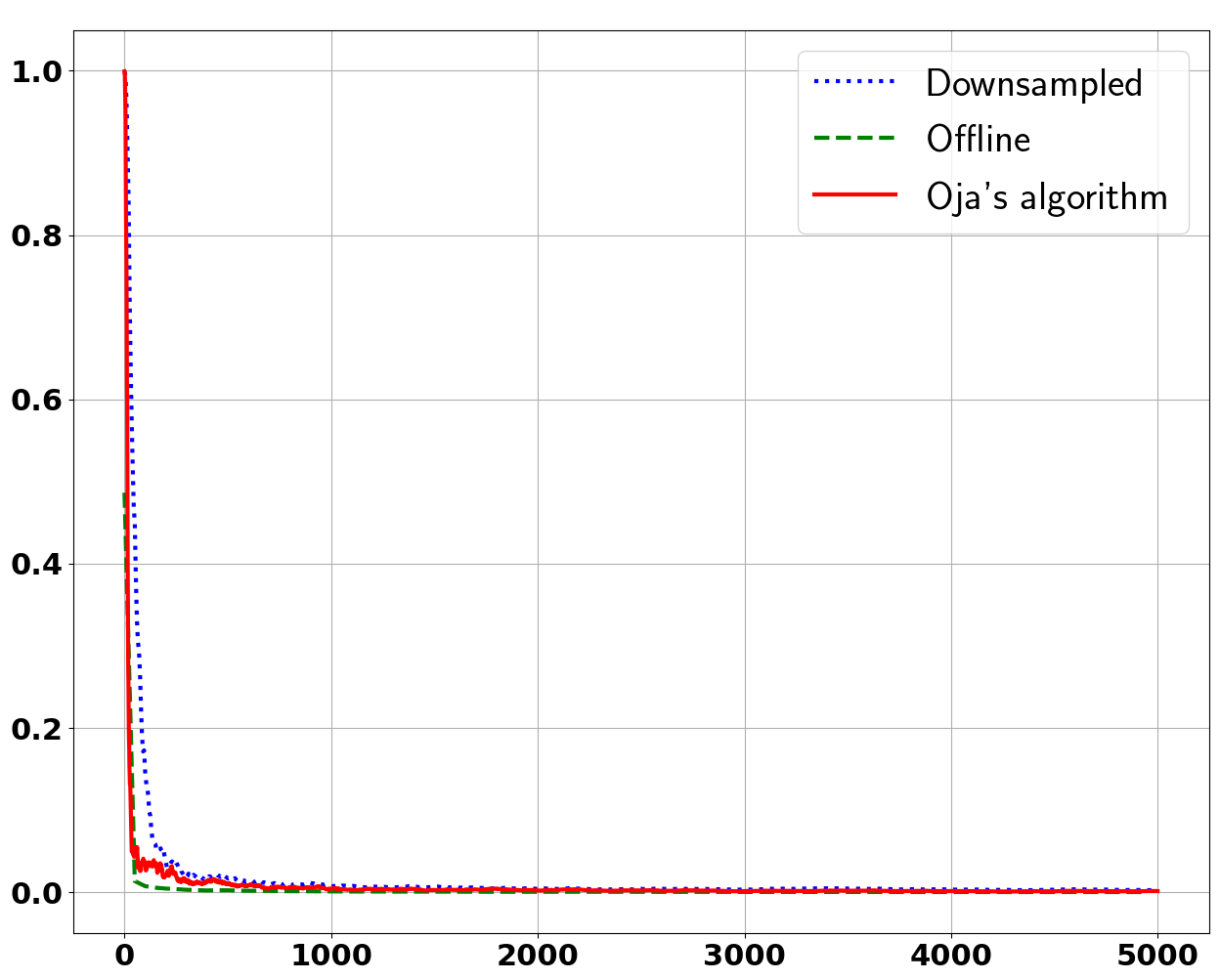}
        \caption{\label{fig:uniform_algo_comparison}}
    \end{subfigure}
    \begin{subfigure}[b]{0.3\textwidth}
        \centering
        \includegraphics[width=\textwidth]{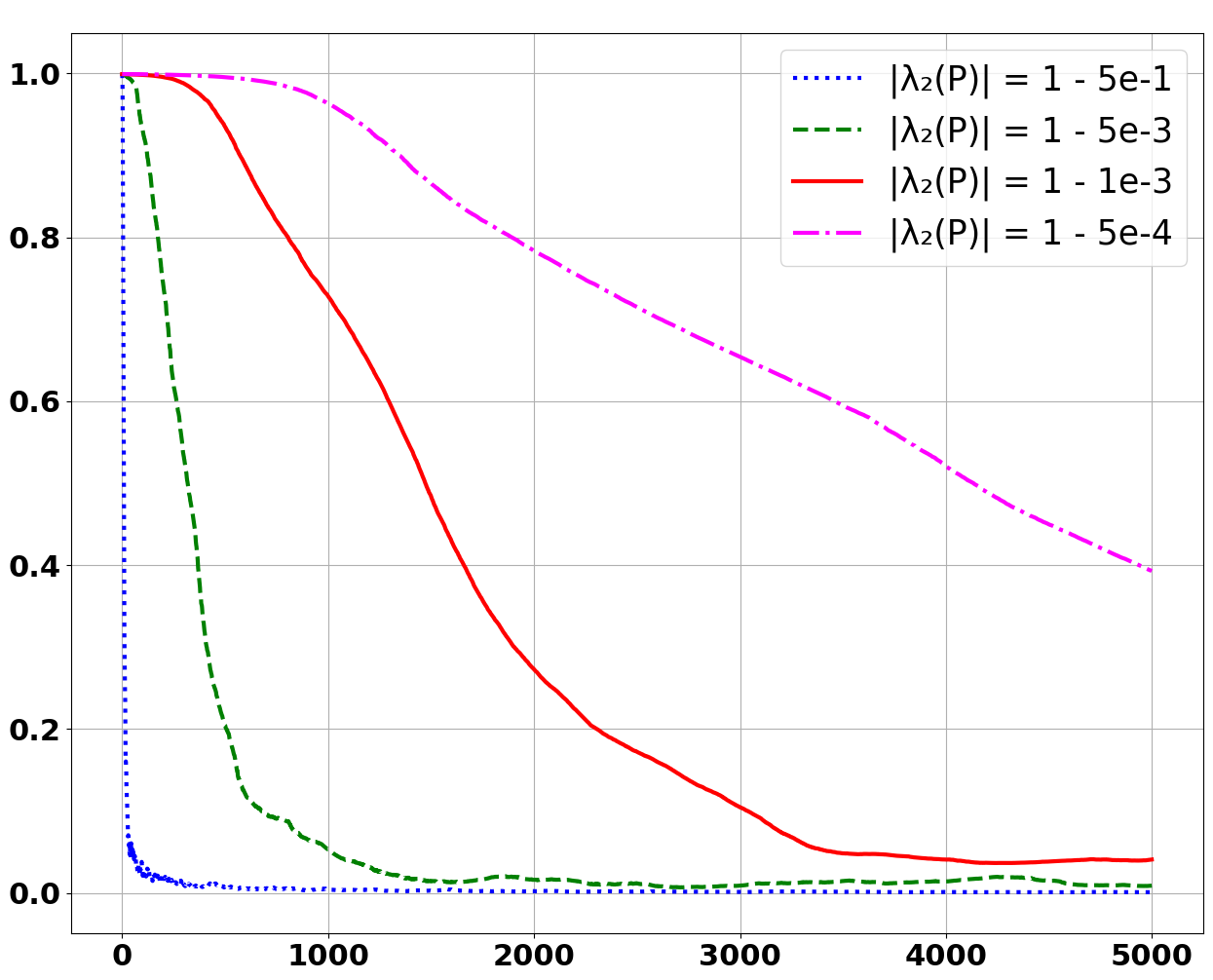}
        \caption{\label{fig:uniform_eigengap_comparison}}
    \end{subfigure}
    \begin{subfigure}[b]{0.3\textwidth}
        \centering
        \includegraphics[width=\textwidth]{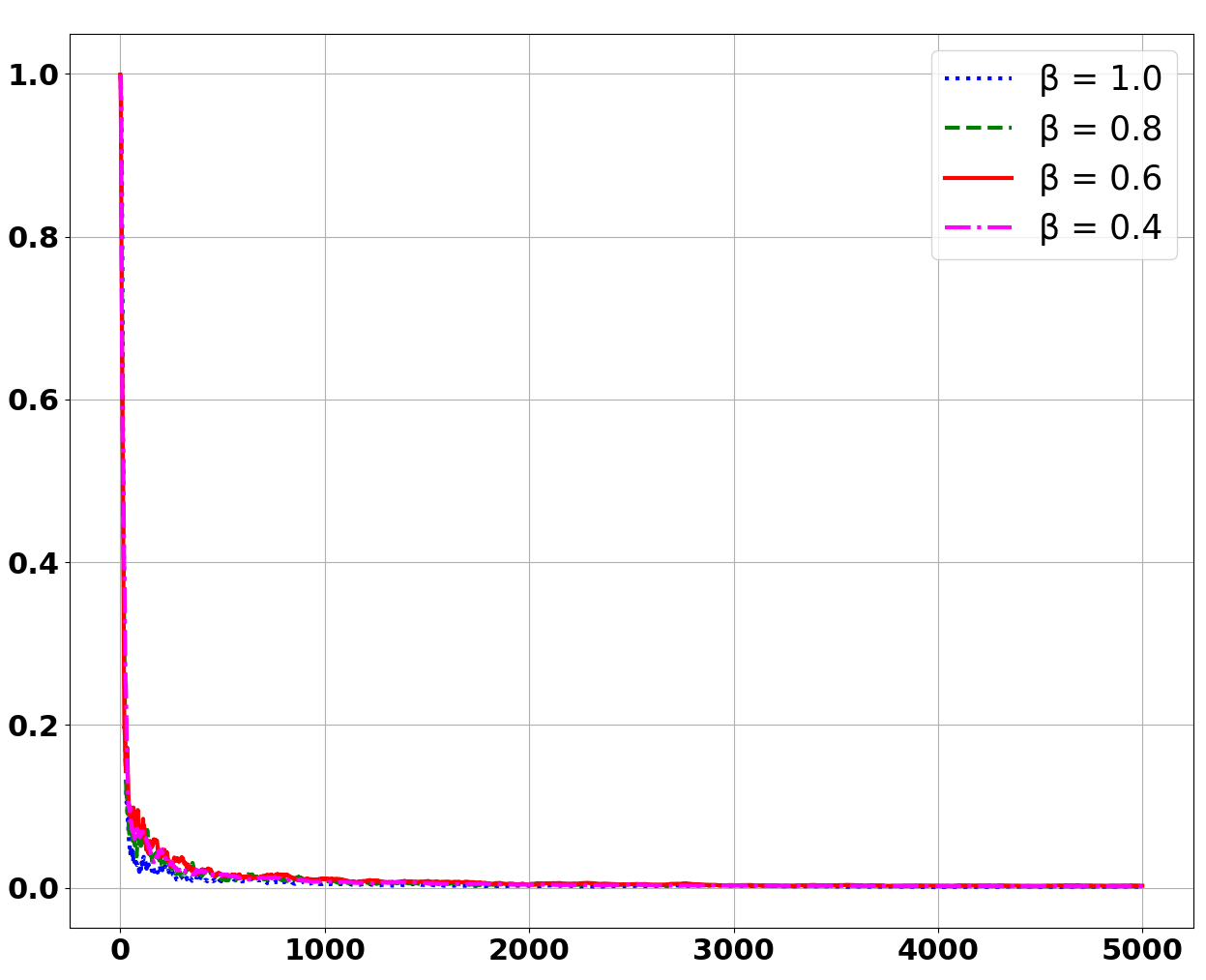}
        \caption{\label{fig:uniform_beta_comparison}}
    \end{subfigure}
    \caption{Experiments with Uniform data. \ref{fig:uniform_algo_comparison} compares the three different algorithms, \ref{fig:uniform_eigengap_comparison} shows effect of changing the eigengap of the transition matrix and \ref{fig:uniform_beta_comparison} records the variation in performance on changing the eigengap of the data covariance matrix.}
    \label{fig:Uniform_extended}
\end{figure}


\begin{figure}[H]
    \centering
    \begin{subfigure}[b]{0.3\textwidth}
        \centering
        \includegraphics[width=\textwidth]{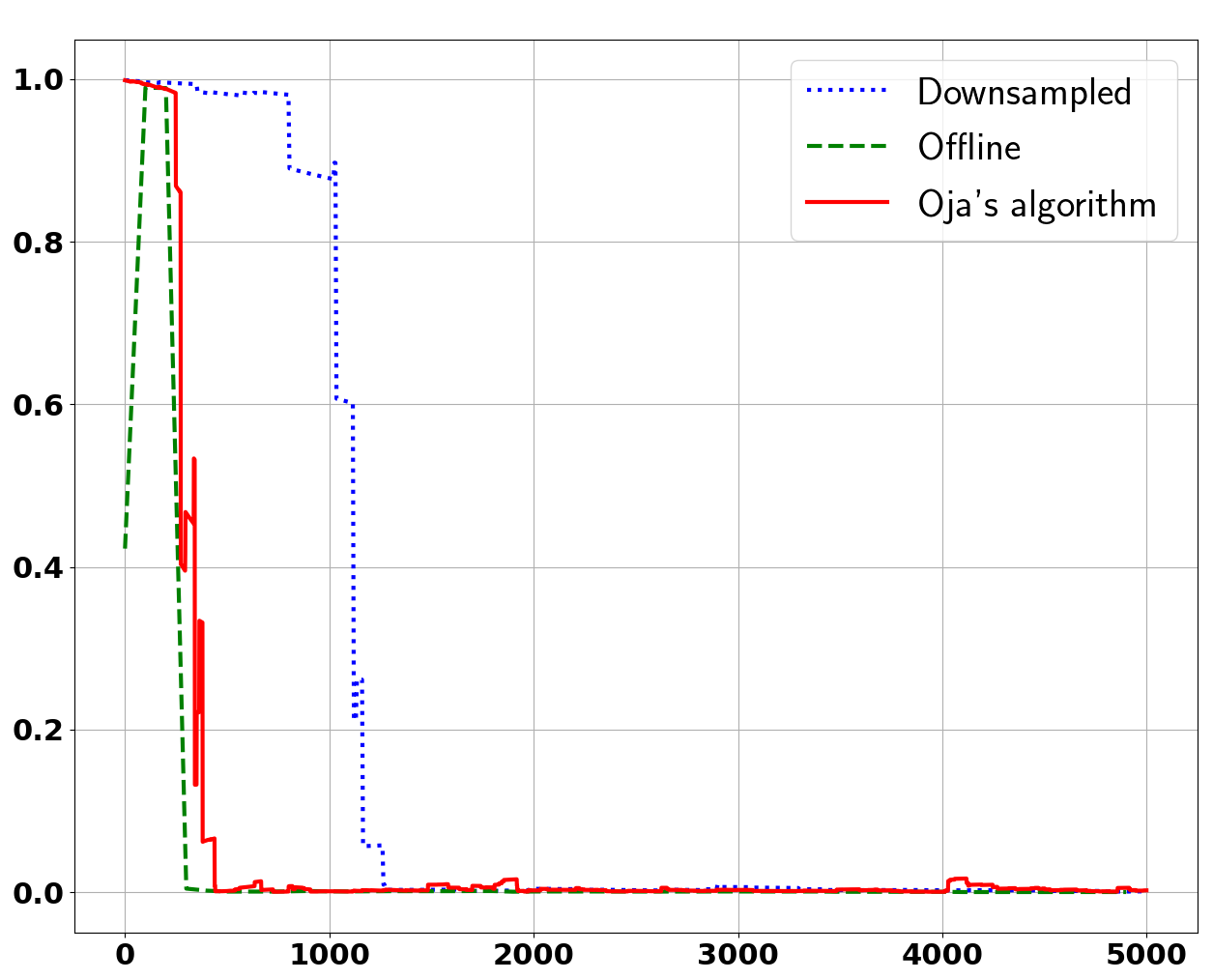}
        \caption{\label{fig:bernoulli_random_run1}}
    \end{subfigure}
    \begin{subfigure}[b]{0.3\textwidth}
        \centering
        \includegraphics[width=\textwidth]{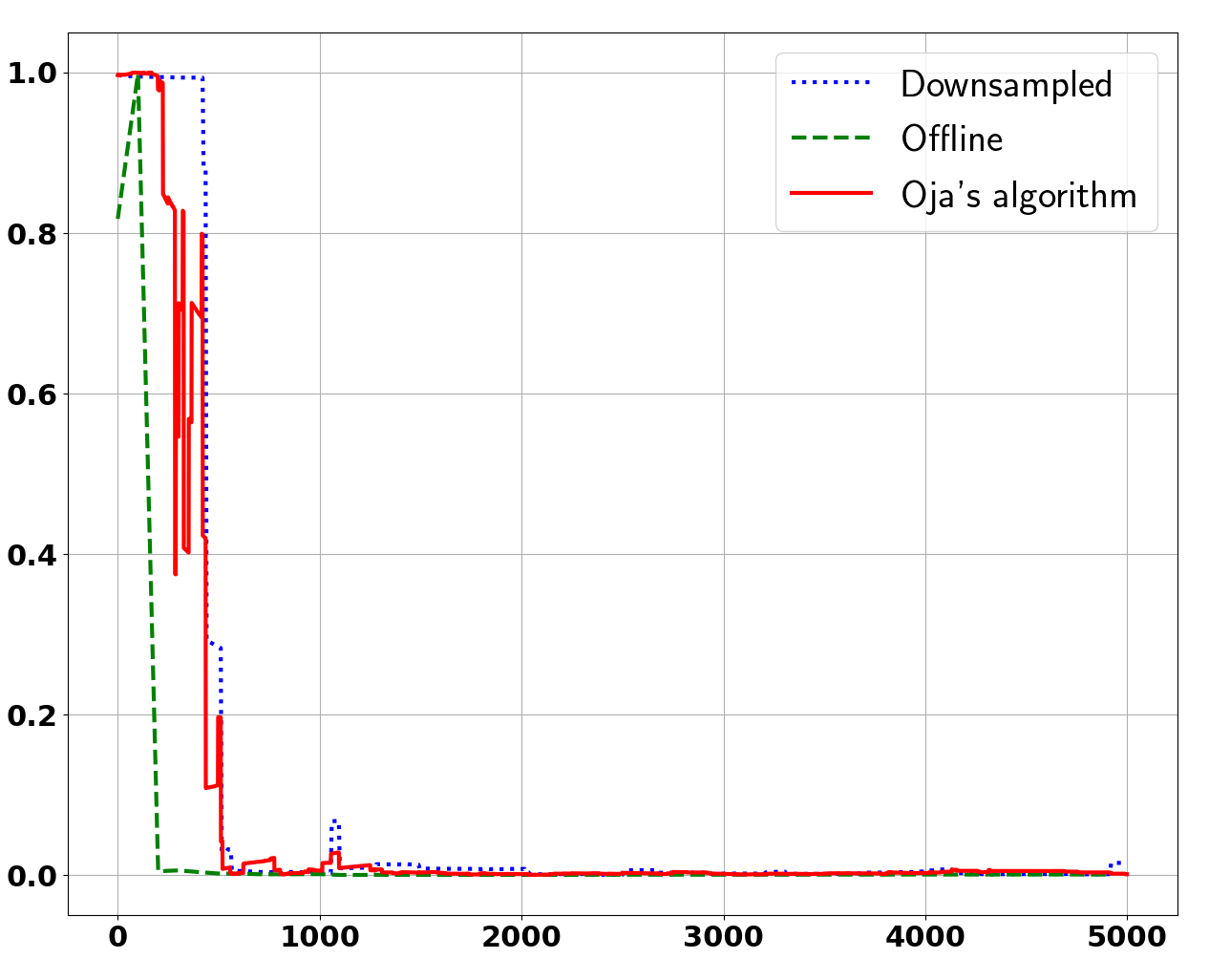}
        \caption{\label{fig:bernoulli_random_run2}}
    \end{subfigure}
    \begin{subfigure}[b]{0.3\textwidth}
        \centering
        \includegraphics[width=\textwidth]{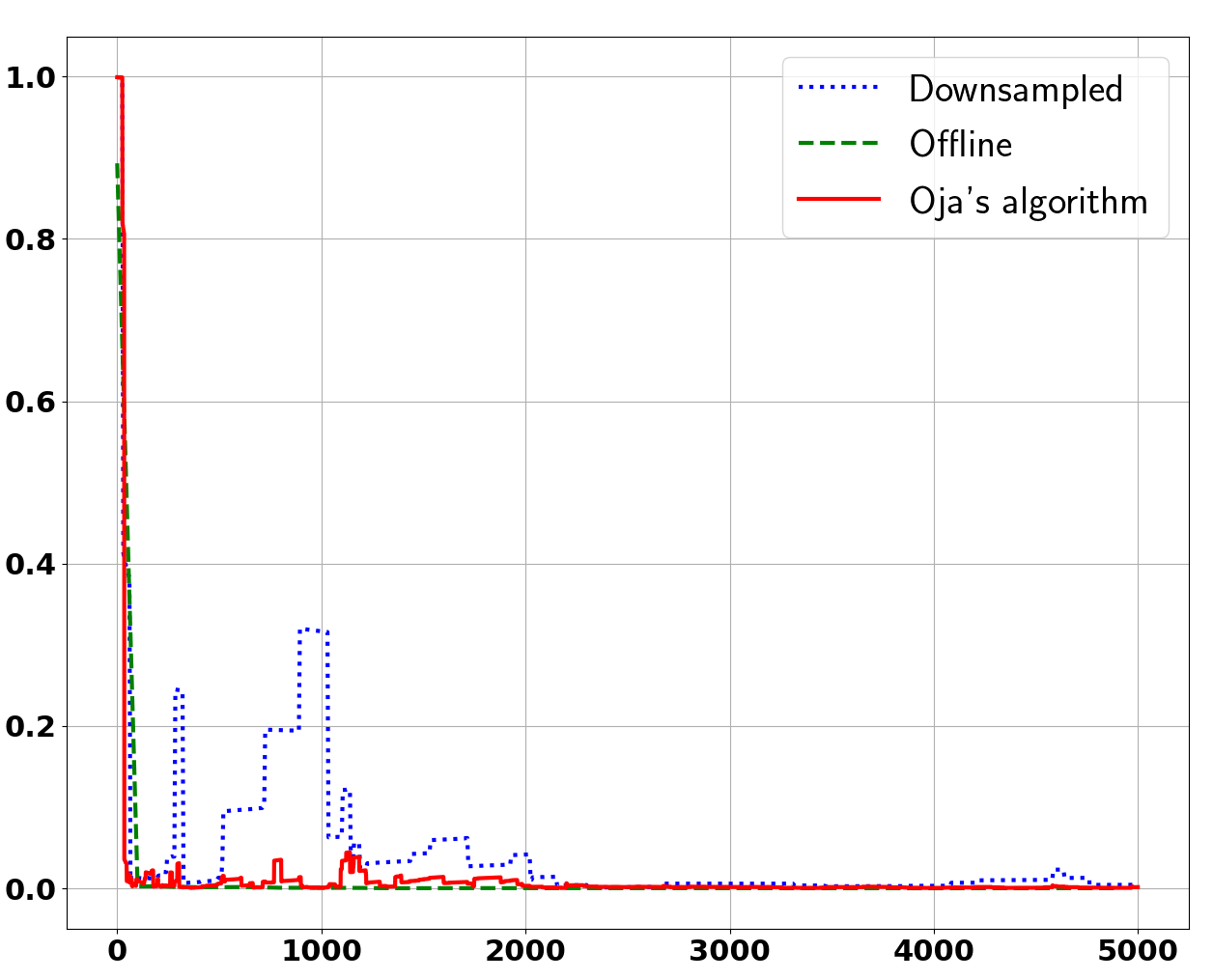}
        \caption{\label{fig:bernoulli_random_run3}}
    \end{subfigure}
    \begin{subfigure}[b]{0.3\textwidth}
        \centering
        \includegraphics[width=\textwidth]{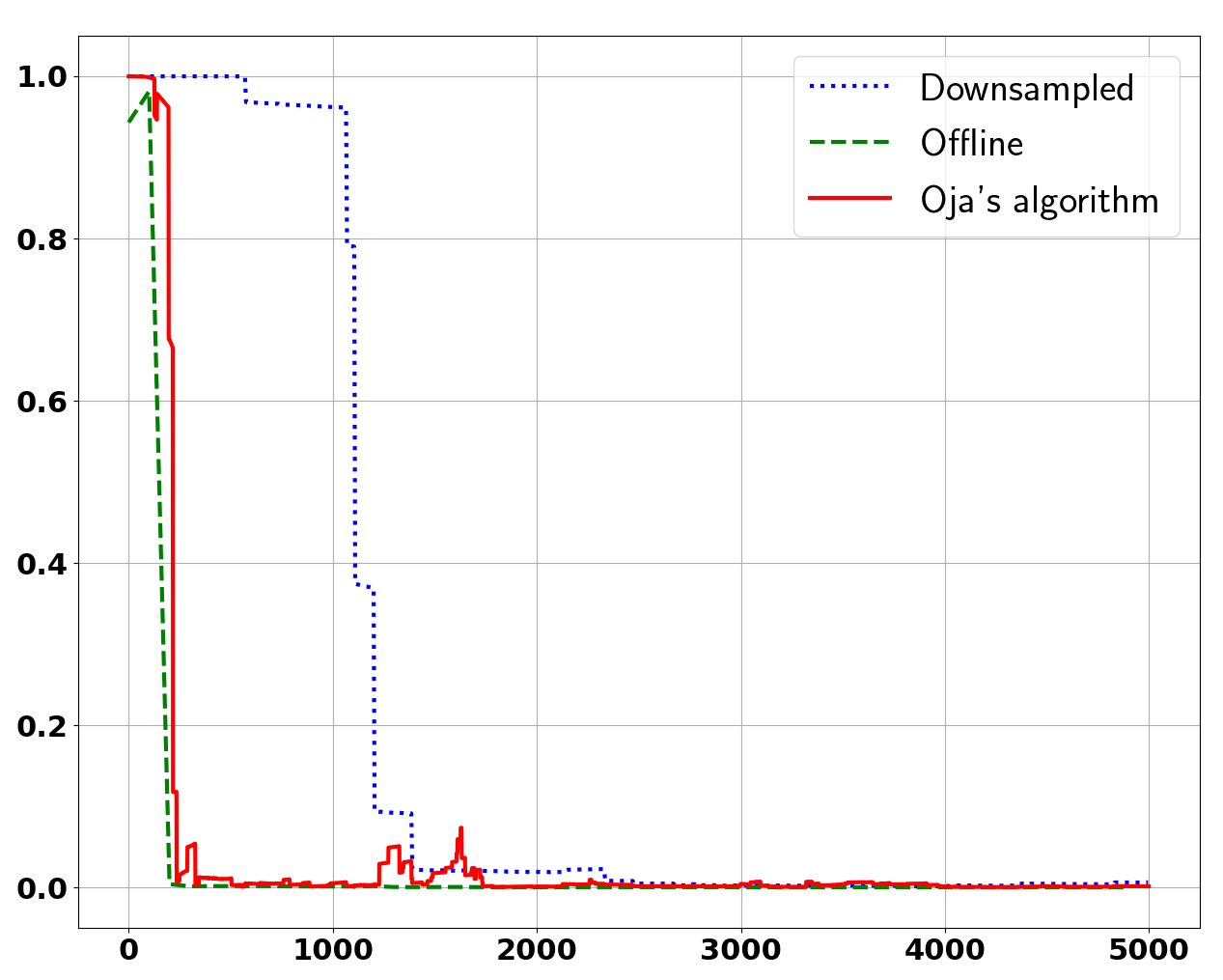}
        \caption{\label{fig:bernoulli_random_run4}}
    \end{subfigure}
    \begin{subfigure}[b]{0.3\textwidth}
        \centering
        \includegraphics[width=\textwidth]{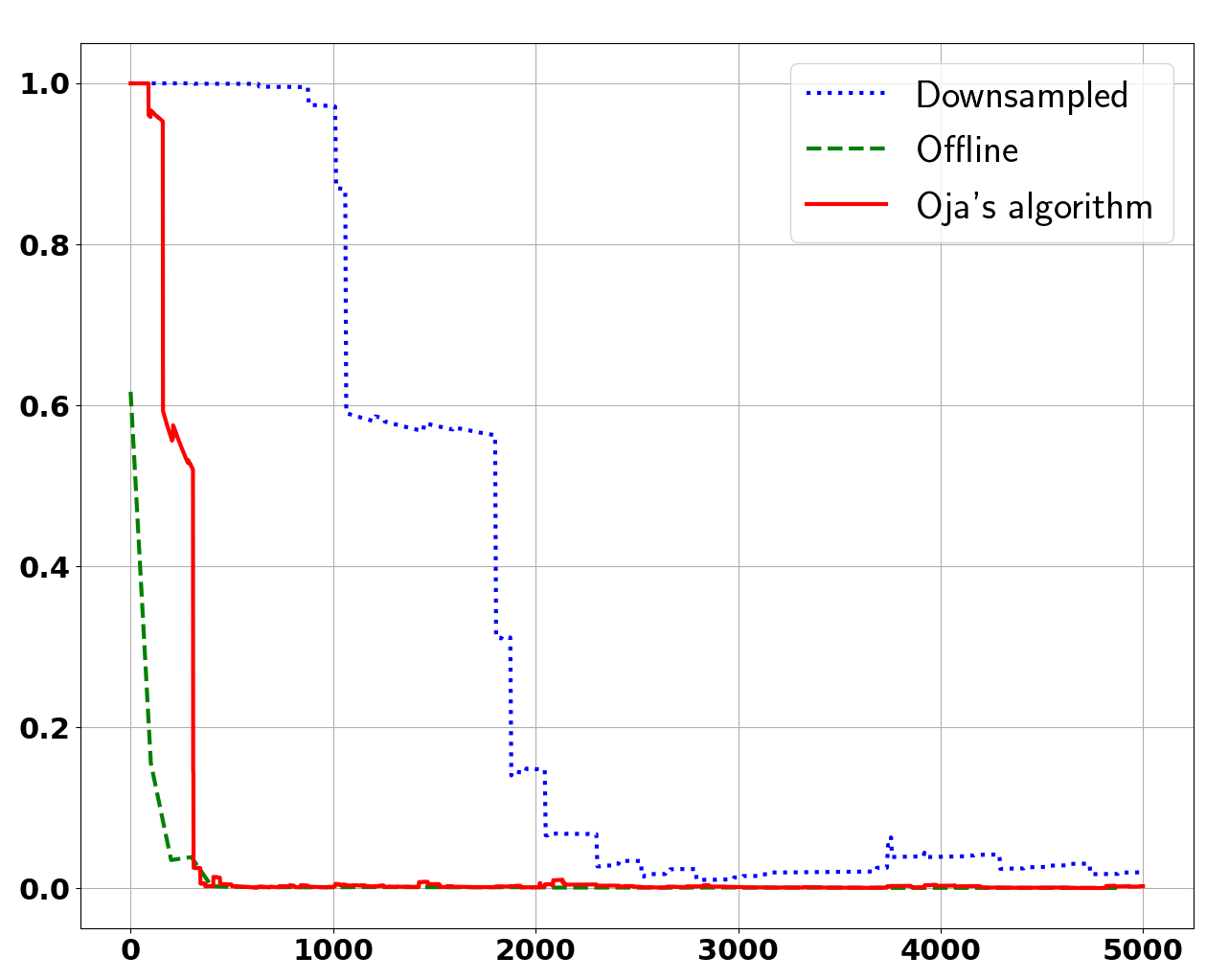}
        \caption{\label{fig:bernoulli_random_run5}}
    \end{subfigure}
    \begin{subfigure}[b]{0.3\textwidth}
        \centering
        \includegraphics[width=\textwidth]{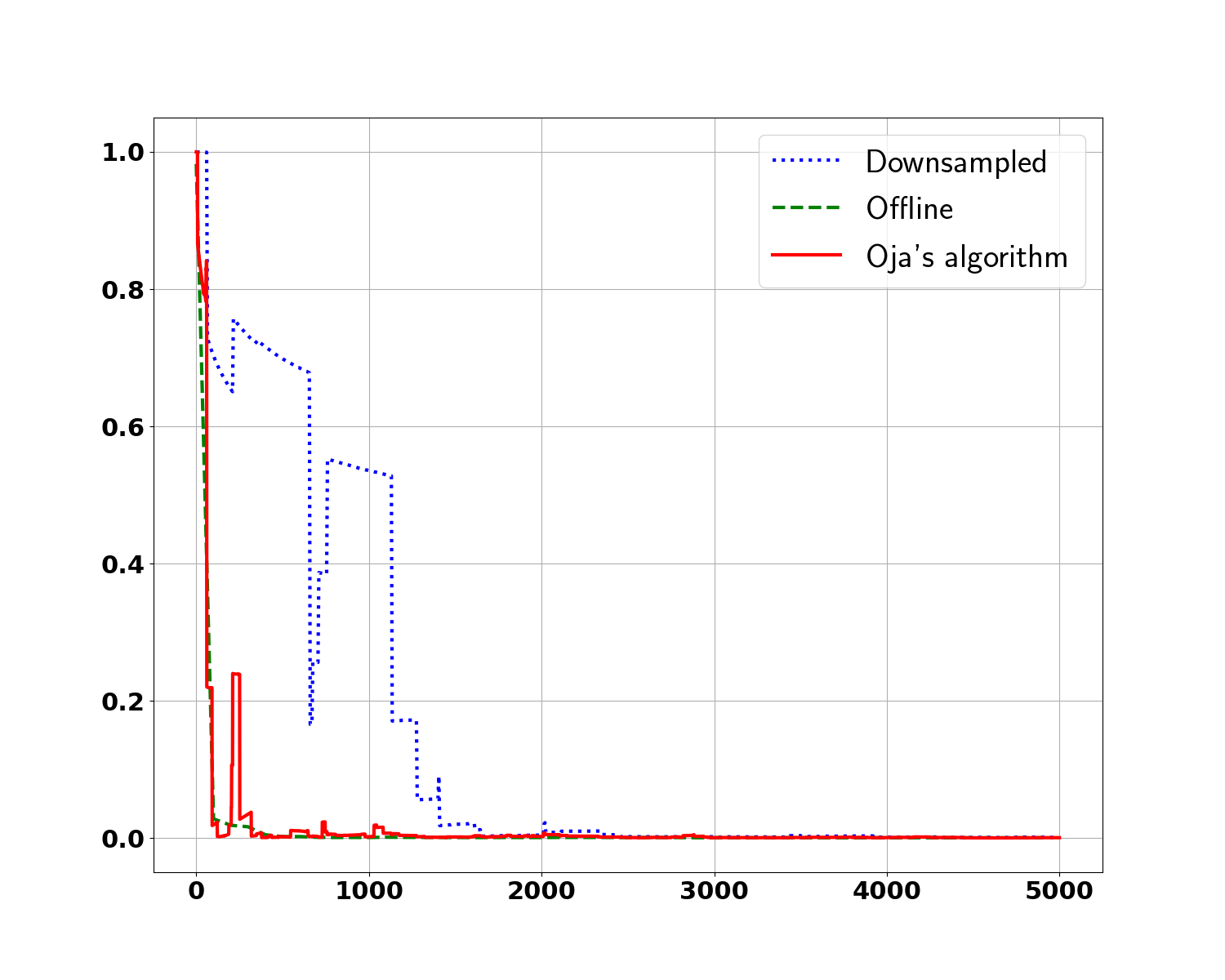}
        \caption{\label{fig:bernoulli_random_run6}}
    \end{subfigure}
    \caption{Randomly chosen runs for the Bernoulli case}
    \label{fig:Bernoulli_random_runs}
\end{figure}

\end{document}